\documentclass[twoside, openany ]{amsart}
\usepackage{amsfonts,amsmath, amssymb}
\usepackage[linktocpage=true, colorlinks=true, linkcolor=RoyalBlue, filecolor=RoyalBlue,urlcolor=RoyalBlue,citecolor=RoyalBlue]{hyperref}
\usepackage{varioref}
\usepackage{graphicx}
\usepackage{aligned-overset}
\usepackage{float}
\usepackage{srcltx}
\usepackage[all]{xy}
\usepackage{bbm}
\usepackage[T2A, T1]{fontenc}
\usepackage[utf8]{inputenc}	
\usepackage{CJKutf8}
\usepackage{adjustbox}
\usepackage[dvipsnames]{xcolor}
\usepackage[normalem]{ulem}
\usepackage{bm}
\usepackage{pifont}
\usepackage{latexsym}
\usepackage{stmaryrd}
\usepackage[2emode]{psfrag}
\usepackage{yhmath}
\usepackage{array}
\usepackage{dsfont}
\usepackage{mathrsfs}
\usepackage{tikz-cd}
\usepackage{comment}
\usepackage{subcaption}
\usetikzlibrary { decorations.pathmorphing, decorations.pathreplacing, decorations.shapes }
\usetikzlibrary{graphs}
\usepackage{maybemath}
\usepackage[capitalize,noabbrev]{cleveref}
\crefname{section}{\S$\!\!$}{\S\S$\!\!$}
\Crefname{section}{Section}{Sections}
\crefname{subsection}{\S$\!\!$}{\S\S$\!\!$}
\Crefname{subsection}{Section}{Sections}
\crefname{subsubsection}{\S$\!\!$}{\S\S$\!\!$}
\Crefname{subsubsection}{Section}{Sections}

\DeclareSymbolFont{cyrillic}{T2A}{cmr}{m}{n}
\DeclareMathSymbol{\Zh}{\mathalpha}{cyrillic}{198}
\DeclarePairedDelimiter{\points}{<}{>}

\DeclareMathOperator*{\bigast}{\raisebox{-7pt}{\scalebox{2}{*}}\kern -2pt}

\DeclareMathOperator{\otb}{\bar{\otimes}}

\definecolor{Maroon}{rgb}{0.6 0 0}
\definecolor{Prussian}{rgb}{0.05 0 0.6}
\definecolor{Emerald}{rgb}{0 0.5 0.1}

\newtheoremstyle{mytheorem}%
{10.0pt plus 2.0pt minus 2.0pt} 
{10.0pt plus 2.0pt minus 2.0pt} 
{\itshape} 
{} 
{\sc} 
{.} 
{ } 
{} 

\newtheoremstyle{mydefinition}%
{10.0pt plus 2.0pt minus 2.0pt} 
{10.0pt plus 2.0pt minus 2.0pt} 
{} 
{} 
{\sc} 
{.} 
{ } 
{} 

\newtheoremstyle{myexample}%
{10.0pt plus 2.0pt minus 2.0pt} 
{10.0pt plus 2.0pt minus 2.0pt} 
{\small} 
{} 
{\sc} 
{.} 
{ } 
{} 

\newtheoremstyle{myremark}%
{10.0pt plus 2.0pt minus 2.0pt} 
{10.0pt plus 2.0pt minus 2.0pt} 
{} 
{} 
{\itshape} 
{.} 
{ } 
{} 

\theoremstyle{mytheorem}
\newtheorem{theorem}{Theorem}[section]
\newtheorem{lemma}[theorem]{Lemma}

\newtheorem{corollary}[theorem]{Corollary}
\newtheorem{proposition}[theorem]{Proposition} 

\theoremstyle{myremark}
\newtheorem{remark}[theorem]{Remark}
\newtheorem{convention}[theorem]{Convention}

\theoremstyle{mydefinition}
\newtheorem{definition}[theorem]{Definition}

\theoremstyle{myexample}
\newtheorem{example}[theorem]{Example}

\newtheoremstyle{myzusatz}
{10.0pt plus 2.0pt minus 2.0pt} 
{10.0pt plus 2.0pt minus 2.0pt} 
{\itshape} 
{} 
{\sc} 
{.} 
{ } 
{\thmname{#1}\thmnumber{ #2}\thmnote{ #3}}

\theoremstyle{myzusatz}

\definecolor{gray1}{gray}{0.8}
\definecolor{gray2}{gray}{0.6}
\definecolor{gray3}{gray}{0.4}
\definecolor{gray4}{gray}{0.2}

\allowdisplaybreaks
\newcommand{\id}{\mathrm{id}}

\newcommand{\Cc}{\mathscr{C}}
\newcommand{\Dd}{\mathscr{D}}

\newcommand{\Nn}{\mathscr{N}}
\newcommand{\Aa}{\mathscr{A}}
\newcommand{\Bb}{\mathscr{B}}
\newcommand{\Ii}{\mathscr{I}}

\newcommand{\Gg}{\mathscr{G}}

\newcommand{\Rr}{\mathscr{R}}

\newcommand{\Hh}{\mathscr{H}}
\newcommand{\Kk}{\mathscr{K}}
\newcommand{\Mu}{\mathrm{M}}

\newcommand{\Nu}{\mathrm{N}}

\newcommand{\ot}{\otimes}

\newcommand{\set}[1]{\lbrace #1\rbrace}

\DeclareMathOperator{\btriangleright}{\protect\raisebox{0.4pt}{{\protect\scalebox{0.75}[0.85]{\protect\ensuremath{\blacktriangleright}}}}}

\def\Set{{\sf Set}}

\def\Cat{{\sf Cat}}

\def\FF{\mathrm{FF}}
\def\End{\mathrm{End}}

\def\Gpd{{\sf Gpd}}
\def\Gpdconn{{\sf Gpd}_{\mathrm{conn}}}
\def\GpdSchur{{\sf Gpd}_{\mathrm{Schur}}}
\def\Gp{{\sf Gp}}
\def\SES{{\sf SES}}

\def\Quiv{{\mathsf{Quiv}}}

\def\source{\mathfrak{s}}
\def\target{\mathfrak{t}}
\def\Source{\mathfrak{S}}
\def\Target{\mathfrak{T}}

\DeclareMathOperator{\Aut}{\mathrm{Aut}}

\newcommand{\Z}{\mathbb{Z}}

\newcommand{\blank}{\raisebox{-2pt}{\text{---}}}

{
	\left\{\begin{aligned}}
	{\end{aligned}
	\right.
}

\newcommand{\One}{\mathbbm{1}}
\DeclareMathOperator{\im}{\mathrm{im}}

\begin{document}
	\hyphenation{qua-si-grou-po-id}
	\hyphenation{qua-si-grou-po-ids}
	\hyphenation{grou-po-id}
	\hyphenation{grou-po-ids}
	\hyphenation{oid-i-fi-ca-tion}
	\hyphenation{a-quo-id}
	\hyphenation{a-quo-ids}
	
	\def\presub#1#2#3%
	{\mathop{}%
		\mathopen{\vphantom{#2}}_{#1}%
		\kern-\scriptspace%
		{#2}_{#3}}
	
	\newcommand{\fibre}[2]{\presub{#1}{\times}{#2}}
	
	\newcommand{\bicross}[2]{\presub{#1}{{\blacktriangleright\kern -3pt \blacktriangleleft}}{#2}}
	
	\newcommand{\lcross}[2]{\presub{#1}{{\blacktriangleright\kern -3pt <}}{#2}}
	
	\newcommand{\rcross}[2]{\presub{#1}{{> \kern -3pt \blacktriangleleft}}{#2}}
	
	\newcommand{\twfibre}[2]{\presub{#1}{\,\protect\scalebox{0.65}[1]{\protect\ensuremath{\bowtie}}\,}{#2}}
	
	\newcommand{\zappaszep}[2]{\presub{#1}{\,\bowtie\,}{#2}}
	
	\newcommand{\lfibre}[2]{\presub{#1}{\rtimes}{#2}}
	
	\newcommand{\rfibre}[2]{\presub{#1}{\ltimes}{#2}}
	
	\newcommand{\twomapsright}[2]{\,\underset{#2}{\overset{#1}{\rightrightarrows}} \, }
	\title[Split Lemma and First Isomorphism Theorem for groupoids]{Split Lemma and First Isomorphism Theorem\\ for groupoids}
	\author{Davide Ferri}
	\begin{abstract} Groupoids are the oidification of groups, and they are largely used in topology and representation theory. We consider here the category $\Gpd$ of all groupoids with all morphisms, and the category $\Gpd_\Lambda$ of groupoids over a fixed set of vertices $\Lambda$, with morphisms fixing $\Lambda$.
		
		In $\Gpd_\Lambda$, a First Isomorphism Theorem is already well known; see Ávila, Marín, and Pinedo (2020). Famously, the First Isomorphism Theorem fails to hold in $\Gpd$. However, we retrieve here a universally lifted version of the First Isomorphism Theorem in $\Gpd$, through the definition of \textit{virtual kernels}. 
		
		Semidirect products of a group by a groupoid are well known. We define crossed products in $\Gpd$, and prove that they are equivalent to split epimorphisms, i.e.\@ that they are the `categorial' notion of semidirect product in $\Gpd$ in the sense of Bourn and Janelidze (1998). We observe that in $\Gpd_\Lambda$ crossed products and semidirect products are essentially equivalent, under mild assumptions, and our Split Lemma in $\Gpd$ collapses to a much simpler Split Lemma in $\Gpd_\Lambda$ that appears in Metere and Montoli (2010) and Ibort and Marmo (2023). 
	\end{abstract}
	\address{%
		\parbox[b]{0.9\linewidth}{University of Turin, Department of Mathematics `G.\@ Peano',\\ via
			Carlo Alberto 10, 10123 Torino, Italy.\\
			Vrije Universiteit Brussel, Department of Mathematics and Data Science,\\ Pleinlaan 2, 1050, Brussels, Belgium.}}
	\email{d.ferri@unito.it, Davide.Ferri@vub.be}
	\keywords{Groupoids, Isomorphism Theorems, Actions, Crossed Products, Semidirect Products, Split Epimorphism, Split Lemma.}
	\subjclass[2020]{Primary 20L05; Secondary 18M05}
	\maketitle
	\tableofcontents 
	
	\section*{Introduction}\label{chapt:intro}
	Groupoids are categories (for us, always small) whose maps are all invertible. In this sense, groups are groupoids with a single object, and groupoids are the `oidification' of groups, in the terminology of \cite{jonsson2017poloids,orchard2020unifying}.
	
	The language of groupoids is extremely useful in topology (\textit{fundamental groupoids} are defined as naturally as fundamental groups \cite{BrownGroupoidsVanKampen,BrownNotJust,ReidemeisterEinfuehrung}), in differential geometry \cite{EhresmannC.StructuralGroupoids,pradines2007ehresmann}, in algebra, and in mathematical physics \cite{GroupoidalPictureQM}. The ultimate goal would be a theory of groupoids that has larger descriptive power than the theory of groups, while at the same time being just as tame.
	
	Two obstacles on this path have been known for decades. Namely:
	\begin{enumerate}
		\item the First Isomorphism Theorem, as stated for groups, fails to hold for groupoids \cite{BrownFibrations,BrownGroupoidsAsCoefficients,GroupsToGroupoidsBrown,BrownBookGroupoids,introGroupoids};
		\item it is hard to find a notion of `semidirect product' of groupoids that satisfies a Split Lemma; i.e., such that every split epimorphism of groupoids yields a semidirect product structure (i.e.\@ a semidirect product in the sense of Bourn and Janelidze \cite{BournJanelidzeProtomodularityDescSemProd}).
	\end{enumerate}
	In this paper, we go as far as we can in solving both the problems. We survey several attempts that have been introduced to deal with these issues in special cases, and we offer a common framework to cover all of them. 
	
	We consider two main categories, namely: the category $\Gpd$ of all small groupoids, on all possible sets of objects (\textit{vertices}), with \textit{morphisms of groupoids} given by all possible functors; and the category $\Gpd_\Lambda$ of groupoids with set of objects $\Lambda$, and \textit{strong morphisms of groupoids}, which are functors that act like the identity on $\Lambda$.
	
	A form of the First Isomorphism Theorem is known when the kernel is a group bundle \cite[Theorem 3]{avila_complete}. We reinterpret this as a First Isomorphism Theorem in $\Gpd_\Lambda$.
	
	As for the First Isomorphism Theorem in $\Gpd$, we prove that every short exact sequence of groupoids can be lifted to a \textit{universal} split short exact sequence on which the First Isomorphism Theorem holds. The word `split' cannot be removed, as we demonstrate with counterexamples. However, the construction is compatible with split sequences, as well as with groupoids coming from equivalence relations \cite[Example 2]{GroupsToGroupoidsBrown}, and in both cases the word `split' can be removed from the universal property. This constructs universal sequences enjoying the First Isomorphism Theorems, and our construction provides an adjunction of functors.
	
	Our universal lifting of short exact sequences suggests us the notion of \textit{virtual kernels}. A First Isomorphism Theorem in $\Gpd$, then, is obtained by (roughly speaking) `replacing kernels with virtual kernels'.
	
	The idea that a classical kernel may not be the right object to use, is not new. In the categories of cocommutative Hopf algebras \cite{HopfCocIsSemiabelian} and of cocommutative Hopf braces \cite{GranSciandra2024}, for instance, semi-abelian structures have been discovered by replacing classical kernels with something else.\footnote{Here, we are not trying to prove that $\Gpd$ is semi-abelian. It is actually known that it is \textit{not}, since it is the category of internal groupoids in $\Set$ which is not semi-abelian \cite{GranGrayInternalGroupoids}.} An alternative notion of kernel for functors was also employed in \cite{RhodesTilsonKerMonMorph}, with the name \textit{derived category}.
	
	As for the Split Lemma in $\Gpd$, we give a notion of \textit{crossed product} that corresponds exactly to split epimorphisms. Our notion of crossed product is naturally a `bilateral' notion, since quotients of groupoids are naturally bilateral. But in the category $\Gpd_\Lambda$, after choosing a distinguished vertex, this bilateral crossed product is canonically isomorphic to a \textit{unilateral} crossed product. Observe that a notion of semidirect product of categories (which is also unilateral) appears in Tilson \cite{TilsonMonKerSemiProd}.
	
	Under the assumption that one of the two groupoids is a group bundle, this unilateral crossed product is in turn isomorphic (although not necessarily canonically) to a `semidirect product' of a group by a groupoid, as presented in Brown \cite[\S 11.4]{BrownBookGroupoids}. The interested reader may wish to know that Brown's semidirect product has been generalised by Metere and Montoli \cite{MetereMontoliSemidirectIntGpd} to internal groupoids in any category satisfying suitable mild assumptions, using a famous categorical notion of semidirect product from Bourn and Janelidze \cite{BournJanelidzeProtomodularityDescSemProd}. This is actually a Split Lemma in $\Gpd_\Lambda$. Recently, the same lemma was considered and independently retrieved by Ibort and Marmo \cite{IbortMarmoGroupoids}, with applications to the groupoidal description of quantum mechanics.
	
	Therefore, we retrieve the Split Lemma in $\Gpd_\Lambda$, as a special case of a Split Lemma in $\Gpd$. 
	
	Incidentally, our work has produced a notion of \textit{balanced tensor product} of groupoids, which we use to define crossed products in \cref{subsec:crossed_Gpd}.
	\smallskip
	
	\noindent\textbf{A notational remark.} Throughout this paper, we use the \textit{Leibniz order} $fg = f\circ g$ for the composition of functions, but we use the \textit{anti-Leibniz} or \textit{diagrammatic order} for the composition of arrows in a groupoid; see \cref{conv:antiLeinbiz}. This notational ambiguity is used in other works \cite{BrownBookGroupoids,ferri2024dynamical}, and is particularly handy in the theory of groupoids. The origin resides in topology, and in the notion of fundamental groupoid, where it is customary to follow the diagrammatic order for the composition of continuous paths.
	\section{Quivers}\label{chapt:quivers}
	\subsection{Quivers, morphism, and strong morphisms} A quiver is a directed multigraph, possibly with loops. More algebraically, we define a quiver $Q$ as the datum of two sets $Q^1$ and $Q^0$, and a pair of maps $Q^1\twomapsright{\source}{\target}Q^0$. We call $Q^1$ the set of \textit{arrows}, $Q^0$ the set of \textit{vertices}, $\source$ and $\target$ the \textit{source} and \textit{target map} respectively (when the quiver $Q$ is not immediately clear from the context, we add subscripts such as $\source_Q, \target_Q$ for clarity). We say that $Q$ is a quiver \textit{over} $Q^0$.
	
	Given a quiver $Q$ and subsets $A,B\subseteq  Q^0$, we denote by $Q(A,B)$ the set of arrows with source in $A$ and target in $B$. Whenever one of the two sets is a singleton $\{\lambda\}$, we simply remove the brackets: for instance, $Q(\lambda,\mu) = Q(\{\lambda\}, \{\mu\})$. We call $Q(\lambda, Q^0)$ the (\textit{outgoing}) \textit{star} of $Q$ at $\lambda\in Q^0$. The set $Q(\lambda,\lambda)$ of \textit{loops} at $\lambda$ will be denoted by $Q_\lambda$.
	
	Although it is often handy to assume that $Q^0 = \im(\source)\cup \im(\target)$ (see \cite[Convention 2.3]{ferri2024dynamical}), we do not need to assume it here.
	
	\begin{definition}\label{def:morphQuiv}
		A \textit{morphism} of quivers $f=(f^1, f^0)\colon Q\to R$ is a pair of maps $f^i\colon Q^i\to R^i$, $i=1,0$, such that $\source_R(f^1(x)) = f^0\source_Q(x)$ and $\target_R(f^1(x))  = f^0\target_Q(x)$ for all $x\in Q^1$.
		
		If $Q$ and $R$ are both quivers over $\Lambda=Q^0=R^0$, we say that a morphism $f$ is \textit{strong} (over $\Lambda$) if $f^0 = \id_\Lambda$. This is also called \textit{a morphism over $\Lambda$} in other places \cite{ferri2024dynamical, matsumotoshimizu}. We shall sometimes be sloppy in using this terminology, and say that $f$ is `strong' if $Q^0\subseteq R^0$ and $f^0$ is the inclusion (see \cite[Definition 9]{avila_complete} and  \cref{rem:strong_hom}).
		
		We denote by $\Quiv$, resp.\@ $\Quiv_\Lambda$, the category of quivers with their morphisms, resp.\@ of quivers over $\Lambda$ with their strong morphisms.
	\end{definition}
	
	The terminologies `weak morphisms' for morphisms, and `morphisms' for morphisms over $\Lambda$, are used in \cite{ferri2024dynamical} for convenience but are highly non-standard.
	
	\subsection{Monomorphism and epimorphisms} We say that a morphism $f$ is \textit{full}, \textit{faithful}, or \textit{fully faithful} if $f^1$ is injective, surjective, or bijective, respectively. We say that $f$ is \textit{injective on the vertices}, \textit{surjective on the vertices}, or \textit{bijective on the vertices}, if $f^0$ is injective, surjective, or bijective, respectively.
	
	The following result is expected, although not entirely trivial. 
	
	\begin{lemma}\label{lem:mono_epi_Quiv}
		A morphism $f$ is a monomorphism, resp.\@ epimorphism in $\Quiv$, if and only if $f^1$ and $f^0$ are both monomorphisms, resp.\@ epimorphisms in $\Set$.
		
		A strong morphism $f$ is a monomorphism, resp.\@ an epimorphism in $\Quiv_\Lambda$, if and only if it is a monomorphism, resp.\@ an epimorphism in $\Quiv$.
	\end{lemma}
	\begin{proof}
		If $f^1$ and $f^0$ are monic, resp.\@ epic in $\Set$, then $f$ is clearly monic, resp.\@ epic in $\Quiv$. We now prove the converse
		.
		
		\textit{Monomorphisms.} Let $f\colon R\to S$ be such that, for all quivers $Q$ and morphisms $\alpha,\beta\colon Q\to R$ satisfying $ f \alpha=  f\beta$, one has $\alpha = \beta$.
		\begin{description}
			\item[Proof that $f^1$ is monic]Let $\alpha^1,\beta^1\colon Q^1\to R^1$ be maps, such that $f^1\alpha^1 = f^1\beta^1$. It suffices to find some set $Q^0$, some maps $\alpha^0, \beta^0\colon Q^0\to R^0$, and some maps $\source_Q,\target_Q\colon Q^1\to Q^0$ such that $\alpha = (\alpha^1,\alpha^0)$ and $\beta = (\beta^1, \beta^0)$ are morphisms in $\Quiv$, and $f^0\alpha^0 = f^0\beta^0$. In this case, we would conclude $\alpha^1 = \beta^1$ from the fact that $f$ is monic in $\Quiv$.
			
			We define $Q^0 $ as the disjoint union of two copies of $Q^1$, the elements of the first copy being denoted $\{ s_x\mid x\in Q^1 \}$ (`formal sources'), and the elements of the second copy being denoted $\{ t_x\mid x\in Q^1 \}$ (`formal targets'). We let $\source_Q(x) = s_x$, $\target_Q(x)= t_x$, and we define
			\begin{align*} &\alpha^0(s_x)=\source_R \alpha^1(x), && \alpha^0(t_x)=\target_R\alpha^1(x),\\ 
				& \beta^0(s_x)=\source_R \beta^1(x), && \beta^0(t_x)=\target_R\beta^1(x). \end{align*}
			One has
			\[ f^0 \alpha^0 (s_x) = f^0\source_R\alpha^1(x)= \source_S f^1\alpha^1(x)=\source_Sf^1\beta^1(x) = f^0\beta^0(s_x),  \]
			and similarly for $t_x$. Thus $Q^0$, $\source_Q$, $\beta_Q$, $\alpha^0$, $\beta^0$ satisfy the desired properties.
			\item[Proof that $f^0$ is monic] Suppose by contradiction that $f$ is monic but $f^0$ is not. Then there exist vertices $\lambda\neq \mu \in R^0$ such that $f^0(\lambda)=f^0(\mu)$. Let $Q^1 = \emptyset$, $Q^0 = \{\bullet\}$, $\source_Q = \target_Q = \emptyset$ be the quiver with one vertex and no arrows, where the symbol $\emptyset$ is used to denote both the empty set and the empty functions $\emptyset\to X$ for any set $X$.\footnote{We need to consider a quiver $Q$ with empty set of arrows, here, because in principle $\lambda$ and $\mu$ might be isolated vertices in $R$. Indeed, we are not assuming that $R^0=\im(\source_R)\cup \im(\target_R)$.} Let $\alpha, \beta\colon Q\to R$ be the morphisms with $\alpha^1 = \beta^1 = \emptyset$, $\alpha^0 (\bullet) = \lambda$, $\beta^0 (\bullet) =\mu$. Clearly, $f\alpha = f\beta$ but $\alpha \neq \beta$, contradiction.
		\end{description}
		
		\textit{Epimorphisms.} Let $f\colon Q\to R$ be such that, for all quivers $S$ and morphisms $\alpha,\beta\colon  R\to S$ satisfying $\alpha f = \beta f$, one has $\alpha = \beta$.
		\begin{description}
			\item[Proof that $f^1$ is epic] Suppose by contradiction that $f$ is epic but $f^1$ is not. Then there exist some set $S^1$ and distinct maps $\alpha^1, \beta^1\colon R^1\to S^1$ such that $\alpha^1 f^1 = \beta^1 f^1$. We put a quiver structure on $S^1$: consider formal sources $\mathrm{FS}=\{s_x\mid x\in S^1\}$ and formal targets $\mathrm{FT}=\{t_x\mid x\in S^1\}$ as before, and let $S^0 = Z\sqcup \big((\mathrm{FS}\cup \mathrm{FT})/\sim\big)$, where $Z = R^0\smallsetminus (\im(\source_R)\cup \im(\target_R))$, and $\sim$ is the equivalence relation generated by
			\begin{align*}&s_{\alpha^1(x)}\sim s_{\alpha^1(y)}\text{ whenever }\source_R(x)=\source_R(y),\\
				&s_{\beta^1(x)}\sim s_{\beta^1(y)}\text{ whenever }\source_R(x)=\source_R(y),\\
				&t_{\alpha^1(x)}\sim t_{\alpha^1(y)}\text{ whenever }\target_R(x)=\target_R(y),\\
				&t_{\beta^1(x)}\sim t_{\beta^1(y)}\text{ whenever }\target_R(x)=\target_R(y),\\
				& t_{\alpha^1(x)} = s_{\alpha^1(y)}\text{ whenever }x\ot y\in (R\ot R)^1,\\
				& t_{\beta^1(x)} = s_{\beta^1(y)}\text{ whenever }x\ot y\in (R\ot R)^1.
			 \end{align*}
			 By definition of the equivalence relation, now the following maps are well-defined:
			\[
			 	\alpha^0(\lambda) = \begin{cases}s_{\alpha^1(x)}&\text{for any }x\in R^1,\, \source(x)=\lambda,\\
			 		t_{\alpha^1(x)}&\text{for any }x\in R^1,\, \target(x)=\lambda,\\
			 		\lambda &\text{if }\lambda\notin \im(\source_R)\cup \im(\target_R);
			 	 \end{cases}
			 \]
			 \[
			 \beta^0(\lambda) = \begin{cases}s_{\beta^1(x)}&\text{for any }x\in R^1,\, \source(x)=\lambda,\\
			 	t_{\beta^1(x)}&\text{for any }x\in R^1,\, \target(x)=\lambda,\\
			 	\lambda &\text{if }\lambda\notin \im(\source_R)\cup \im(\target_R).
			 \end{cases}
			 \]
			 With these maps, $\alpha= (\alpha^1, \alpha^0)$ and $\beta = (\beta^1, \beta^0)$ become morphisms of quivers, and $\alpha f = \beta f$, contradiction.
			 \item[Proof that $f^0$ is epic] Suppose by contradiction that $f$ is epic but $f^0$ is not, hence $X= R^0\smallsetminus\im(f^0)\neq \emptyset$. let $S$ be the quiver with
			 \begin{align*} &S^0 = (R^0\smallsetminus\{\lambda\})\sqcup X\times \{0, 1\},\\
			 	 &S^1= \Big(R^1\smallsetminus (R(X, R^0)\cup R(R^0,X))\Big) \sqcup \Big( (R(X, R^0)\cup R(R^0,X))\times \{0,1\}\Big),\end{align*}
			 \[\source_S(x)=\begin{cases}(\source_R(x), 0)&\text{if }x\in R(X,R^0)\times \{0\},\\
			 (\source_R(x), 1)&\text{if }x\in R(X,R^0)\times \{1\},\\
			 \source_R(x)&\text{otherwise;}\end{cases}\]\[ \target_S(x)=\begin{cases}(\target_R(x), 0)&\text{if }x\in R(R^0,X)\times \{0\},\\
			 	(\target_R(x), 1)&\text{if }x\in R(R^0,X)\times \{1\},\\
			 	\target_R(x)&\text{otherwise.} \end{cases} \]
			  This is the quiver identical to $R$, except that the vertex $\lambda$ and every arrow meeting $\lambda$ are duplicated. Now let $\alpha^0,\beta^0\colon R^0\to S^0$ send $\lambda$ to $\lambda_0,\lambda_1$ respectively, and act like the identity on $R^0\smallsetminus\{\lambda\}$. Similarly, let $\alpha^1,\beta^1\colon R^1\to S^1$ send $x\in R(\lambda, R^0)\cup R(R^0,\lambda)$ to $(x,0)$ and $(x,1)$ respectively, and act like the identity on $R^1\smallsetminus (R(\lambda, R^0)\cup R(R^0,\lambda))$. Clearly $\alpha f = \beta f$, but $\alpha$ and $\beta$ are distinct, contradiction.
		\end{description} 
		
		A strong morphism that is monic, resp.\@ epic as a morphism in $\Quiv$, is obviously monic, resp.\@ epic in $\Quiv_\Lambda$. Conversely, if $f\colon Q\to R $ is monic, resp.\@ epic in $\Quiv_\Lambda$, we notice that $f^1$ restricts to functions $Q(\lambda,\mu)\to R(\lambda,\mu)$ for all $(\lambda,\mu)\in\Lambda\times\Lambda$. From this, it is immediate to observe that $f$ must be faithful, resp.\@ full. Thus both $f^1$ and $f^0 = \id_\Lambda$ are monic, resp.\@ epic in $\Set$. By the first part of the Lemma, then, one has that $f= (f^1, \id_\Lambda)$ is monic, resp.\@ epic in $\Quiv$.
	\end{proof}
	
	In the same vein as \cref{lem:mono_epi_Quiv}, monomorphisms and epimorphisms for other categories of graphs have been characterised by Plessas \cite{PlessasThesis}.
	
	Adapting our terminology from \cite{VanOystaeyenQuivers} (see also \cite{ferri2024dynamical}), we use the term \textit{Schurian} for a quiver $Q$ satisfying $|Q(\lambda,\mu)|\le 1$ for all $\lambda,\mu\in Q^0$ (a category with the same property was dubbed \textit{trivial} by Tilson \cite[\S3]{TilsonCatAsAlg}). Notoriously, a Schurian quiver $Q$ is the same as a relation on the set $Q^0$; a Schurian (small) category $\Cc$ is the same as a reflexive transitive relation on $\Cc^0$; and a Schurian (small) groupoid  $\Gg$ (see \cref{def:groupoid}) is the same as an equivalence relation on $\Gg^0$ \cite[Example 2]{GroupsToGroupoidsBrown}.
	
	\subsection{Equivalence pairs and quotients} Quotients of sets are taken with respect to equivalence relations. Analogously, quotients of quivers are taken with respect to \textit{equivalence pairs}. The name is introduced here, but the concept dates back to \cite[(1.9), (1.10)]{TilsonCatAsAlg}. Some of the ideas are also contained in \cite{GeneralizedCongruences,ChustCoelhoPathAlgebras} and many similar works. 
	
	\begin{definition}A \textit{relation pair} on $Q$ is a pair $(\sim,\approx)$ of relations on $Q^1$ and $Q^0$ respectively, such that $\source$ and $\target$ pass to well-defined maps $Q^1/\!\sim\,\to Q^0/\!\approx$; i.e., such that $x\sim y$ implies $\source(x)\approx \source(y)$ and $\target(x)\approx\target(y)$ for all $x,y\in Q^1$.
		
		An \textit{equivalence pair} on $Q$ is a relation pair $(\sim,\approx)$ such that both $\sim$ and $\approx$ are equivalence relations.\end{definition}
	
	If $(\sim,\approx)$ is an equivalence pair on $Q$, then there is a canonical morphism \[ \pi = (\pi^1, \pi^0)\colon \left(Q^1\twomapsright{\source}{\target}Q^0 \right)\to \left(Q^1/\!\sim\twomapsright{\source}{\target}Q^0/\!\approx\right),\]
	where $\pi^1$ and $\pi^0$ are the projections modulo $\sim$ and $\approx$ respectively. 
	
	\begin{lemma}\label{lem:EquivalenceToEquivalencePair} Let $Q$ be a quiver. 
		\begin{enumerate}
			\item For every equivalence relation $\sim$ on $Q^1$, there exists a minimum equivalence relation $\approx$ on $Q^0$ such that $(\sim,\approx)$ is an equivalence pair. 
			\item For every equivalence relation $\approx$ on $Q^0$, there exists a maximum equivalence relation $\sim$ on $Q^1$ such that $(\sim,\approx)$ is an equivalence pair.
			\item For every equivalence relation $\approx$ on $Q^0$, there exists a minimum equivalence relation $\sim$ on $Q^1$ such that $(\sim,\approx)$ is an equivalence pair: namely, the trivial relation $\{(x,x)\mid x\in Q^1\}$.
			\item If $Q$ is Schurian, for each equivalence relation $\approx$ on $Q^0$ there exists a minimum equivalence relation $\sim$ on $Q^1$ such that $(\sim,\approx)$ is an equivalence pair and the quotient is Schurian.
		\end{enumerate}
	\end{lemma}
	\begin{proof}We construct the relations, and leave to the reader the easy verification that they satisfy the desired properties. \begin{enumerate}
			\item Let $\lambda \equiv \mu$ if and only if there exist $x,y\in Q^1$ with $x\sim y$ such that $\source(x) = \lambda$, $\source(y) = \mu$, or there exist $x,y\in Q^1$ with $x\sim y$ such that $\target(x) = \lambda$, $\target(y) = \mu$. Define $\approx$ as the equivalence relation generated by $\equiv$.
			\item Define $x\sim y$ if and only if $\source(x)\approx \source(y)$ and $\target(x)\approx\target(y)$.
			\item Trivial.
			\item Define $x\sim y$ if and only if $\source(x)\approx\source(y)$ and $\target(x)\approx\target(y)$.\qedhere
		\end{enumerate}
	\end{proof}
	\subsection{Twisted fibre product}\label{subsec:twisted_fibre} The \textit{fibre product} of quivers is a classical object; see e.g.\@ \cite{andruskiewitsch2005quiver}. We give here a slight generalisation, which will be needed later.
	
	If $Q$ and $R$ are two quivers, $\Lambda$ a set, and $q\colon R^1\to \Lambda^{Q^1}$ and $p\colon Q^1\to \Lambda^{R^1}$ are two maps sending $b\in R^1$ to $q_b\colon Q^1\to \Lambda$, respectively $a\in Q^1$ to $p_a\colon R^1\to \Lambda$, then one can define the \textit{twisted fibre product} $Q\twfibre{q}{p} R $. This is a quiver with set of vertices $Q^0\cup R^0$, set of arrows
	\[ (Q\twfibre{q}{p}R)^1 = \{ a\times b \in Q^1\times R^1\mid q_b(a) = p_a(b) \}, \]
	and source and target maps $\source(a\times b) = \source_Q(a)$, $\target(a\times b) = \target_R(b)$. We use the notation $Q\lfibre{q}{p}R$ if $q_b$ does not depend on $b$; $Q\rfibre{q}{p}R$ if $p_a$ does not depend on $a$; and $Q\fibre{q}{p}R$ when both maps $q_b,p_a$ are independent of $b,a$ respectively. In the latter case, $Q\fibre{q}{p}R$ is classically called the \textit{fibre product} of $Q$ and $R$, and with a slight abuse we identify $q,p$ with functions $q\colon Q^1\to \Lambda$, $p\colon R^1\to \Lambda$ respectively. We call $Q\fibre{\id}{\id} R$ the \textit{cartesian product} of the two quivers, and we denote it by $Q\times R$.
	
	\subsection{Monoidal structure on $\maybebm{\Quiv_\Lambda}$}\label{subsec:QuivLambda_monoidal} The category $\Quiv_\Lambda$ is monoidal, with the following monoidal product, described by Matsumoto and Shimizu \cite{matsumotoshimizu}. 
	
	Given $Q$ and $R$ in $\Quiv_\Lambda$, define $Q\ot R = Q\fibre{\target_Q}{\source_R}R$, which is again a quiver over $\Lambda$. A pair $q\times r \in (Q\otimes R)^1$ will be called a pair of \textit{consecutive arrows}, and written as $q\otimes r$. 
	
	The monoidal unit is the quiver $\One_\Lambda$ that has exactly one loop on each vertex. This is not unique, but it is clearly unique up to strong isomorphism over $\Lambda$.
	
	This category is not strict, but we assume to be working in a strictification whenever needed.
	
	\section{On the category of groupoids}\label{chapt:groupoids}
	\subsection{Groupoids, morphisms, and strong morphisms} A groupoid is usually defined as a category whose morphisms are all isomorphisms. Here, we always assume groupoids to be \textit{small}, and hence we can give the following equivalent definition (a weaker form of which dates back to Brandt \cite{BrandtGroupoids}).
	
	\begin{definition}\label{def:groupoid}
		A groupoid $(\Gg,\cdot)$ is a quiver $\Gg$ (where $\Gg^1, \Gg^0\neq \emptyset$) with a binary operation $\cdot$ on $\Gg^1$ that is a morphism $\Gg\ot \Gg\to \Gg$, such that
		\begin{enumerate}
			\item $a(bc)=(ab)c$ for all $a\ot b\ot c$ (\textit{associativity});
			\item for every $\lambda\in \Gg^0$ there exists a loop $1_\lambda$ over $\lambda$, satisfying $a1_\lambda = a$ and $1_\lambda b = b$ for all $a\ot 1_\lambda$, $1_\lambda\ot b$ (\textit{bundle of neutral elements});
			\item for all $a\in \Gg^1$ there exists $a^{-1}\in \Gg^1$ satisfying $aa^{-1} = 1_{\source(a)}$, $a^{-1}a = 1_{\target(a)}$ (\textit{inverses}). 
		\end{enumerate} 
		A \textit{morphism of groupoids} is simply a functor; i.e.\@ a morphism of quivers that intertwines the two binary operations. A \textit{strong morphism of groupoids} (\textit{over $\Lambda$}) is a strong morphism of quivers (over $\Lambda$) that is also a morphism of groupoids. We denote by $\Gpd$, resp.\@ $\Gpd_\Lambda$, the category of groupoids with morphisms, resp.\@ groupoids over $\Lambda$ with strong morphisms over $\Lambda$.
	\end{definition}
	Observe that the local neutral element $1_\lambda$, once it exists, is unique for all $\lambda$; and similarly for the inverses. A morphism of groupoids $f = (f^1, f^0)\colon \Gg\to \Hh$ is forced to satisfy $1_{f^0(\lambda)} = f^1(1_\lambda)$ for all $\lambda\in\Gg^0$.
	
	For a subset $S\subseteq \Gg^0$, we use the notation $1_S$ for the family of loops $\{1_\lambda\}_{\lambda \in S}$; and we write $1_\Gg$ for $1_{\Gg^0}$. Observe that, as a groupoid, the subgroupoid $1_\Gg$ is isomorphic to $\One_{\Gg^0}$. 
	
	For every $\lambda\in\Gg^0$, the set of loops $\Gg_\lambda$ is a group, called the \textit{isotropy group} at $\lambda$.
	\begin{convention}\label{conv:antiLeinbiz}
		As already anticipated, we are using here the \textit{anti-Leibniz order} or \textit{diagrammatic convention} on the binary operation of groupoids, reading the multiplication on consecutive arrows \textit{from left to right}; i.e., the product of
		\[\begin{tikzcd}
			\lambda & \mu & \nu
			\arrow["a"{description}, from=1-1, to=1-2]
			\arrow["b"{description}, from=1-2, to=1-3]
		\end{tikzcd}\]
		is $ab$ (not $ba$) from $\lambda$ to $\nu$. This convention is handy but not completely standard.
	\end{convention}
	\begin{remark}\label{rem:strong_hom}
		Ávila, Marín and Pinedo \cite[Definition 9]{avila_complete} call a morphism of groupoids $f\colon \Gg\to \Hh$ a `strong homomorphism' if, whenever $f^1(a)$ and $f^1(b)$ are consecutive, then also $a$ and $b$ are consecutive. 
		
		Notice that $f$ is a strong homomorphism in the above sense if and only if $f^0$ is injective. Indeed, if $f^0$ is injective then clearly $\target(f^1(a)) = f^0(\target(a))$ equals $\source(f^1(b)) = f^0(\source(b))$ if and only if $\target(a) = \source(b)$. Conversely, if $f^0(\lambda) = f^0(\mu)$ for some distinct vertices $\lambda$ and $\mu$, then $1_\lambda$ and $1_\mu$ are not composable, but $f^1(1_\lambda)$ and $ f^1(1_\mu)$ are; thus $f$ is not strong. 
		
		This justifies our terminology `strong morphism' from \cref{def:morphQuiv}.
	\end{remark}
	
	\subsection{Describing connected groupoids as a product}\label{subsec:Gp_and_CoarseGpD} Two important examples of groupoids, lying somewhat on two opposite extrema, are the \textit{bundles of groups} (see e.g.\@ \cite[\S 1.2]{andruskiewitsch2005quiver} or \cite[Definition 2.4]{ferri2024dynamical}) and the \textit{coarse groupoids} (see e.g.\@ \cite{andruskiewitsch2005quiver,GroupsToGroupoidsBrown}), the latter being also termed \textit{groupoids of pairs} in other works (see e.g.\@ \cite[Example 1.11]{introGroupoids} or \cite{ferri2024dynamical,FerriShibukawa}). 
	
	Given a set $\Lambda$, the coarse groupoid over $\Lambda$ is denoted by $\widehat{\Lambda}$, following \cite{ferri2024dynamical}. The isomorphism class of $\widehat{\Lambda}$ depends only on $\kappa = |\Lambda|$, thus we may write $\widehat{\kappa}$ when we only care about the groupoid $\widehat{\Lambda}$ up to isomorphism. For a number $n$, we usually enumerate the vertices of $\widehat{n}$ as $0,\dots, n-1$. In $\widehat{\Lambda}$, the unique arrow $\lambda\to \mu$ will be denoted by $[\lambda, \mu]$. One has $1_\lambda = [\lambda,\lambda]$, and $[\lambda,\mu][\mu,\nu] = [\lambda,\nu]$. We adopt the same notation for any other Schurian groupoid. 
	
	In some sense, every groupoid can be obtained from these two extremal cases. This is very well known; see e.g.\@ Brown \cite{GroupsToGroupoidsBrown}.
	
	\begin{remark} \label{rem:product_group_coarse_groupoid}
		Given a group $G$ and a set $\Lambda$, one can put on the quiver $G\times \widehat{\Lambda}$ the following groupoid structure:
		\[ (g\times [\lambda,\mu]) \cdot (h\times [\mu,\nu]) = gh\times [\lambda,\nu]. \]
	\end{remark}
	
	\begin{proposition}[{see e.g.\@ Brown \cite{GroupsToGroupoidsBrown}}] \label{prop:product_groupd_coarse_groupoid}
		Every connected groupoid $\Gg$ is isomorphic (non-ca\-no\-ni\-cally) to $\Gg_\lambda\times \widehat{\Gg^0}$, where $\Gg_\lambda$ is the isotropy group of any vertex $\lambda\in\Gg^0$.
	\end{proposition}
	In \cref{prop:product_groupd_coarse_groupoid}, two things need to be chosen in $\Gg$ in order to define the isomorphism: a vertex $\lambda\in\Gg^0$, and a maximal Schurian subgroupoid of $\Gg$, which exists and is wide, and hence is a coarse groupoid because $\Gg$ is connected (see \cite[Remark 5.6 and Lemma 5.7]{ferri2024dynamical}). Now, if we identify with $[\lambda,\mu]$ the unique arrow from $\lambda$ to $\mu$ in the chosen maximal Schurian subgroup of $\Gg^1$, then the isotropy group $\Gg_\mu$ is identified with $\Gg_\lambda$ by sending $g\in\Gg_\lambda$ to the loop $[\mu,\lambda] g [\lambda,\mu]\in\Gg_\mu$. In some sense, then, $\Gg$ resembles a semidirect product via an action by conjugation. We shall make this insight more precise in \cref{rem:prod_GP_times_Set_is_crossed}.
	
	Let $\Gg$ be connected, and $f\colon \Gg\to \Hh$ be a morphism. The image of $f$ is entirely contained in a connected component, thus we also assume $\Hh$ connected, without loss of generality. Choose $\lambda \in\Gg^0$. Then the morphism $f$ induces morphisms $f^\Gp$ and $f^\Set$, in $\Gp$ and $\Set$ respectively, by $f^\Gp = f|_{\Gg_\lambda}^{\Hh_{f^0(\lambda)}}$ and $f^\Set = f^0$.
	\begin{proposition}\label{prop:map_product_group_coarse}	
		Let $\Gg$ and $\Hh$ be connected, with chosen vertices $\lambda\in\Gg^0 = \Lambda$ and $\mu\in\Hh^0 = \Mu$, and chosen maximal Schurian sugroupoids $\Gg',\Hh'$, thus inducing isomorphisms $\Gg \cong G\times \widehat{\Lambda}$ and $\Hh \cong H\times \widehat{\Mu}$, with $G = \Gg_\lambda$ and $H = \Hh_\mu$. 
		
		From every pair of morphisms $(\alpha,\beta)\colon G\times (\Lambda,\lambda)\to H\times (\Mu,\mu)$ in $\Gp\times \Set^*$, there exists a morphism $f\colon \Gg\to \Hh$ in $\Gpd$ satisfying $(f^\Gp,f^\Set)=(\alpha,\beta)$ and $f^0(\lambda) = \mu$.
	\end{proposition} 
	\begin{proof}
		Define $f^0= \beta$ and $f^1(g\times [\lambda,\mu]) = \alpha(g)\times [\beta(\lambda),\beta(\mu)]$, where $[a,b]$ is the unique arrow in $\Gg'$ from $a $ to $b$. The verifications are immediate.
	\end{proof}
	The connected groupoids form a full subcategory $\Gpdconn $ of $\Gpd$. Because the choices of $\lambda\in\Gg^0$ and of a maximal Schurian subgroupoid $\widehat{\Lambda}$ need to be made,  \cref{prop:product_groupd_coarse_groupoid,prop:map_product_group_coarse} do not exactly yield an equivalence between $\Gpdconn$ and $\Gp\times \Set$, although they get very close.
	\subsection{On the geometry of groupoids}\label{par:geometry_groupoids} A \textit{complete quiver of degree $d$} is a quiver that has, for every pair of (non necessarily distinct) vertices, exactly $d$ arrows between them \cite[Definition 2.9]{ferri2024dynamical}.
	
	It is well known that every groupoid decomposes as a disjoint union of \textit{connected components} \cite[\S 2.1]{ferri2024dynamical}, where each component is a complete quiver of some degree, and these degrees need not be all the same \cite[\S 2.2]{ferri2024dynamical}.
	
	\subsection{Subgroupoids} A \textit{subgroupoid} of $\Gg$ is a subquiver of $\Gg$ that becomes a groupoid with the restricted operation. We say that a subgroupoid is \textit{full}, resp.\@ \textit{wide}, if it is a full, resp.\@ a wide subquiver.
	
	\subsection{On the images of groupoid morphisms} For a functor $f\colon \Cc\to \Dd$ between categories, it is known that the image need not be a subcategory of $\Dd$: indeed, even if $f^1(ab) = f^1(a)f^1(b)$ holds for all consecutive arrows $a,b\in \Cc^1$, it may very well happen that $f^1(a)$ and $f^1(b)$ are consecutive in $\Dd$ without $a$ and $b$ being consecutive in $\Cc$; see \cite[Example 3.8]{GeneralizedCongruences}. This is avoided, of course, if $f^0$ is injective.
	
	Morphisms of groupoids enjoy the same property---or suffer from the same issue. The image of a morphism need not be a subgroupoid, and an example of this behaviour is reported in \cref{fig:image_not_subgroupoid,fig:image_not_subgroupoid_conn}.
	\begin{figure}[t]
		\[\begin{tikzcd}
			\lambda & \mu &{}&{}& {f^0(\lambda)} && {f^0(\mu)=f^0(\mu')} \\
			{\lambda'} & {\mu'} &&&& {f^0(\lambda')}
			\arrow["f",thick,from=1-3, to=1-4,shift right=10]
			\arrow[from=1-1, to=1-1, loop, in=150, out=210, distance=5mm]
			\arrow["a"{description}, bend left=10, from=1-1, to=1-2]
			\arrow[bend left=10, from=1-2, to=1-1]
			\arrow[from=1-2, to=1-2, loop, in=330, out=30, distance=5mm]
			\arrow[from=1-5, to=1-5, loop, in=105, out=165, distance=5mm]
			\arrow["{f^1(a)}"{description}, bend left=10, from=1-5, to=1-7]
			\arrow["x"{description}, bend left=10, from=1-5, to=2-6]
			\arrow[bend left=10, from=1-7, to=1-5]
			\arrow[from=1-7, to=1-7, loop, in=60, out=120, distance=5mm]
			\arrow["{f^1(b)}"{description}, bend left=10, from=1-7, to=2-6]
			\arrow[from=2-1, to=2-1, loop, in=150, out=210, distance=5mm]
			\arrow[ bend left=10, from=2-1, to=2-2]
			\arrow["b"{description},bend left=10, from=2-2, to=2-1]
			\arrow[from=2-2, to=2-2, loop, in=330, out=30, distance=5mm]
			\arrow[bend left=10, from=2-6, to=1-5]
			\arrow[bend left=10, from=2-6, to=1-7]
			\arrow[from=2-6, to=2-6, loop, in=240, out=300, distance=5mm]
		\end{tikzcd}\]
		\caption{A morphism of groupoids whose image is not a subgroupoid. Since the underlying quivers are Schurian, the groupoid structures are unambiguous. One has $x = f^1(a)f^1(b)$, but $x$ does not lie in the image of $f^1$.}\label{fig:image_not_subgroupoid}
	\end{figure}
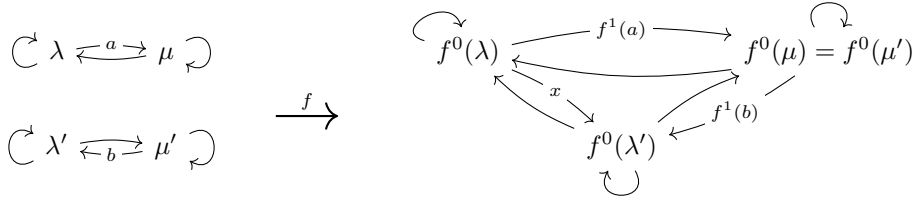
	\begin{figure}[h]
		\[\begin{tikzcd}
			\lambda & \mu & {} & {} & \bullet
			\arrow[ from=1-1, to=1-1, loop, in=150, out=210, distance=5mm]
			\arrow["a"{description}, bend left=10, from=1-1, to=1-2]
			\arrow[ bend left=10, from=1-2, to=1-1]
			\arrow[from=1-2, to=1-2, loop, in=330, out=30, distance=5mm]
			\arrow[thick,"f", from=1-3, to=1-4]
			\arrow["3"{description}, from=1-5, to=1-5, loop, in=45, out=135, distance=22mm]
			\arrow["2"{description}, from=1-5, to=1-5, loop, in=50, out=130, distance=15mm]
			\arrow["1"{description}, from=1-5, to=1-5, loop, in=55, out=125, distance=10mm]
			\arrow["0"{description}, from=1-5, to=1-5, loop, in=60, out=120, distance=5mm]
		\end{tikzcd}\]
		\caption{An example of a morphism $\Gg\to \Hh$ with $\Gg$ connected, such that the image is not a subgroupoid of $\Hh$. Here $\Hh = \Z/4\Z$, and $f^1(a) = 1$, $f^1(a^{-1})= 3$. Observe that $a^2$ is not defined in $\Gg$, while $f^1(a)^2 =2$ is defined in $\Hh$.}
		\label{fig:image_not_subgroupoid_conn}
	\end{figure}
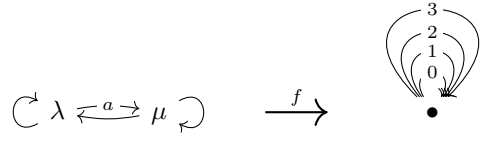
	
	\subsection{Mo\-no\-mor\-phisms and epimorphisms in $\maybebm{\Gpd}$} Mo\-no\-mor\-phisms, resp.\@ epimorphisms in $\Gpd$ are exactly the groupoid morphisms that are monomorphisms, resp.\@ epimorphisms in $\Cat$. 
	
	Let $f\colon \Cc\to \Dd$ be a functor between small categories. Monomorphisms are easier to handle: $f$ is a monomorphism in $\Cat$ if and only if it is faithful and injective on the objects. The same characterisation holds for monomorphisms in $\Gpd$.
	
	Epimorphisms are much more complex. Since the image of $f$ is generally not a subcategory of $\Dd$, we denote by $\points{\im(f)}$ the subcategory of $\Dd$ on $f^0(\Cc^0)$ generated by the subquiver $\im(f) = \left(f^1(\Cc^1)\rightrightarrows f^0(\Cc^0)\right)$. If $\points{\im(f)} = \Dd$, then $f$ is an epimorphism in $\Cat$; but the converse is not true \cite[\S 2]{GeneralizedCongruences}. A characterisation of epimorphisms in $\Cat$ is given by Isbell \cite{IsbellEpimorphisms3}. In $\Gpd$, however, the situation is less complicated.
	
	\begin{proposition}\label{prop:epimorphisms}
		Let $f\colon \Gg\to \Hh$ be a morphism of groupoids. Then $f$ is an epimorphism in $\Gpd$ if and only $\points{\im(f)} = \Hh$.
	\end{proposition}
	\begin{proof}
		It suffices to prove that, if $\points{\im(f)}$ is strictly contained in $\Hh$, then there exist  morphisms $\alpha,\beta\colon \Hh\to \Ii$ such that $\alpha f = \beta f$ but $\alpha\neq\beta$.
		
		If $\Hh$ contains an entire connected component $\Nn$ that is disjoint from $\points{\im(f)}$, then the conclusion follows very easily: choose a vertex $\lambda\in\points{\im(f)}^0$, take $\alpha\colon \Hh\to \Hh$ to be the morphism that sends $\Nn^0 $ to $\lambda$ and $\Nn^1$ to $1_\lambda$, and does nothing on the rest of the quiver; and take $\beta = \id_\Hh$. Clearly, $\alpha f = \beta f$ but $\alpha\neq \beta$.
		
		Therefore, we assume without loss of generality that every connected component of $\Hh$ intersects $\points{\im(f)}$. Observe that two distinct connected components of $\Hh$ cannot intersect the same connected component of $\points{\im(f)}$. Thus, up to breaking down the groupoids into suitable disjoint unions, we may safely assume that $\Hh$ is connected, and $\Gg =\bigsqcup_{i\in I} \Gg_i$ where the $\Gg_i$'s are connected. Let $\Gg_i \cong G_i\times \widehat{\Lambda_i}$ as in \cref{prop:product_groupd_coarse_groupoid}, and fix a vertex $\lambda$ in $\points{\im(f)}^0$, so that $\Hh \cong \Hh_\lambda\times \widehat{\Mu}$.
		
		Suppose that there is a vertex $\mu$ in $\Mu\smallsetminus \points{\im(f)}^0$. Let $\alpha$ be defined by means of $(\alpha^\Gp,\alpha^\Set) = (\id_{\Hh_\lambda}, q)$ where $q\colon \Mu\to \Mu$ is the surjection identifying $\lambda$ and $\mu$, and leaving the other vertices unchanged. Let $\beta = \id_\Hh$. Clearly $\alpha f = \beta=f$ but $\alpha \neq \beta$.
		
		Suppose now that there is an arrow $a$ in $\Hh^1\smallsetminus \points{\im(f)}^1$. Thus there is a loop $g\in \Hh_\lambda \smallsetminus \points{\im(f)}^1$ such that $a$ corresponds to $[\source(a), \lambda]\cdot g \cdot [\lambda,\target(a)]$. This means that the injection $\iota^\Gp\colon \points{\im(f)}_\lambda\to \Hh_\lambda$ is not surjective, and hence it is not an epimorphism in the category of groups: thus one can find groups $K_1, K_2$ and homomorphisms $\alpha^\Gp\colon \Hh_\lambda\to K_1$, $\beta^\Gp\colon \Hh_\lambda\to K_2$, such that $\alpha^\Gp \iota^\Gp = \beta^\Gp \iota^\Gp$, but $\alpha^\Gp\neq \beta^\Gp$. Choose $\alpha^\Set = \beta^\Set=\id_\Lambda$,  define $\alpha\colon \Hh\to K_1\times\widehat{\Lambda}$ and $\alpha\colon \Hh\to K_2\times\widehat{\Lambda}$ by means of $\alpha^\Gp$, $\alpha^\Set$, $\beta^\Gp$, $\beta^\Set$. Again, $\alpha f = \beta f $ but $\alpha\neq \beta$.	
	\end{proof}
	
	\subsection{Normal subgroupoids} In analogy with normal subgroups, a notion of normal subgroupoid has been defined. This encloses the notion of a normal subgroup bundle \cite[\S 1.2]{andruskiewitsch2005quiver}.
	
	\begin{definition}[{see \cite[\S 1]{BrownFibrations}}] \label{def:normal_subgroupoid}
		A subgroupoid $\Nn$ of $\Gg$ is \textit{normal} if it is wide, and $a  \Nn_{\target(a)} a^{-1} \subseteq \Nn_{\source(a)}$ for all $a\in\Gg^1$.
	\end{definition}
	
	For a quiver $Q$, following the notation of \cite{ferri2024dynamical} let $Q^\circlearrowright$ be the wide subquiver of $Q$ whose arrows are exactly the loops of $Q$. This is the object that Andruskiewitsch denotes by $Q^{\text{bundle}}$, see \cite{andruskiewitsch2005quiver}.
	
	\begin{remark}
		A wide subgroupoid $\Nn$ of $\Gg$ is normal in $\Gg$, if and only if $\Nn^\circlearrowright$ is normal in $\Gg$, if and only if $\Nn^\circlearrowright$ is a normal subgroup bundle in the sense of \cite[\S 1.2]{andruskiewitsch2005quiver}. This is obvious from the fact that \cref{def:normal_subgroupoid} amounts exclusively to a condition on the loops of $\Nn$; see \cite[Lemma 3.1]{paques2018galois}.
	\end{remark}
	
	\subsection{Quotients} We now describe the quotients with respect to normal sub\-grou\-poids, following again \cite{BrownFibrations,paques2018galois}.
	\begin{definition}[{cf.\@ \cite[Lemma 3.8]{paques2018galois}}] \label{def:groupoid_quotients}
		Let $\Nn$ be a subgroupoid of $\Gg$. Define the \textit{left quotient} $\Nn\!\bbslash\! \Gg$ as the quotient quiver of $\Gg$ modulo the equivalence pair $(\sim_L,\approx)$ given by
		\begin{align*}& x\sim_L y \iff y^{-1}  x\text{ is defined, and lies in }\Nn; \\
			& \lambda\approx \mu \iff \text{there exists }n\in\Nn\text{ such that }\source(n)=\lambda,\target(n) = \mu.\end{align*}
		
		One may similarly give the definition of the \emph{right quotient} $\Gg\!\sslash\!\Nn$. The \textit{two-sided quotient} $\Gg / \Nn$ is the quotient of $\Gg$ modulo the equivalence pair $(\sim,\approx)$, where $\approx$ is defined as above, and 
		\begin{align*}
			& x\sim y \iff \text{there exist }n,m\in\Nn\text{ such that }nym=x.
		\end{align*}
	\end{definition}
	\begin{remark}If $\Nn$ is a normal subgroupoid of $\Gg$, then the quotient quivers $\Nn\!\bbslash\! \Gg$, $\Gg\!\sslash\!\Nn$, and $\Gg/ \Nn$ inherit a groupoid structure from $\Gg$, as follows. 
		\begin{description}
			\item[Left quotient] If $[g]_L$ and $[h]_L$ are two $\sim_L$-equivalence classes that are consecutive in $\Nn \!\bbslash\!\Gg$, this means that the composition $gn_1h$ is well-defined in $\Gg$, for some $n_1\in \Nn$. We thereby define $[g]_L\cdot [h]_L = [gn_1 h]_L$. The definition is well-posed, because if $n_2\in\Nn$ is another arrow such that $g n_2 h$ is defined, one has $(gn_2 h)^{-1} (gn_1 h) = h^{-1} n_2^{-1}n_1 h \in \Nn^\circlearrowright$, because $n_2^{-1} n_1$ is a loop in $\Nn^\circlearrowright$, and $\Nn$ is normal.
			\item[Right quotient] The groupoid structure is analogous to the one on left quotients.
			\item[Two-sided quotient] If $[g]$ and $[h]$ are two $\sim$-equivalence classes that are consecutive in $\Gg/\Nn$, this means that the composition $n_1 g m_1 n_2 h m_2$ exists in $\Gg$ for arrows $n_1, m_1, n_2, m_2\in\Nn$. Disregarding the superfluous arrows, we thereby define $[g]\cdot [h] = [gm_1n_2 h]$; see \cite{BrownBookGroupoids}. As above, it is easy to check the good definition.
		\end{description}
	\end{remark}
	
	\begin{remark}
		Let $G$ be a group and $N\triangleleft G$ be a normal subgroup. It is well known that, for $x,y\in G$, one has $x\sim y$ if and only if $x\sim_L y$. Indeed, one implication is trivial. As for the other one, $x = nym$ with $n,m\in N$ implies $y^{-1}n^{-1}x = m \in N$, thus $y^{-1} xx^{-1}n^{-1}x  = m \in N$. But $x^{-1}n^{-1} x$ lies in $N$, by normality: thus $y^{-1}x\in N$ as desired. Therefore $N\!\bbslash\! G = G/N = G\!\sslash\! N$.
		
		In the case of groupoids, the implication 
		\[ y^{-1}n^{-1}x = m  \implies y^{-1} xx^{-1}n^{-1}x  = m \]
		would fail in general: indeed $y^{-1}$ and $x$ are composable if and only if $n$ is a loop (see \cref{fig:two-sided_quotient}). Using the same proof as for groups, one can see that $\Nn\!\bbslash\!\Gg = \Gg/\Nn = \Gg\!\sslash\! \Nn$ still holds if $\Nn = \Nn^\circlearrowright$.
		
		\begin{figure}[t]
			\[\begin{tikzcd}
				\bullet & \bullet && \bullet & \bullet
				\arrow["n", from=1-1, to=1-2]
				\arrow[""{name=0, anchor=center, inner sep=0}, "x"', bend left=40, from=1-1, to=1-5]
				\arrow[""{name=1, anchor=center, inner sep=0}, "y", from=1-2, to=1-4]
				\arrow["m", from=1-4, to=1-5]
			\end{tikzcd}\]
			\caption{When $\Nn$ is not a subgroup bundle, the implication $x\sim_Ly \implies x\sim y$ fails in general.}\label{fig:two-sided_quotient}
		\end{figure}
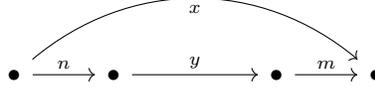
	\end{remark}
	
	\begin{example}
		Consider the coarse groupoid $\Gg = \widehat{6}$, and the normal subgroupoid $\Nn = \widehat{\{0,1,2\}}\sqcup \widehat{\{3,4,5\}}$. In $\Nn\!\bbslash\!\Gg$, two arrows $[i,j]$ and $[i',j']$ are equivalent if and only if $i = i'$ and the vertices $j,j'$ lie in the same connected component of $\Nn$. Thus $\Nn\!\bbslash\!\Gg$ is a complete quiver of degree $3$ on two vertices. On the other hand, $\Gg/\Hh\cong \widehat{2}$. 
	\end{example}

	The natural projection $\pi = (\pi^1,\pi^0)\colon \Gg\to \Gg/\Nn$ is an epimorphism of groupoids with kernel $\Nn$; see \cite[Lemmata 3.8 and 3.12]{paques2018galois} and \cite[\S 1]{BrownFibrations}. Conversely, if $f$ is a morphism of groupoids, then $\ker(f)$ is normal. Thus the normal subgroupoids are precisely the kernels of the morphisms.
	
	Observe that $\ker(f)$ is a bundle of loops if and only if $f^0$ is injective. 
	
	Every morphism of quivers $f\colon Q\to R$ that is injective on the vertices can be read as a strong morphism over $Q^0$, from $Q$ to $\points{\im(f)} = \im(f)$. Thus the normal subgroup bundles are exactly the kernels of the strong morphisms in $\Gpd_\Lambda$, which in turn are exactly the kernels of the morphisms that are injective on the vertices.\footnote{This is already proven in  \cite[Proposition 1.2]{BrownFibrations}. While reading  the work of R.\@ Brown, however, it seems to us that the word \textit{discrete} is used improperly for these kernels. A groupoid is discrete if it is a bundle of units, while in  \cite[Proposition 1.2]{BrownFibrations} and the precedent paragraphs the word `discrete' is used for objects that are simply bundles of groups.}
	\begin{remark}
		Even though the normality of $\Nn$ depends solely on $\Nn^\circlearrowright$, the quotients $\Gg/\Nn$ and $\Gg/\Nn^\circlearrowright$ are usually different: in particular, they usually have different sets of vertices. The same holds for the left and right quotients.
	\end{remark}
	\begin{example}
		Consider the groupoid $\Nn_{4,5,6,7}\cong \widehat{4}$ of \cite[Example 4.27]{ferri2024dynamical}, and the wide subquiver $\Nn$ that only includes the arrows labelled `$0$' or `$2$'. This is a normal subgroupoid, because $\Nn^\circlearrowright = 1_{\Nn_{4,5,6,7}}$ is the loop bundle of the units. It is easy to see that $\Nn_{4,5,6,7}/\Nn$ is the coarse groupoid on two vertices, while $\Nn_{4,5,6,7}/\Nn^\circlearrowright = \Nn_{4,5,6,7}$ has four vertices.
	\end{example}
	
	Famously, the First Isomorphism Theorem does \textit{not} hold in $\Gpd$; meaning that the image of a morphism does actually depend on the map, and not only on its domain and its kernel. Some counterexamples are given in Brown \cite[\S 4]{GroupsToGroupoidsBrown}, and many others are easy to figure out. 
	\begin{example}
		The easiest counterexample to the First Isomorphism Theorem in $\Gpd$ is the epimorphism
		\[\begin{tikzcd}
			{\Gg=} & \lambda & \mu & {} & {} & \bullet & {=\Hh}
			\arrow["{1_\lambda}", from=1-2, to=1-2, loop, in=55, out=125, distance=10mm]
			\arrow["{1_\mu}", from=1-3, to=1-3, loop, in=55, out=125, distance=10mm]
			\arrow[shift left, "\pi", two heads, from=1-4, to=1-5]
			\arrow["{1_\bullet}", from=1-6, to=1-6, loop, in=55, out=125, distance=10mm]
		\end{tikzcd}\]
		sending $\lambda,\mu\mapsto \bullet$. Clearly $\ker(\pi) = \Gg = \One_\Gg$, and $\Gg/\One_\Gg \cong \Gg \ncong \Hh$. Observe moreover that the map $\pi$ is particularly well-behaved, since it also admits a section (the map $\bullet \mapsto \lambda$, $1_\bullet \mapsto 1_\lambda$) which is a monomorphism of groupoids.
	\end{example}
	\begin{example}[{see \cite[\S4]{GroupsToGroupoidsBrown}}]  Consider, just like in \cref{fig:image_not_subgroupoid_conn}, the (non-splitting) epimorphism $f$ from the coarse groupoid $\Gg = \widehat{2}$ to a cyclic group $\Z/n\Z$. Then $\points{\im(f)}$ is isomorphic to $\Gg/\ker(f)$ if and only if $n=1$. 
	\end{example}
	\begin{definition}
		We call \textit{\textsc{fit} sequence} a short exact sequence of groupoids $\Nn\to \Gg\to \Hh$ satisfying the First Isomorphism Theorem, i.e.\@ such that the morphism $\Gg\to\Hh$ induces an isomorphism $\Hh \cong \Gg/\Nn$.
	\end{definition}
	\begin{remark}We recall from \cite[Theorem 1]{avila_complete} that the First Isomorphism Theorem holds true in $\Gpd_\Lambda$, thus every short exact sequence in $\Gpd_\Lambda$ is \textsc{fit}.\end{remark}
	
	\subsection{On the geometry of quotients} For every groupoid $\Gg$, the maximal loop subbundle $\Gg^\circlearrowright$ is clearly normal. The two-sided quotient $\Gg/\Gg^\circlearrowright$ is a Schurian groupoid (actually, already the left quotient $\Gg^\circlearrowright\!\bbslash\!\Gg$ is Schurian): indeed, if $\lambda,\mu\in\Gg^0$ are two vertices, and $x, y\in \Gg(\lambda,\mu)$ are two arrows, then $x^{-1}y$ lies in $\Gg^\circlearrowright$, thus $x$ and $y$ are equivalent.
	
	Let $\Gg$ be a groupoid, and $\Gg'$ be a maximal Schurian subgroupoid, which is a wide coarse subgroupoid. The quotient $\Gg/\Gg'$ has as many vertices as the connected components of $\Gg$. If $\Gg$ is complete of degree $d$, then $\Gg/\Gg'$ is a looped vertex of degree $d$. 
	
	The two above situations are the extremal cases of the following.
	\begin{proposition}\label{prop:degree_of_quotients}
		Let $\Gg$ be a connected groupoid, and hence complete of degree $d$, over a set of vertices $\Lambda$ of cardinality $n$. Let $\Nn$ be a normal subgroupoid. Let $\Gg\cong G\times \widehat{\Lambda}$ for $G = \Gg_\lambda$ the isotropy group at $\lambda \in \Nn^0$. Let $N = \Nn_\lambda$, and suppose that $\Nn$ has $m$ connected components. Then:
		\begin{enumerate}
			\item $\Nn$ is wide;
			\item all the connected components have same degree $|N|$ (but not necessarily same number of vertices);
			\item all the isotropy groups $\Nn_\mu$, for $\mu\in\Gg^0$,  are isomorphic;
			\item $\Gg/\Nn \cong (G/N)\times \widehat{m}$.
		\end{enumerate}
	\end{proposition}
	\begin{proof} Since $\Gg$ is connected, any two vertices $\mu, \mu'$ in $\Gg^0$ are connected by an arrow $g\in \Gg$, and the conjugation by $g$ is an isomorphism between $\Nn_\mu$ and $\Nn_{\mu'}$ because $\Nn$ is normal. Therefore, the isotropy groups $\Nn_\mu$ are all isomorphic, and all are isomorphic to $N$. In particular, no isotropy group $\Nn_\mu$ is empty, thus $\Nn$ is wide.
		
		This also implies that every connected component of $\Nn$ has degree $|N|$. It does not imply, however, that every connected component has the same number of vertices (we shall indeed provide a counterexample in \cref{fig:vker_1}, with the groupoid $\tilde{\Gg}\cong \widehat{4}$ and its normal subgroupoid $\ker(\tilde{f})\cong \widehat{1}\sqcup \widehat{1}\sqcup \widehat{2}$).
		
		By definition of $\approx$, one obviously has $|(\Gg/\Nn)^0|=m$. The projection $\pi\colon \Gg\to \Gg/\Nn$ restricts to a projection $G\to (\Gg/\Nn)_{[\lambda]}$ where $[\lambda]$ is the $\approx$-equivalence class of $\lambda$; thus $(\Gg/\Nn)_{[\lambda]}\cong G/N$. Observe that $\Gg/\Nn$ is connected, because $\Gg$ is; and it has exactly $m$ vertices, because $\Nn$ is wide; hence $\Gg/\Nn \cong (G/N)\times \widehat{m}$ as in \cref{prop:product_groupd_coarse_groupoid}.
	\end{proof}
	\section{A lifted First Isomorphism Theorem}
	\subsection{Virtual kernels: two case studies} Consider the morphism $f\colon \Gg\to \Hh$ described in \cref{fig:case_study_1}, where $f^1$ and the groupoid structures are unambiguous because the quivers are Schurian. Here $f^0$ identifies $\mu$ and $\mu'$, thus we would like to describe the image $\im(f)$ as the quotient of $\Gg$ by the groupoid $\tilde{\Nn}$ on the left-hand side of \cref{fig:case_study_1}. However, the following problems occur:
	\begin{enumerate}
		\item $\tilde{\Nn}$ is not a subgroupoid of $\Gg$, because the arrows $\mu\leftrightarrows\mu'$ (the dashed arrows in the figure) do not belong to $\Gg$;
		\item the image $\im(f)$ is not a groupoid, while $\points{\im(f)}\cong \widehat{3}$ cannot be obtained as a quotient of $\Gg$ in any possible way;
		\item the kernel of $f$ is $\ker(f)=\One_{\Gg^0}$, not $\tilde{\Nn}$. Moreover, $\Gg/\ker(f) \cong \Gg$ is not isomorphic to $\points{\im(f)}$.
	\end{enumerate}
	We wonder what is the `smallest' groupoid $\tilde{\Gg}$ in which both $\Gg$ and $\tilde{\Nn}$ can be immersed. This is clearly the coarse groupoid $\widehat{\{\lambda, \lambda',\mu,\mu'\}}\cong \widehat{4}$. Observe that $f$ induces a unique morphism $\tilde{f}\colon \tilde{\Gg}\to \Hh$ which is now full and surjective on the vertices, whose kernel is $\tilde{\Nn}$, and for which $\Hh\cong \tilde{\Gg}/\tilde{\Nn}$ holds; see \cref{fig:vker_1}.
	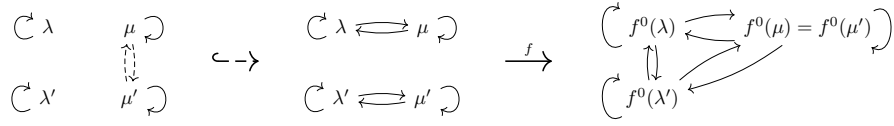
\begin{figure}[t]		
		\[\adjustbox{scale=0.75,center}{\begin{tikzcd}
				\lambda & \mu & {} & {} & \lambda & \mu & {} & {} & {f^0(\lambda)} & {f^0(\mu)=f^0(\mu')} \\
				{\lambda'} & {\mu'} &&& {\lambda'} & {\mu'} &&& {f^0(\lambda')}
				\arrow[from=1-1, to=1-1, loop, in=150, out=210, distance=5mm]
				\arrow[from=1-2, to=1-2, loop, in=330, out=30, distance=5mm]
				\arrow[bend left=10, dashed, from=1-2, to=2-2]
				\arrow[thick,shift right=7, dashed, hook, from=1-3, to=1-4]
				\arrow[from=1-5, to=1-5, loop, in=150, out=210, distance=5mm]
				\arrow[bend left=10, from=1-5, to=1-6]
				\arrow[bend left=10, from=1-6, to=1-5]
				\arrow[from=1-6, to=1-6, loop, in=330, out=30, distance=5mm]
				\arrow[thick,"f", shift right=7, from=1-7, to=1-8]
				\arrow[from=1-9, to=1-9, loop, in=150, out=210, distance=5mm]
				\arrow[bend left=10, from=1-9, to=1-10]
				\arrow[bend left=10, from=1-10, to=1-9]
				\arrow[shift left=7, from=1-10, to=1-10, loop, in=330, out=30, distance=5mm]
				\arrow[bend left=10, from=1-10, to=2-9]
				\arrow[from=2-1, to=2-1, loop, in=150, out=210, distance=5mm]
				\arrow[bend left=10, dashed, from=2-2, to=1-2]
				\arrow[from=2-2, to=2-2, loop, in=330, out=30, distance=5mm]
				\arrow[from=2-5, to=2-5, loop, in=150, out=210, distance=5mm]
				\arrow[bend left=10, from=2-5, to=2-6]
				\arrow[bend left=10, from=2-6, to=2-5]
				\arrow[from=2-6, to=2-6, loop, in=330, out=30, distance=5mm]
				\arrow[bend left=10, from=2-9, to=1-10]
				\arrow[from=2-9, to=2-9, loop, in=150, out=210, distance=5mm]
				\arrow[from=2-9, to=1-9, bend left=10]
				\arrow[from=1-9, to=2-9, bend left=10]
		\end{tikzcd}}\]
		\caption{The morphism $f\colon \Gg\to \Hh$ identifies $\mu$ and $\mu'$, thus we would like to take the groupoid on the left as its kernel; but the dashed arrows do not exist in $\Gg$.}\label{fig:case_study_1}
	\end{figure}
	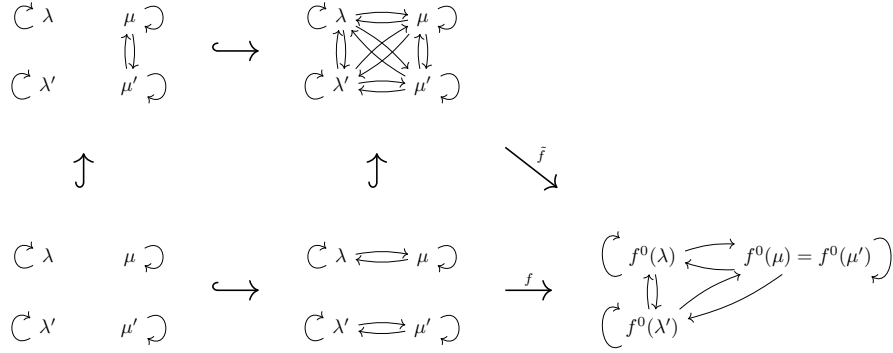
\begin{figure}[t]
		\[\adjustbox{scale=0.75,center}{\begin{tikzcd}
				\lambda & \mu & {} & {} & \lambda & \mu \\
				{\lambda'} & {\mu'} &&& {\lambda'} & {\mu'} \\
				{} &&&& {} && {} \\
				{} &&&& {} &&& {} \\
				\lambda & \mu & {} & {} & \lambda & \mu & {} & {} & {f^0(\lambda)} & {f^0(\mu)= f^0(\mu')} \\
				{\lambda'} & {\mu'} &&& {\lambda'} & {\mu'} &&& {f^0(\lambda')}
				\arrow[from=1-1, to=1-1, loop, in=150, out=210, distance=5mm]
				\arrow[from=1-2, to=1-2, loop, in=330, out=30, distance=5mm]
				\arrow[bend left=10, from=1-2, to=2-2]
				\arrow[thick,shift right=7, hook, from=1-3, to=1-4]
				\arrow[from=1-5, to=1-5, loop, in=150, out=210, distance=5mm]
				\arrow[bend left=10, from=1-5, to=1-6]
				\arrow[bend left=10, from=1-5, to=2-5]
				\arrow[bend left=10, from=1-5, to=2-6]
				\arrow[bend left=10, from=1-6, to=1-5]
				\arrow[from=1-6, to=1-6, loop, in=330, out=30, distance=5mm]
				\arrow[bend left=10, from=1-6, to=2-5]
				\arrow[bend left=10, from=1-6, to=2-6]
				\arrow[from=2-1, to=2-1, loop, in=150, out=210, distance=5mm]
				\arrow[bend left=10, from=2-2, to=1-2]
				\arrow[from=2-2, to=2-2, loop, in=330, out=30, distance=5mm]
				\arrow[bend left=10, from=2-5, to=1-5]
				\arrow[bend left=10, from=2-5, to=1-6]
				\arrow[from=2-5, to=2-5, loop, in=150, out=210, distance=5mm]
				\arrow[bend left=10, from=2-5, to=2-6]
				\arrow[bend left=10, from=2-6, to=1-5]
				\arrow[bend left=10, from=2-6, to=1-6]
				\arrow[bend left=10, from=2-6, to=2-5]
				\arrow[from=2-6, to=2-6, loop, in=330, out=30, distance=5mm]
				\arrow[thick,"{\tilde{f}}", from=3-7, to=4-8]
				\arrow[thick,shift right=7.5, hook, from=4-1, to=3-1]
				\arrow[thick,shift right=7.5, hook, from=4-5, to=3-5]
				\arrow[from=5-1, to=5-1, loop, in=150, out=210, distance=5mm]
				\arrow[from=5-2, to=5-2, loop, in=330, out=30, distance=5mm]
				\arrow[thick,shift right=7, hook, from=5-3, to=5-4]
				\arrow[from=5-5, to=5-5, loop, in=150, out=210, distance=5mm]
				\arrow[bend left=10, from=5-5, to=5-6]
				\arrow[bend left=10, from=5-6, to=5-5]
				\arrow[from=5-6, to=5-6, loop, in=330, out=30, distance=5mm]
				\arrow[thick,"f", shift right=7, from=5-7, to=5-8]
				\arrow[from=5-9, to=5-9, loop, in=150, out=210, distance=5mm]
				\arrow[bend left=10, from=5-9, to=5-10]
				\arrow[bend left=10, from=5-9, to=6-9]
				\arrow[bend left=10, from=5-10, to=5-9]
				\arrow[shift left=7, from=5-10, to=5-10, loop, in=330, out=30, distance=5mm]
				\arrow[bend left=10, from=5-10, to=6-9]
				\arrow[from=6-1, to=6-1, loop, in=150, out=210, distance=5mm]
				\arrow[from=6-2, to=6-2, loop, in=330, out=30, distance=5mm]
				\arrow[from=6-5, to=6-5, loop, in=150, out=210, distance=5mm]
				\arrow[bend left=10, from=6-5, to=6-6]
				\arrow[bend left=10, from=6-6, to=6-5]
				\arrow[from=6-6, to=6-6, loop, in=330, out=30, distance=5mm]
				\arrow[bend left=10, from=6-9, to=5-9]
				\arrow[bend left=10, from=6-9, to=5-10]
				\arrow[from=6-9, to=6-9, loop, in=150, out=210, distance=5mm]
		\end{tikzcd}}\]
		\caption{The map $f$ from \cref{fig:case_study_1}, with $\ker(f)$, the groupoids $\tilde{\Gg}$ and $\tilde{\Nn}$, the induced morphism $\tilde{f}$, and the various inclusions.}\label{fig:vker_1}
	\end{figure}
	
	We now consider another example. Let $f\colon \Gg\to \Hh$ be as in \cref{fig:image_not_subgroupoid_conn}, with $\Gg=\widehat{2}$ and $\Hh = \Z/4\Z$. Clearly $\ker(f)= \Gg^\circlearrowright$, but $\Gg/\Gg^\circlearrowright$ is isomprhic to the trivial group $1$, not to $\Z/4\Z$. The `smallest' groupoid $\tilde{\Gg}$ such that $\Gg$ embeds in $\tilde{\Gg}$ and $\Hh$ is a quotient of $\tilde{\Gg}$, is the groupoid isomorphic to $\Z/4\Z\times \widehat{2}$. Then $\Hh$ is isomorphic to $\tilde{\Gg}/\tilde{\Nn}$, where $\tilde{\Nn}$ is isomorphic to $\Gg$, and $\ker(f)$ again embeds in $\tilde{\Nn}$; see \cref{fig:vker_2}.
	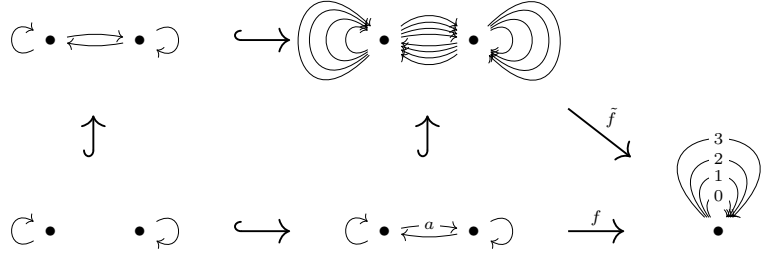
\begin{figure}[t]
		\[\adjustbox{scale=0.9,center}{\begin{tikzcd}
				\bullet & \bullet & {} & {} & \bullet & \bullet & & & \\
				{} & {} & {} & {} & {} & {} & {} & {} &{}\\
				{} & {} & {} & {} & {}& {} & {} & {} &{} \\
				\bullet & \bullet & {} & {} & \bullet & \bullet & {} & {} & \bullet
				\arrow[from=1-1, to=1-1, loop, in=150, out=210, distance=5mm]
				\arrow[bend left=10, from=1-1, to=1-2]
				\arrow[bend left=10, from=1-2, to=1-1]
				\arrow[from=1-2, to=1-2, loop, in=330, out=30, distance=5mm]
				\arrow[thick,hook, from=1-3, to=1-4]
				\arrow[from=1-5, to=1-5, loop, in=150, out=210, distance=5mm]
				\arrow[from=1-5, to=1-5, loop, in=145, out=215, distance=10mm]
				\arrow[from=1-5, to=1-5, loop, in=140, out=220, distance=15mm]
				\arrow[from=1-5, to=1-5, loop, in=135, out=225, distance=20mm]
				\arrow[bend left=10, from=1-5, to=1-6]
				\arrow[bend left=20, from=1-5, to=1-6]
				\arrow[bend left=30, from=1-5, to=1-6]
				\arrow[bend left=40, from=1-5, to=1-6]
				\arrow[bend left=10, from=1-6, to=1-5]
				\arrow[bend left=20, from=1-6, to=1-5]
				\arrow[bend left=30, from=1-6, to=1-5]
				\arrow[bend left=40, from=1-6, to=1-5]
				\arrow[from=1-6, to=1-6, loop, in=330, out=30, distance=5mm]
				\arrow[from=1-6, to=1-6, loop, in=325, out=35, distance=10mm]
				\arrow[from=1-6, to=1-6, loop, in=320, out=40, distance=15mm]
				\arrow[from=1-6, to=1-6, loop, in=315, out=45, distance=20mm]
				\arrow[thick,shift right=7.5, hook, from=3-1, to=2-1]
				\arrow[thick,shift right=7.5, hook, from=3-5, to=2-5]
				\arrow[from=4-1, to=4-1, loop, in=150, out=210, distance=5mm]
				\arrow[from=4-2, to=4-2, loop, in=330, out=30, distance=5mm]
				\arrow[thick,hook,from=4-3, to=4-4]
				\arrow[from=4-5, to=4-5, loop, in=150, out=210, distance=5mm]
				\arrow["a"{description}, bend left=10, from=4-5, to=4-6]
				\arrow[bend left=10, from=4-6, to=4-5]
				\arrow[from=4-6, to=4-6, loop, in=330, out=30, distance=5mm]
				\arrow[thick,"f", from=4-7, to=4-8]
				\arrow["0"{description}, from=4-9, to=4-9, loop, in=60, out=120, distance=5mm]
				\arrow["1"{description}, from=4-9, to=4-9, loop, in=55, out=125, distance=10mm]
				\arrow["2"{description}, from=4-9, to=4-9, loop, in=50, out=130, distance=15mm]
				\arrow["3"{description}, from=4-9, to=4-9, loop, in=45, out=135, distance=22mm]
				\arrow[thick, "\tilde{f}", from=2-7, to=3-8]
		\end{tikzcd}}\]
		\caption{The morphism $f$ from \cref{fig:image_not_subgroupoid_conn}, sending $a$ to $1\in \Z/4\Z$; the groupoids $\tilde{\Nn}$ and $\tilde{\Gg}$, and the induced morphism $\tilde{f}\colon \tilde{\Gg}\to \Hh$. }\label{fig:vker_2}
	\end{figure}
	\subsection{Categories of short exact sequences} If $\Cc$ is a category where the notion of short exact sequence makes sense (e.g.\@ $\Cc= \Gpd$ or $\Cc = \Gpd_\Lambda$), we can consider the classical \textit{category of short exact sequences}, here denoted by $\SES(\Cc)$, having short exact sequences $A\to B\to C$ as objects, and morphisms from $A\to B\to C$ to $A'\to B'\to C'$ given by triples of morphisms $A\to A'$, $B\to B'$, $C\to C'$ in $\Cc$, making the obvious squares commute. 
	
	For an object $X$ in $\Cc$, we let $\SES_X(\Cc)$ be the subcategory of $\SES(\Cc)$ having $X$ as the last term, where a morphism from $A\to B\to X$ to $A'\to B'\to X$ is a morphism in $\SES(\Cc)$ whose third term is the identity $\id_X$. 
	
	We use the notations $\SES^{\text{\textup{split}}}(\Cc)$, $\SES^{\text{\textup{split}}}_X(\Cc)$ for the subcategories of $\SES(\Cc)$, resp.\@ $\SES_X(\Cc)$, consisting of split sequences, and morphisms that intertwine the two splitting maps. We use the notations \[ \SES^{\text{\textup{\textsc{fit}}}}(\Cc),\quad \SES^{\text{\textup{split,\textsc{fit}}}}(\Cc),\quad \SES^{\text{\textup{\textsc{fit}}}}_X(\Cc),\quad \SES^{\text{\textup{split,\textsc{fit}}}}_X(\Cc)\] for the full subcategories of $\SES(\Cc)$, $\SES^{\text{\textup{split}}}(\Cc)$, $\SES_X(\Cc)$, and $\SES^{\text{\textup{split}}}_X(\Cc)$ respectively, consisting of \textsc{fit} sequences.
	\begin{remark}
		If the sequence $A\to B \overset{f}{\to} X$ is exact, then it is isomorphic to $\ker(f)\to B\to X$. Thus the category $\SES_X(\Cc)$ is equivalent to the \textit{arrow category} on $X$, having as objects the morphisms $f\colon B\to X$. The category $\SES_X^{\mathrm{split}}(\Cc)$, in turn, is equivalent to a subcategory of the \textit{category of points}, consisting of split epimorphisms in $\Cc$ having target $X$. 
		
		The reader who is more familiar with this setting, may reinterpret the subsequent sections in the language of arrow and point categories.
	\end{remark}
	
	\subsection{A lifted First Isomorphism Theorem in $\maybebm{\Gpd}$}\label{subsec:liftedFITGpd} Starting from a short exact sequence of groupoids, we shall now construct a split \textsc{fit} exact sequence, which satisfies a universal property very close to the one of a free object. 
	
	Before doing this, we need a preliminary observation in group theory. Recall that the free product of groups (see \cite{SchreierFreieGruppen}) is the coproduct in the category of groups (see e.g.\@ \cite[\S III.3]{categories-for-the-work}).
	
	\begin{remark}\label{lem:univ_groups}
		Let $\{G_i\}_{i\in J}$ be a family of groups, indexed by a set $J$. Let $\tilde{G}$ be the free product of this family, with canonical monomorphisms $\iota_i\colon G_i\to \tilde{G}$. Let $H$ be a group, and $\varphi_i\colon G_i\to H$ be a family of homomorphisms, thus inducing a homomorphism $\varphi\colon \tilde{G}\to H$ by the universal property of the coproduct.
		
		Then for every group $R$ and homomorphisms $\xi_i\colon G_i\to R$, and for every homomorphism $r\colon R\to H$ such that $r \xi_i = \varphi_i$ for all $i$, there exists a unique homomorphism $\xi\colon \tilde{G}\to R$ such that $\xi \iota_i =\xi_i$ for all $i$, and such that $r\xi = \varphi$.
		
		In other words, $\varphi\colon \tilde{G}\to H$ is the coproduct of the maps $\varphi_i\colon G_i\to H$ in the slice category $\Gp/H$, where $\Gp$ is the category of groups. This follows from the proof of \cite[Proposition 3.5.5]{riehl-context}.
	\end{remark}
	
	We now construct, for every short exact sequence $\mathbf{S}$ (or equivalently epimorphism) in $\Gpd$, a universal split \textsc{fit} sequence $\tilde{\mathbf{S}}$ as promised. To do so, we heavily use the decomposition $\Gg = G\times \widehat{\Lambda}$ from \cref{prop:map_product_group_coarse}. This decomposition simplifies the entire discourse, but care needs to be taken in decomposing the different groupoids involved, in a way that is compatible with the morphisms between them: most of the details, in our proofs, will be devoted to this technicality.
	
	Let $f\colon \Gg\to \Hh$ be an epimorphism of groupoids. As usual, without loss of generality, let $\Hh$ be connected, and let $\Gg = \bigsqcup_{i\in I} \Gg_i$ where each $\Gg_i$ is a connected component, and $f$ restricts to morphisms (not necessarily epic) $f_i = f|_{\Gg_i}$.
	
	Chosen a vertex $\mu$ in $\Mu=\Hh^0$ and a maximal coarse subgroupoid $\widehat{\Mu}$ of $\Hh$, one gets $\Hh\cong H\times \widehat{\Mu}$ for a group $H$. Every $\Gg_i$ will be isomorphic to $G_i\times \widehat{\Lambda_i}$ for suitable groups $G_i$ and sets $\Lambda_i = \Gg_i^0$, once a family of vertices $\lambda_i\in\Lambda_i$ has been chosen. However, observe that one may \textit{not} be able to choose $\mu$ and $\{\lambda_i\}_{i\in I}$ such that $f_i^0(\lambda_i)=\mu$ for all $i$, because $f^0$ is surjective but the single $f_i^0$'s need not be; see \cref{fig:no_coherent_family}.
	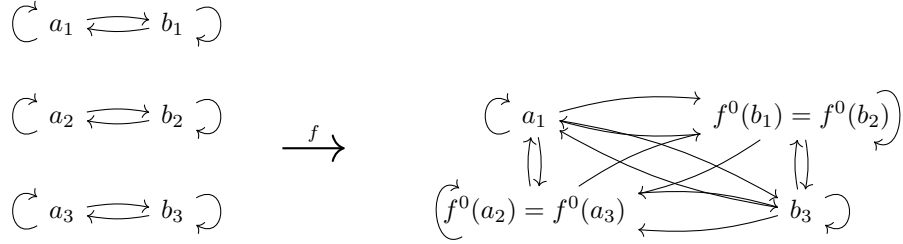
\begin{figure}[t]
		\[\begin{tikzcd}
			{a_1} & {b_1} \\
			{a_2} & {b_2} & {} & {} & {a_1} & {f^0(b_1)=f^0(b_2)} \\
			{a_3} & {b_3} &&& {f^0(a_2)=f^0(a_3)} & {b_3}
			\arrow[from=1-1, to=1-1, loop, in=150, out=210, distance=5mm]
			\arrow[bend left=10, from=1-1, to=1-2]
			\arrow[bend left=10, from=1-2, to=1-1]
			\arrow[from=1-2, to=1-2, loop, in=330, out=30, distance=5mm]
			\arrow[from=2-1, to=2-1, loop, in=150, out=210, distance=5mm]
			\arrow[bend left=10, from=2-1, to=2-2]
			\arrow[bend left=10, from=2-2, to=2-1]
			\arrow[from=2-2, to=2-2, loop, in=330, out=30, distance=5mm]
			\arrow[thick, "f",shift right=5, from=2-3, to=2-4]
			\arrow[from=2-5, to=2-5, loop, in=150, out=210, distance=5mm]
			\arrow[bend left=10, from=2-5, to=2-6]
			\arrow[bend left=10, from=2-5, to=3-5]
			\arrow[bend left=10, from=2-5, to=3-6]
			\arrow[bend left=10, from=2-6, to=2-5]
			\arrow[shift left=5, from=2-6, to=2-6, loop, in=330, out=30, distance=5mm]
			\arrow[bend left=10, from=2-6, to=3-5]
			\arrow[bend left=10, from=2-6, to=3-6]
			\arrow[from=3-1, to=3-1, loop, in=150, out=210, distance=5mm]
			\arrow[bend left=10, from=3-1, to=3-2]
			\arrow[bend left=10, from=3-2, to=3-1]
			\arrow[from=3-2, to=3-2, loop, in=330, out=30, distance=5mm]
			\arrow[bend left=10, from=3-5, to=2-5]
			\arrow[bend left=10, from=3-5, to=2-6]
			\arrow[shift left=5, from=3-5, to=3-5, loop, in=150, out=210, distance=5mm]
			\arrow[bend left=10, from=3-5, to=3-6]
			\arrow[bend left=10, from=3-6, to=2-5]
			\arrow[bend left=10, from=3-6, to=2-6]
			\arrow[bend left=10, from=3-6, to=3-5]
			\arrow[from=3-6, to=3-6, loop, in=330, out=30, distance=5mm]
		\end{tikzcd}\]
		\caption{This morphism $f\colon \Gg\to \Hh$, identifying $b_1$ with $b_2$ and $a_2$ with $a_3$, does not admit a choice of $\mu\in \Hh^0$ that lies in the image of every map $f^0_i$.}\label{fig:no_coherent_family}
	\end{figure}
	
	This means that, in general, the best that one can do is choosing a family of vertices $\{\lambda_i\}_{i\in I}$, and then a family $\mu_i = f_i^0(\lambda_i)$, so that each $f^1_i$ induces a group homomorphism $f_i^\Gp\colon G_i=(\Gg_i)_{\lambda_i}\to \Hh_{\mu_i}$; and then we may compose $f_i^\Gp$ with the isomorphism $C_{[\mu,\mu_i]}\colon  \Hh_{\mu_i}\to \Hh_\mu =H$ given by the conjugation by $[\mu,\mu_i]$ in $\Hh$, thus obtaining homomorphisms $\varphi_i\colon G_i\to H$. These homomorphisms $\varphi_i$ are neither monic nor epic in general.
	
	\begin{remark}\label{rem:disagreeing_Schurian}
		Observe that, once we choose wide Schurian subgroupoids of $\Hh$ and of the $\Gg_i$'s, it is not granted that $f$ will respect our choices; nor that there exists any choice of Schurian subgroupoids that will be respected by $f$ (see the counterexample in \cref{fig:disagreeing_Schurian}). However, this is not a problem for the rest of the construction.
	\end{remark}
	
	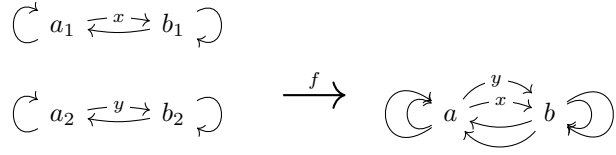
\begin{figure}[t]
		\[\begin{tikzcd}
			{a_1} & {b_1} \\
			{a_2} & {b_2} & {} & {} & a & b
			\arrow[from=1-1, to=1-1, loop, in=150, out=210, distance=5mm]
			\arrow["x"{description}, bend left=10, from=1-1, to=1-2]
			\arrow[bend left=10, from=1-2, to=1-1]
			\arrow[from=1-2, to=1-2, loop, in=330, out=30, distance=5mm]
			\arrow[from=2-1, to=2-1, loop, in=150, out=210, distance=5mm]
			\arrow["y"{description}, bend left=10, from=2-1, to=2-2]
			\arrow[bend left=10, from=2-2, to=2-1]
			\arrow[from=2-2, to=2-2, loop, in=330, out=30, distance=5mm]
			\arrow["f", thick, from=2-3, to=2-4, shift left=3]
			\arrow[from=2-5, to=2-5, loop, in=150, out=210, distance=5mm]
			\arrow[from=2-5, to=2-5, loop, in=145, out=215, distance=10mm]
			\arrow["x"{description}, bend left=20, from=2-5, to=2-6]
			\arrow["y"{description}, bend left = 50, from=2-5, to=2-6]
			\arrow[bend left=20, from=2-6, to=2-5]
			\arrow[bend left = 50, from=2-6, to=2-5]
			\arrow[from=2-6, to=2-6, loop, in=325, out=35, distance=10mm]
			\arrow[from=2-6, to=2-6, loop, in=330, out=30, distance=5mm]
		\end{tikzcd}\]
		\caption{An example of $f\colon \Gg\to \Hh$, with $f^0(a_i)=a$, $f^0(b_i)=b$, in which there is no good choice of the maximal Schurian subgroupoids that will be respected by $f$. Indeed, every $\Gg_i$ is already Schurian, thus the choice is forced for them; and no choice on $\Hh$ agrees. This is no problem for the construction in \cref{thm:construction_Gtilde}.}\label{fig:disagreeing_Schurian}
	\end{figure}
	
	Assume that the set of indices $I$ does not contain $0$, and define $G_0=H$ and $\varphi_0 =\id\colon H\to H$. We let $\tilde{G}$ be the group
	\[ \tilde{G}= \bigast_{i\in I\cup\{0\}} G_i,  \]
	where $*$ denotes the free product of groups.
	
	We denote the canonical injections $G_i\to \tilde{G}$ by $\iota_i^\Gp$. 
	Let $\Lambda =\bigsqcup_{i\in I}\Lambda_i$. We define $\tilde{\Gg} = \tilde{G}\times \widehat{\Lambda}$, and define $\tilde{f}\colon \tilde{\Gg}\to \Hh$ as the map induced by $(\tilde{f}^\Gp, \tilde{f}^\Set)$, where $\tilde{f}^\Set = f^0$, and $\tilde{f}^\Gp$ is induced by the family $\{\varphi_i\}_{i\in I\cup\{0\}}$ via the universal property of $\tilde{G}$.
	\begin{lemma}\label{lem:morphi_SES}Let $f\colon \Gg\to \Hh$ and $\tilde{f}\colon \tilde{\Gg}\to \Hh$ be as above. Then there are morphisms $\Gg\to \tilde{\Gg}$ and $\Nn\to \ker(\tilde{f})$, such that they induce a morphism in $\SES_\Hh(\Gpd)$
		\[\begin{tikzcd}
			{\ker(\tilde{f})} & {\tilde{\Gg}} & \Hh \\
			\Nn & \Gg & \Hh
			\arrow[hook, from=1-1, to=1-2]
			\arrow["{\tilde{f}}", from=1-2, to=1-3]
			\arrow[from=2-1, to=1-1]
			\arrow[hook, from=2-1, to=2-2]
			\arrow[from=2-2, to=1-2]
			\arrow["f", from=2-2, to=2-3]
			\arrow[equals, from=2-3, to=1-3]
		\end{tikzcd}\]
	\end{lemma}
	\begin{proof}
		Consider the morphisms $\iota_i\colon \Gg_i\to \tilde{\Gg}$ given by $\iota_i^\Gp$ defined as above, and by $\iota_i^\Set\colon \Lambda_i\to \Lambda$ defined as the obvious inclusions. Then the family $\{\iota_i\}_{i\in I}$ assembles to a monomorphism $\iota\colon \Gg\to \tilde{\Gg}$. Observe that this morphism induces a monomorphism $\ker(f)\to \ker(\tilde{f})$, by restriction.
	\end{proof}
	\begin{theorem}\label{thm:construction_Gtilde}
		Let $f\colon \Gg\to \Hh$ and $\tilde{f}\colon \tilde{\Gg}\to \Hh$ be as above, inducing a morphism of short exact sequences as in \cref{lem:morphi_SES}. Then:
		\begin{enumerate}
			\item the morphism $\tilde{f}\colon \tilde{\Gg}\to \Hh$ is a split epimorphism, for some splitting $\tilde{s}\colon \Hh\to \tilde{\Gg}$;
			\item the First Isomorphism Theorem holds for $\tilde{f}$;
			\item the operation of sending a short exact sequence \[\mathbf{S} =\Big( \begin{tikzcd}
				{\Nn} & {\Gg} & \Hh
				\arrow[from=1-1, to=1-2]
				\arrow["{f}"',  from=1-2, to=1-3]
			\end{tikzcd} \Big)\] into the sequence \[\tilde{\mathbf{S}} =\Big( \begin{tikzcd}
				{\ker(\tilde{f})} & {\tilde{\Gg}} & \Hh
				\arrow[from=1-1, to=1-2]
				\arrow["{\tilde{f}}"', shift right, from=1-2, to=1-3]
				\arrow["{\tilde{s}}"', shift right, from=1-3, to=1-2]
			\end{tikzcd} \Big)\] yields a functor $\SES_\Hh(\Gpd)\to \SES_\Hh^{\mathrm{split},\text{\textup{\textsc{fit}}}}(\Gpd)$. Moreover, $\tilde{s} f$ is the canonical inclusion $\Gg\to\tilde{\Gg}$.
		\end{enumerate}
	\end{theorem}
	\begin{proof}
		\begin{enumerate}
			\item We define a section $\tilde{s}\colon \Hh\to \tilde{\Gg}$. Since $\points{\im(f)} = \Hh$, clearly $f^\Set=\tilde{f}^\Set$ is surjective: let $\tilde{s}^\Set$ be any set-theoretic section of $\tilde{f}^\Set$. Take the canonical injection $\iota_0^\Gp\colon H\to \tilde{G}$ as $\tilde{s}^\Gp$. This is obviously a section, since $\tilde{f}^\Gp s^\Gp(h) = \varphi_0 (h) = h$ for all $h\in H$.
			\item The morphism $\tilde{f}$ induces an isomorphism $ \tilde{\Gg}/\ker(\tilde{f})\cong (\tilde{G}/\ker(\tilde{f}^\Gp))\times \widehat{\Mu}$ (but observe that, here, the isomorphism between the maximal Schurian subgroupoids need not be induced by the restriction of $\tilde{f}$, for the reason expressed in \cref{rem:disagreeing_Schurian}). Now $\tilde{f}^\Gp$ is surjective, thus it induces an isomorphism $\tilde{G}/\ker(\tilde{f}^\Gp)\cong H$ by the First Isomorphism Theorem for groups, and this concludes.
			\item Let \[\mathbf{S}' =\Big( \begin{tikzcd}
				{\Nn'} & {\Gg'} & \Hh
				\arrow[from=1-1, to=1-2]
				\arrow["{f'}"',  from=1-2, to=1-3]
			\end{tikzcd} \Big)\]
			be another short exact sequence in $\Gpd$, and let $(\eta,\xi,\id_\Hh)$ be a morphism $\mathbf{S}\to \mathbf{S}'$ in $\SES_\Hh(\Gpd)$. All the constructions for $\mathbf{S}'$ are denoted with the same letter as for $\mathbf{S}$, with an apex.
			
			We induce a morphism $(\tilde{\eta},\tilde{\xi},\id_\Hh)\colon \tilde{\mathbf{S}}\to \tilde{\mathbf{S}}'$ as follows. It suffices to define $\tilde{\xi}$ so that $\tilde{f}'\tilde{\xi}= \tilde{f}$, since $\tilde{\eta}$ will be defined as the restriction of $\tilde{\xi}$ to $\ker(\tilde{f})$, and the image will automatically be contained in $\ker(\tilde{f}')$.  
			
			In order to define $\tilde{\xi}$, we decompose the connected components $\Gg_i$ as $G_i\times \Lambda_i$, with chosen vertices $\lambda_i\in\Lambda_i$ yielding groups $G_i = (\Gg_i)_{\lambda_i}$. We do the same for $\Gg'$, choosing vertices $\lambda_i'\in\Lambda_i'$ and groups $G_i'$. The image of $\xi$ restricted to $\Gg_i$ is entirely contained in a connected component $\Gg_k'$ of $\Gg'$. 
			
			We choose vertices $\lambda_{k,i}' = \xi^0(\lambda_i)$ in $\Gg'_k$, so that the commutativity of the diagram implies  $(f')^0(\lambda_{k,i}') = f^0 (\lambda_i)$; and we choose a vertex $\lambda_k'$ among the family $\{\lambda_{k,i}'\}_i$. Let $G'_{k,i} = (\Gg_{k,i}')_{\lambda_{k,i}'}$ and $G'_k = (\Gg_k')_{\lambda_k'}$. Up to composing with isomorphisms $ (\Gg_k')_{}\to (\Gg_k')_{\lambda_k'}$ given by the conjugation by the chosen arrow $[\lambda_k', \xi^0(\lambda_i)] $ in $\Gg'$, we obtain group homomorphisms $\xi_i^\Gp\colon G_i\to G_k'$ for all $G_i$ whose image is contained in $G_k'$. Moreover, by definition of the homomorphisms $\xi_i^\Gp$, the commutativity of the triangle defined by the maps $f,f', \xi$ implies, at the level of groups, that $\varphi_k'\xi_i^\Gp = \varphi_i$. Let $\tilde{\xi}^\Set = \xi^\Set$. Finally, let $\xi_0^\Gp = \id_H$.
			
			Composing $\xi_i^\Gp$ with the canonical morphism $G_k'\to \tilde{G}'$, we get homomorphisms $G_i\to \tilde{G}'$. By the universal property of the coproduct, these induce a homomorphism $\tilde{\xi}^\Gp\colon \tilde{G}\to\tilde{G}'$. We define $\tilde{\xi}^\Set= \xi^\Set$, and we now need to verify that $\tilde{f}'\tilde{\xi} = \tilde{f}$ and that $\tilde{\xi}\tilde{s} = \tilde{s}'$. 
			
			At the level of set-theoretic maps, $(\tilde{f}')^\Set\tilde{\xi}^\Set = \tilde{f}^\Set$ is trivially true because $\tilde{f}^\Set = f^\Set$, $(\tilde{f}')^\Set = (f')^\Set$, and $\tilde{\xi}^\Set = \xi^\Set$. Recall that $(\tilde{f})^\Gp$ and $(\tilde{f}')^\Gp$ are induced, respectively, by the families $\{\varphi_i\colon G_i\to H\}_{i\in I}$ and $\{\varphi_i'\colon G_i'\to H \}_{i\in I'}$ through the universal properties of $\tilde{G}$ and $\tilde{G}'$. By construction of $\tilde{\xi}$, and from the fact that $\varphi_k'\xi_i^\Gp = \varphi_i$ whenever the image of $\xi_i$ is contained in $\Gg_k'$, one easily has that $(\tilde{f}')^\Gp\tilde{\xi}^\Gp$ and $\tilde{f}^\Gp$ are both maps $\tilde{G}\to H$ that commute with the family of homomorphisms $\{\varphi_i\colon G_i\to H\}_{i\in I}$; and hence they are the same map, by the universal property of $\tilde{G}$: namely, $(\tilde{f}')^\Gp\tilde{\xi}^\Gp = \tilde{f}^\Gp$ holds.
			
			As for the relation $\tilde{\xi}\tilde{s} = \tilde{s}'$, it simply follows from $\tilde{s}$ and $\tilde{s}'$ being the canonical inclusions of $H$ into $\tilde{G}$ and $\tilde{G}'$ respectively.
			\qedhere
		\end{enumerate}
	\end{proof}
	\begin{remark}
		If we do not include $H = G_0$ and $\id_H = \varphi_0$ in the family on which we take the free product, then a group homomorphism \[\tilde{g}^\Gp\colon \left(\bigast_{i\in I}G_i \right)\;\to \; H\]
		can still be defined, and it is again surjective. Indeed, since $\Hh = \points{\im(f)}$, one has that $H$ is generated by $\bigcup_{i\in I} \varphi_i(G_i) = \bigcup_{i\in I}\tilde{g}^\Gp\iota_i^\Gp(G_i)$.  However, this is not a split epimorphism in general (as a counterexample, consider for instance $I = \{1\}$ and $\varphi_1\colon G_1\to H$ a group epimorphism that is not split). Moreover, the upper row is not necessarily a \textsc{fit} sequence (as a counterexample, consider the morphism in \cref{fig:image_not_subgroupoid_conn}). 
	\end{remark}
	The following theorem states that the canonical morphism $\mathbf{S}\to \tilde{\mathbf{S}}$ is universal among the morphisms in $\SES_\Hh(\Gpd)$ whose image lies in $\SES^{\text{split,\textsc{fit}}}_\Hh(\Gpd)$.
	\begin{theorem}\label{thm:universal_prop_Gtilde}
		Let $f,\Gg,\tilde{\Gg}$ be as above. The groupoid $\tilde{\Gg}$ satisfies the following universal property. For every split \textsc{fit} sequence in $\Gpd$ \[\mathbf{R}=\Big(\begin{tikzcd}
			{\Kk} & \Rr & \Hh
			\arrow[from=1-1, to=1-2]
			\arrow["r"below, from=1-2, to=1-3, shift right=1]
			\arrow["s"above, from=1-3, to=1-2, shift right=1]
		\end{tikzcd}\Big)\] and for every  morphism $(\eta,\xi,\id_\Hh)\colon \mathbf{S}\to \mathbf{R}$ in $\SES_\Hh(\Gpd)$ such that $sf = \xi$, there is a unique morphism $ (\tilde{\eta},\tilde{\xi},\id_\Hh) \colon \tilde{\mathbf{S}}\to \mathbf{R}$ in $\SES^{\mathrm{split},\text{\textup{\textsc{fit}}}}_\Hh(\Gpd)$ 
		such that the triangle in $\SES_\Hh(\Gpd)$
		\[\begin{tikzcd}
			{\tilde{\mathbf{S}}} & {\mathbf{R}} \\
			{\mathbf{S}}
			\arrow[from=1-1, to=1-2]
			\arrow[from=2-1, to=1-1]
			\arrow[from=2-1, to=1-2]
		\end{tikzcd}\]
		commutes.
	\end{theorem}
	\begin{proof}
		Let $\xi\colon \Gg\to\Rr$ be the unique morphism that restricts to the morphisms $\xi_i\colon \Gg_i\to \Rr$ on the connected components $\Gg_i$. Let $\nu_i = \xi^0_i(\lambda_i)$, so that $r^0(\nu_i)= \mu_i$ holds by the assumption $r\xi = f$. Let $\Nu = \Rr^0$, choose $\nu\in\Nu$ such that $r^0(\nu)= \mu$, and let $R = \Rr_\nu$. 
		
		We first prove that $\points{\im(\xi)}$ is a connected subgroupoid of $\Rr$; i.e.\@, that for all $a,b\in\Lambda$ there is an arrow in $\Rr$ between $\xi^0(a)$ and $\xi^0(b)$. This is actually the crucial part of the Theorem. We start by observing two facts.
		\begin{enumerate}
			\item[1)] If $a,b\in\Lambda_i$ for the same $i$, then they are connected by an arrow $x$ in $\Gg$, and hence their images are connected by $\xi^1(x)$ in $\Rr$.
			\item[2)] If $f^0(a) = f^0(b)$ in $\Mu$ (i.e.\@ $a$ and $b$ are connected in $\tilde{\Nn}$), one has the chain of implications
			\begin{align*}
				f^0(a) = f^0(b)&\implies r^0\xi^0(a) = r^0\xi^0(b)\\
				&\implies \xi^0(a)\text{ and }\xi^0(b)\text{ are connected in }\Kk\\
				&\implies \xi^0(a)\text{ and }\xi^0(b)\text{ are connected in }\Rr,
			\end{align*} 
			where the first implication comes from $f = r\xi$.
		\end{enumerate}
		
		Since $\points{\im(f)} = \Hh$ is connected, for all $a,b\in\Lambda$ one can find a sequence\footnote{The notion of connectivity in quivers is defined by the existence of a \textit{finite} path; thus such a finite sequence can always be found, and no problems arise if $\Lambda,\Mu,\Nu$ are infinite sets.} $a=a_1, a_2, a_3,\dots, a_n = b$ such that $a_{1}$ and $ a_2$ are connected in $\Gg$, $f^0(a_{2}) = f^0(a_{3})$, $a_3$ and $a_4$ are connected in $\Gg$, $f^0(a_4) = f^0(a_5)$, etc. If we apply 1) and 2) on this sequence, we obtain that $\im(\xi)$ is a connected quiver, and hence $\points{\im(\xi)}$, which has the same set of vertices as $\im(f)$, is a connected subgroupoid of $\Rr$.
		
		Since $\points{\im(\xi)}$ is connected, the chosen coarse subgroupoid of $\Hh$ that we identified with $\widehat{\Mu}$ is sent by $s$ into a coarse subgroupoid of $\Rr$, which we identify with $\widehat{\Nu}$; and there exist isomorphisms $C_{[\nu, \nu_i]}\colon \Rr_{\nu_i}\to \Rr_\nu = R$. Up to composing $\xi_i^\Gp$ with $C_{[\nu,\nu_i]}$, we obtain morphisms $\psi_i\colon G_i\to R$ for $i\in I$. Up to composing with similar isomorphisms, we also get a homomorphism $\psi_0 = s^\Gp\colon G_0=H\to R$, induced by the section $s$. Thus, by the universal property of $\tilde{G}$, one gets a unique morphism $\tilde{\xi}^\Gp\colon \tilde{G}\to R$. 
		
		The embedding of $\Gg^0$ into $\tilde{\Gg}^0$ is the identity $\id_\Lambda$, and the triangle given by $\id_\Lambda$, $\xi^0$ and $\tilde{\xi}^0$ must commute, therefore the definition $\tilde{\xi}^\Set = \xi^\Set$ is forced. One easily gets $r^\Set\tilde{\xi}^\Set = \tilde{f}^\Set$, and $r^\Gp\tilde{\xi}^\Gp = \tilde{f}^\Gp$ holds from \cref{lem:univ_groups}, thus $r\tilde{\xi} = \tilde{f}$.
		
		In order for $(\tilde{\eta}, \tilde{\xi}, \id_\Hh)$ to be a morphism of short exact sequences, we need to define $\tilde{\eta}$ as the restriction of $\tilde{\xi}$ to $\ker(\tilde{f})$. Clearly $\eta^\Set = \tilde{\eta}^\Set$. If an arrow $x$ lies in $\Nn^1$, one has $\tilde{f}^1\xi^1(x) = f^1(x) = 1$, thus the arrow $\xi^1(x) = \eta^1(x)$ lies in $\ker(\tilde{f})$, and this proves that the triangle
		\[\begin{tikzcd}
			{\ker(\tilde{f})} & \Kk \\
			\Nn
			\arrow["{\tilde{\eta}}", from=1-1, to=1-2]
			\arrow[hook, from=2-1, to=1-1]
			\arrow["\eta"', from=2-1, to=1-2]
		\end{tikzcd}\]
		commutes. Thus $(\tilde{\eta}, \tilde{\xi}, \id_\Hh)$ is a morphism in $\SES_\Hh(\Gpd)$, and the composition of $(\tilde{\eta}, \tilde{\xi}, \id_\Hh)$ with the canonical morphism $\mathbf{S}\to \tilde{\mathbf{S}}$ is a factorisation of the morphism $(\eta,\xi,\id_\Hh)$.
		
		We only need to prove that $(\tilde{\eta}, \tilde{\xi}, \id_\Hh)$ is a morphism in $\SES_\Hh^{\text{split}}(\Gpd)$. Namely, we need to prove that $\tilde{\xi}\tilde{s} = s$. If we call $\iota$ the canonical injection $\Gg\to \tilde{\Gg}$, one has
		\[ \tilde{\xi}\tilde{s} f = \tilde{\xi}\iota = \xi = sf,  \]
		whence $\tilde{\xi}\tilde{s} = s$ by cancelling the epimorphism $f$ on the right. 
	\end{proof}
	\begin{remark}
		Observe that the universal property of $\tilde{\mathbf{S}}$ is very close to the one of a free object. Indeed, it would have been a free object if, in the statement of \cref{thm:universal_prop_Gtilde}, \textit{every} morphism $(\eta,\xi,\id_\Hh)$ in $\SES_\Hh(\Gpd)$ induced a unique morphism $(\tilde{\eta},\tilde{\xi},\id_\Hh)$ in $\SES_\Hh^\mathrm{split}(\Gpd)$. However, in \cref{thm:universal_prop_Gtilde} it is moreover assumed that $(\eta,\xi,\id_\Hh)$ satisfies $sf = \xi$. This breaks the property of a free object.
	\end{remark}
	Since $\tilde{\mathbf{S}}$ satisfies a universal property, which determines it uniquely up to isomorphism, we may drop the explicit construction, and define $\tilde{\mathbf{S}}$ abstractly through the sole universal property. For the rest of this paper, though, we shall always refer to the explicit realisation of $\tilde{\mathbf{S}}$, in order to ease computations.
	\begin{definition}
		We call $\ker(\tilde{f})$ the \textit{virtual kernel} of $f$.
	\end{definition}
	\begin{remark}
		The above definition is well posed, independently of the concrete realisation of $\tilde{\mathbf{S}}$. Indeed, if $(\tilde{\Gg}, \tilde{f})$ and $(\tilde{\Gg}',\tilde{f}')$ satisfy the universal property in \cref{thm:universal_prop_Gtilde}, then there is an isomorphism $\tilde{\Gg}\cong \tilde{\Gg}'$ that closes the triangle with $\tilde{f}$ and $\tilde{f}'$, and hence induces an isomorphism $\ker(\tilde{f})\cong \ker(\tilde{f}')$.
	\end{remark}
	\begin{remark}
		In the hypotheses of \cref{thm:construction_Gtilde}, the universal groupoid $\tilde{\Gg}$ is always connected. One can see it as a consequence of \cref{thm:universal_prop_Gtilde}, with $\Rr = \tilde{\Gg}$, $r = \tilde{f}$ and $\eta,\xi$ the canonical morphisms: we have seen in the proof of \cref{thm:universal_prop_Gtilde} that $\points{\im(\xi)}$ is connected, but $\points{\im(\xi)}^0 = \Gg^0 = \tilde{\Gg}^0$, thus any two vertices in $\tilde{\Gg}^0$ are connected. 
	\end{remark}
	
	\begin{remark}\label{rem:splitting_needed} Consider the case when $\Gg = G$ and $\Hh= H$ are groups, and $f$ is a group homomorphism. Here $ \tilde{\Gg}= \tilde{G} =  (G*H) $ is also a group. The induced injection $H\to \tilde{G}$ is a section for $\tilde{f}$. 
		
		This constructs above $\ker(f)\to G\to H$ a `minimal' short exact sequence that splits. This is not the minimal short exact sequence enjoying a First Isomorphism Theorem: indeed, $\ker(f)\to G\to H$ is already a \textsc{fit} sequence. Thus the category $\SES_\Hh^{\text{split,\textsc{fit}}}(\Gpd)$ cannot be replaced with $\SES_\Hh^{\text{\textsc{fit}}}(\Gpd)$ in the statement of \cref{thm:universal_prop_Gtilde}.
	\end{remark}
	
	\subsection{Lifted First Isomorphism Theorem for Schurian groupoids} We denote by $\GpdSchur$ the full subcategory of $\Gpd$ consisting of Schurian groupoids. 
	
	In \cref{thm:universal_prop_Gtilde}, the sequence $\tilde{\mathbf{S}}$ is a universal splitting \textsc{fit} sequence for $\mathbf{S}$; and the word `splitting' cannot be removed from the universal property, as \cref{rem:splitting_needed} shows. However, in the subcategory $\GpdSchur$, the sequence $\tilde{\mathbf{S}}$ will actually be a universal \textsc{fit} sequence.
	\begin{proposition}\label{prop:universal_FIT_for_Shurian}
		In the construction of $\tilde{\mathbf{S}}$ from \cref{thm:construction_Gtilde}, let $\Gg$ and $\Hh$ be Schurian. Then $\tilde{\Gg}$ is Schurian (actually a coarse groupoid), and it satisfies the following universal property. If\[\mathbf{R} = \Big( \begin{tikzcd}
			\Kk & \Rr & \Hh
			\arrow[hook, from=1-1, to=1-2]
			\arrow["r", from=1-2, to=1-3]
		\end{tikzcd} \Big)\] is a \textsc{fit} sequence of Schurian groupoids, and $(\eta,\xi,\id_\Hh)\colon \mathbf{S}\to\mathbf{R}$ is a morphism in $\SES_\Hh(\GpdSchur)$, then there is a unique morphism $(\tilde{\eta},\tilde{\xi},\id_\Hh)\colon \tilde{\mathbf{S}}\to \mathbf{R}$ that makes the obvious triangle commute. 
		
		In other words: the functor $\mathbf{S}\mapsto \tilde{\mathbf{S}}$ from $\SES_\Hh(\GpdSchur)$ to $\SES_\Hh^{\text{\upshape\textsc{fit}}}(\GpdSchur)$ admits the inclusion $\SES_\Hh^{\text{\upshape\textsc{fit}}}(\GpdSchur)\to \SES_\Hh(\GpdSchur)$ as a right adjoint.
	\end{proposition}
	The proof of \cref{prop:universal_FIT_for_Shurian} will be handled more easily in a purely set-theoretic language. We thereby reformulate it in terms of sets and equivalence relations.
	
	As we anticipated, a Schurian groupoid $\Gg$ is simply an equivalence relation $\equiv_\Gg$ on $\Gg^0$; see \cite[Example 2]{GroupsToGroupoidsBrown}. 
	\begin{definition}
		Let $(\Lambda, \equiv_\Gg)$ and $(\Mu,\equiv_\Hh)$ be pairs of sets with equivalence relations. A map $f^0 \colon \Lambda \to \Mu$ satisfying
		\[ a\equiv_\Gg b \implies f^0(a)\equiv_\Hh f^0(b) \] 
		is called a \textit{morphism of equivalence relations}.
		
		For a morphism $f^0\colon (\Lambda, \equiv_\Gg)\to (\Mu,\equiv_\Hh)$, the \textit{kernel} is $\ker(f^0) = (\Lambda,\equiv_\Nn)$, where the equivalence relation $\equiv_\Nn$ on $\Lambda$ is defined as
		\[ a\equiv_\Nn b \iff a\equiv_\Gg b\,\text{ and }f^0(a) = f^0(b). \]
	\end{definition}
	It is easy to see that a morphism $f\colon \Gg\to \Hh$ between Schurian groupoids is identified with a map $f^0\colon \Gg^0\to \Hh^0$ between the sets of vertices which is a morphism of equivalence relations $(\Gg^0, \equiv_\Gg)\to (\Hh^0, \equiv_\Hh)$. Then,  $\GpdSchur$ is canonically equivalent to the category of equivalence relations in $\Set$. The above definition of kernels corresponds to the kernels of morphisms in $\GpdSchur$. 
	\begin{definition}If $f^0\colon \Lambda\to \Mu$ is a map of sets, and $\equiv$ is an equivalence relation on $\Lambda$, we call \textit{push-forward} of $\equiv$ the equivalence relation ${}_{f^0}\!\!\equiv$ on $\Mu$ generated by $\mu\, {}_{f^0}\!\!\equiv \mu'$ if there exist $\lambda,\lambda'\in \Lambda$, $\lambda\equiv \lambda'$, $f^0(\lambda)=\mu$, $f^0(\lambda') = \mu'$.\end{definition}
	Observe that, if $f$ is a morphism between two Schurian groupoids $\Gg$ and $\Hh$, then the subgroupoid of $\Hh$ corresponding to ${}_{f^0}\!\!\equiv_\Gg$ is exactly $\points{\im(f)}$.
	
	The short exact sequence $\Nn\to\Gg\to\Hh$ in $\GpdSchur$ with $\Hh$ connected, then, translates as a sequence of morphisms of equivalence relations
	\[\begin{tikzcd}
		{(\Lambda,\equiv_\Nn)} & {(\Lambda,\equiv_\Gg)} & {(\Mu, \Mu\times\Mu)}
		\arrow["{\iota^0}", from=1-1, to=1-2]
		\arrow["{f^0}", from=1-2, to=1-3]
	\end{tikzcd}\]
	such that ${}_{f^0}\!\!\equiv_\Gg$ is $\Mu\times \Mu$, and $(\Lambda,\equiv_\Nn)$ is the kernel of $f^0$. 
	
	Since the inclusion $\id_\Lambda\colon(\Lambda,\equiv_\Nn)\to (\Lambda,\equiv_\Gg)$ is a morphism of equivalence relations, the relation $\equiv_\Gg$ induces a well-defined relation on $\Lambda/\!\equiv_\Nn$, which we denote again by $\equiv_\Gg$. 
	
	\begin{remark}\label{rem:FIT_sequence_Schurian}The quotient groupoid $\Gg/\Nn$ corresponds to the equivalence relation $\equiv_\Gg$ on the set $\Lambda/\!\equiv_\Nn$.
		As a consequence, the short exact sequence $\Nn\to\Gg\to\Hh$ is \textsc{fit} if and only if $f^0$ induces a bijection $\Lambda/\!\equiv_\Nn\,\to \Mu$ that is an isomorphism of equivalence relations. This happens if and only if the condition
		\begin{equation}\label{eq:condition_FIT_Schurian}  f^0(a) = f^0(b) \iff a\equiv_\Nn b \end{equation}
		holds for all $a,b\in\Lambda$, as it is easy to see. Indeed, the surjectivity of $f^0$ is obvious, while the implication `$\!\implies\!$' in condition \eqref{eq:condition_FIT_Schurian} translates the injectivity. If \eqref{eq:condition_FIT_Schurian} holds, then, $f^0$ induces a bijection $\Lambda/\!\equiv_\Nn\,\to \Mu$ which is a morphism of equivalence relations; and the inverse map is also a morphism, again by the implication `$\!\implies\!$' in \eqref{eq:condition_FIT_Schurian}.\end{remark}
	We can finally translate \cref{prop:universal_FIT_for_Shurian} word by word in the set-theoretic language, and prove it.
	\begin{proposition}\label{prop:universal_FIT_for_equivalences}
		Let $\equiv_\Gg$ be an equivalence relation on a set $\Lambda$, let $\Mu$ be a set equipped with the coarsest equivalence relation $\Mu\times\Mu$, and let $f^0\colon \Lambda\to \Mu$ be a morphism of equivalence relations, such that the push-forward of $\equiv_\Gg$ via $f^0$ is $\Mu\times \Mu$ (in particular, $f^0$ must be surjective). 
		
		Let $(\Lambda,\equiv_\Nn) = \ker(f^0)$, and let $\tilde{\equiv}_\Nn$ be the equivalence relation on $\Lambda$
		\[  \lambda\,\tilde{\equiv}_\Nn\,\lambda' \iff f^0(\lambda)= f^0(\lambda')\]
		(in other words, $\tilde{\equiv}_\Nn$ is the kernel of $f^0$ seen as a morphism $(\Lambda, \Lambda\times\Lambda)\to (\Mu,\Mu\times\Mu)$).
		
		Then $(\Lambda,\tilde{\equiv}_\Nn)\to (\Lambda, \Lambda\times\Lambda)\to (\Mu,\Mu\times\Mu)$ is a \textsc{fit} sequence, the triple $(\id_\Lambda,\id_\Lambda, \id_\Mu)$ is a morphism of short exact sequences\footnote{Observe that $\id_\Lambda\colon (\Lambda,\equiv_\Gg)\to (\Lambda,\Lambda\times\Lambda)$ is obviously a morphism, but $\id_\Lambda\colon (\Lambda,\Lambda\times\Lambda)\to (\Lambda,\equiv_\Gg)$ is not a morphism, unless $\equiv_\Gg\, =  \Lambda\times\Lambda$. The same holds for $\id_\Lambda\colon (\Lambda,\equiv_\Nn)\to (\Lambda,\tilde{\equiv}_\Nn)$. This demonstrates the well-known fact that, in the category of equivalence relations, bijective morphisms need not be isomorphisms.}
		\[\begin{tikzcd}
			{(\Lambda,\tilde{\equiv}_\Nn)} & {(\Lambda,\Lambda\times\Lambda)} & {(\Mu,\Mu\times \Mu)} \\
			{(\Lambda,\equiv_\Nn)} & {(\Lambda,\equiv_\Gg)} & {(\Mu,\Mu\times \Mu)}
			\arrow[from=1-1, to=1-2]
			\arrow["{f^0}", from=1-2, to=1-3]
			\arrow[equals, from=1-3, to=2-3]
			\arrow["{\id_\Lambda}", from=2-1, to=1-1]
			\arrow[from=2-1, to=2-2]
			\arrow["{\id_\Lambda}", from=2-2, to=1-2]
			\arrow["{f^0}", from=2-2, to=2-3]
		\end{tikzcd}\]	
		and the sequence $(\Lambda,\tilde{\equiv}_\Nn)\to (\Lambda, \Lambda\times\Lambda)\to (\Mu,\Mu\times\Mu)$ is `minimal' above the sequence $(\Lambda,\equiv_\Nn)\to (\Lambda,\equiv_\Gg)\to (\Mu,\Mu\times\Mu)$ in the following universal sense: if the sequence 
		\[\begin{tikzcd}
			{\ker(r^0)} & {(\Nu,\equiv_\Rr)} & {(\Mu,\Mu\times \Mu)}
			\arrow[from=1-1, to=1-2]
			\arrow["{r^0}", from=1-2, to=1-3]
		\end{tikzcd}\]
		is a \textsc{fit} sequence of equivalence relations, and if there is a morphism of short exact sequences
		\[\begin{tikzcd}
			{\ker(r^0)} & {(\Nu,\equiv_\Rr)} & {(\Mu,\Mu\times \Mu)} \\
			{(\Lambda,\equiv_\Nn)} & {(\Lambda,\equiv_\Gg)} & {(\Mu,\Mu\times \Mu)}
			\arrow[from=1-1, to=1-2]
			\arrow["{r^0}", from=1-2, to=1-3]
			\arrow[equals, from=1-3, to=2-3]
			\arrow["{\eta^0}", from=2-1, to=1-1]
			\arrow[from=2-1, to=2-2]
			\arrow["{\xi^0}", from=2-2, to=1-2]
			\arrow["{f^0}", from=2-2, to=2-3]
		\end{tikzcd}\]
		then there exists a unique morphism of short exact sequences 
		\[\begin{tikzcd}
			{\ker(r^0)} & {(\Nu,\equiv_\Rr)} & {(\Mu,\Mu\times \Mu)} \\
			{(\Lambda,\tilde{\equiv}_\Nn)} & {(\Lambda,\Lambda\times\Lambda)} & {(\Mu,\Mu\times \Mu)}
			\arrow[from=1-1, to=1-2]
			\arrow["{r^0}", from=1-2, to=1-3]
			\arrow[equals, from=1-3, to=2-3]
			\arrow["{\tilde{\eta}^0}", from=2-1, to=1-1]
			\arrow[from=2-1, to=2-2]
			\arrow["{\tilde{\xi}^0}", from=2-2, to=1-2]
			\arrow["{f^0}", from=2-2, to=2-3]
		\end{tikzcd}\]
		that makes the obvious triangle commute.
	\end{proposition}
	\begin{proof}
		Let $\ker(r^0) = (\Nu, \equiv_\Kk)$. Since the sequence $\ker(r^0)\to (\Nu,\equiv_\Rr)\to (\Mu,\Mu\times \Mu)$ is \textsc{fit}, it follows from \eqref{eq:condition_FIT_Schurian} that $\equiv_\Kk$ is the equivalence relation
		\[ a \equiv_\Kk b \iff r^0(a) = r^0(b). \]
		
		It is clear from \eqref{eq:condition_FIT_Schurian} that the lifted sequence $(\Lambda,\tilde{\equiv}_\Nn)\to (\Lambda,\Lambda\times\Lambda)\to (\Mu,\equiv_\Hh)$ is \textsc{fit}. 
		
		We define the map $\tilde{\eta}^0 $ as $\eta^0$, and the map $\tilde{\xi}^0$ as $\xi^0$. This is the only possible choice, since the triangles
		\[\begin{tikzcd}
			{(\Lambda,\tilde{\equiv}_\Nn)} & {(\Nu, \equiv_\Kk)} && {(\Lambda,\Lambda\times\Lambda)} & {(\Nu,\equiv_\Rr)} \\
			{(\Lambda,\equiv_\Nn)} &&& {(\Lambda,\equiv_\Gg)}
			\arrow["{\tilde{\eta}^0}", from=1-1, to=1-2]
			\arrow["{\tilde{\eta}^0}", from=1-4, to=1-5]
			\arrow["{\id_\Lambda}", from=2-1, to=1-1]
			\arrow["{\eta^0}"', from=2-1, to=1-2]
			\arrow["{\id_\Lambda}", from=2-4, to=1-4]
			\arrow["{\eta^0}"', from=2-4, to=1-5]
		\end{tikzcd}\]
		need to commute. The third triangle with $\tilde{\xi}^0$, $r^0$, and $f^0$ commutes automatically, because $r^0\xi^0=f^0$ by assumption. If $\tilde{\eta}^0$ and $\tilde{\xi}^0$ are morphisms, it is clear that the triple $(\tilde{\eta}^0, \tilde{\xi}^0, \id_\Mu)$ is a morphism of short exact sequences.
		
		We first observe that $\tilde{\eta}^0$ is a morphism: indeed,
		\begin{align*}
			a\,\tilde{\equiv}_\Nn b &\implies f^0(a)= f^0(b)\\
			&\implies r^0 \xi^0(a) = r^0\xi^0(b)\\
			&\implies \xi^0(a)\equiv_\Kk \xi^0(b)
		\end{align*}
		by definition of $\equiv_\Kk$.
		
		Proving that $\tilde{\xi}^0$ is a morphism is less immediate (and it corresponds to the part of the proof of \cref{thm:universal_prop_Gtilde} where we proved that $\points{\im(\xi)}$ is a  connected subgroupoid of $\Rr$). Since every two elements $a,b\in\Lambda$ are equivalent under the relation $\Lambda\times\Lambda$, we need to prove that $\xi^0(a)\equiv_\Rr \xi^0(b)$ for all $a,b\in\Lambda$. We begin with the following facts.
		\begin{enumerate} \item[1)] If $a\equiv_\Gg b$, then $\xi^0(a)\equiv_\Rr \xi^0(b)$. This holds because $\xi^0$ is a morphism $(\Lambda, \equiv_\Gg)\to (\Nu,\equiv_\Rr)$.
			\item[2)] If $a\,\tilde{\equiv}_\Nn b$, then $\xi^0(a)\equiv_\Rr \xi^0(b)$. This follows again from $r^0\xi^0 = f^0$, and from the definition of $\tilde{\equiv}_\Nn$:
			\begin{align*}
				a\,\tilde{\equiv}_\Nn b&\iff f^0(a) = f^0(b)\\
				&\implies r^0\xi^0(a) = r^0\xi^0(b)\\
				&\implies \xi^0(a)\equiv_\Kk \xi^0(b)\\
				&\implies \xi^0(a)\equiv_\Rr \xi^0(b),
			\end{align*}
			where the last step follows from the fact that the inclusion $(\Nu,\equiv_\Kk)\to (\Nu, \equiv_\Rr)$ is a morphism.
		\end{enumerate}
		Our strategy is the following: we shall find a sequence $a = a_1, a_2, a_3, \dots, a_n = b$, such that 
		\[ a_1 \equiv_\Gg a_2 \, \tilde{\equiv}_\Nn a_3\equiv_\Gg a_4  \tilde{\equiv}_\Nn \dots \]
		If we succeed in finding such a sequence, then clearly $\xi^0(a)\equiv_\Rr\xi^0(b)$. 
		
		The existence of this sequence comes from the fact that ${}_{f^0}\!\!\equiv_\Gg$ is the coarsest equivalence $\Mu\times\Mu$ (in terms of groupoids: the quiver $\im(f)$ is connected, and hence $\Hh = \points{\im(f)}$). The equivalence ${}_{f^0}\!\!\equiv_\Gg$ is generated by the relation
		\[  x \approx y \iff x = f^0(a), \; y = f^0(b),\; a\equiv_\Gg b, \]
		which is generally not transitive. Thus the equivalence ${}_{f^0}\!\!\equiv_\Gg$ is the transitive closure of $\approx$, and hence $x\, {}_{f^0}\!\!\equiv_\Gg y $ if and only if there exists a sequence $x = x_1, x_2, x_3,\dots, x_n= y$ satisfying
		\begin{equation}\label{eq:sequence_x}
			x_1\approx x_2,\quad x_2 = x_3, \quad x_3\approx x_4,\quad x_4 = x_5,\dots
		\end{equation}
		Using the fact that $f^0$ is surjective, we now write $x= f^0(a)$, $y= f^0(b)$. Thus $f^0(a)\,{}_{f^0}\!\!\equiv_\Gg f^0(b)$ if and only if there exists a sequence $a = a_1, a_2, a_3,\dots, a_n = b$ satisfying 
		\begin{equation}\label{eq:sequence_a} a_1 \equiv_\Gg a_2 \, \tilde{\equiv}_\Nn a_3\equiv_\Gg a_4 \, \tilde{\equiv}_\Nn \dots \end{equation}
		where \eqref{eq:sequence_a} translates \eqref{eq:sequence_x} verbatim. This concludes the proof.
	\end{proof}
	An immediate consequence of \cref{prop:universal_FIT_for_equivalences} is the following.
	\begin{corollary} Let $f^0\colon(\Lambda, \equiv_\Gg)\to (\Mu, \Mu\times\Mu)$ be a morphism and $\tilde{\equiv}_\Nn$ be defined as above. The smallest equivalence relation $\tilde{\equiv}_\Gg$ on $\Lambda$ that makes \[\begin{tikzcd}
			{(\Lambda,\tilde{\equiv}_\Nn)} & {(\Lambda,\tilde{\equiv}_\Gg)} & {(\Mu, \Mu\times\Mu)}
			\arrow["\id_\Lambda", from=1-1, to=1-2]
			\arrow["{f^0}", from=1-2, to=1-3]
		\end{tikzcd}\] into a short exact sequence is forced to be $\tilde{\equiv}_\Gg = \Lambda\times\Lambda$---no matter what $\equiv_\Gg$ is.\end{corollary}
	\begin{proof} If $\tilde{\equiv}_\Gg$ is such an equivalence relation, one has the implications
		\[ a\equiv_\Gg b \implies f^0(a) = f^0(b)\implies a\,\tilde{\equiv}_\Nn b \implies a\,\tilde{\equiv}_\Gg b, \]
		which means that $\id_\Lambda \colon (\Lambda,\equiv_\Gg)\to (\Lambda, \tilde{\equiv}_\Gg)$ is a morphism. Thus the identities $(\id_\Lambda,\id_\Lambda, \id_\Mu)$ assemble to a morphism of short exact sequences
		\[\begin{tikzcd}
			{(\Lambda,\tilde{\equiv}_\Nn)} & {(\Lambda, \tilde{\equiv}_\Gg)} & {(\Mu, \Mu\times \Mu)} \\
			{(\Lambda,\equiv_\Nn)} & {(\Lambda, \equiv_\Gg)} & {(\Mu, \Mu\times \Mu)}
			\arrow[from=1-1, to=1-2]
			\arrow["{f^0}", from=1-2, to=1-3]
			\arrow["{\id_\Lambda}", from=2-1, to=1-1]
			\arrow[from=2-1, to=2-2]
			\arrow["{\id_\Lambda}", from=2-2, to=1-2]
			\arrow["{f^0}", from=2-2, to=2-3]
			\arrow[equals, from=2-3, to=1-3]
		\end{tikzcd}\]
		which satisfies the hypotheses of \cref{prop:universal_FIT_for_equivalences}, and hence factors through the sequence $(\Lambda,\tilde{\equiv}_\Nn) \to (\Lambda, \Lambda\times\Lambda)\to (\Mu, \Mu\times \Mu)$. This means in particular that $\tilde{\equiv}_\Gg$ contains $\Lambda\times \Lambda$, whence $\tilde{\equiv}_\Gg = \Lambda\times\Lambda$.
	\end{proof}
	\begin{example}
		The situation in \cref{fig:vker_1} is an example of the lifted First Isomorphism Theorem in $\GpdSchur$. Here $\Lambda = \{\lambda,\mu,\lambda',\mu'\}$, the equivalence relation $\equiv_\Gg$ is given by  $\{\lambda\equiv_\Gg \mu,\, \lambda'\equiv_\Gg\mu'\}$, and the equivalence relation $\tilde{\equiv}_\Nn$ is $\{\mu \,\tilde{\equiv}_\Nn \mu'\}$. Since the image of $f$ is connected, one has 
		\[ \lambda \equiv_\Gg \mu \,\tilde{\equiv}_\Nn \mu' \equiv_\Gg \lambda', \]
		thus the minimum equivalence relation on $\Lambda$ that contains both $\,\tilde{\equiv}_\Nn$ and $\equiv_\Gg$ is clearly $\Lambda\times\Lambda$.
	\end{example}
	\begin{remark}
		In the category $\Gpd$, the case of groups and the case of Schurian groupoids are, in some sense, two opposite extrema. The universal split \textsc{fit} sequence $\tilde{\mathbf{S}}$ collapses, in these two extrema, to two different universal objects. Namely, to:
		\begin{enumerate}
			\item a universal split sequence, in the case of groups;
			\item a universal \textsc{fit} sequence, in the case of Schurian groupoids.
		\end{enumerate}
		The fact that $\tilde{\mathbf{S}}$ is universal split for groups is very much expected, because \textit{every short exact sequence of groups is already \textsc{fit}}. 
		
		The fact that $\tilde{\mathbf{S}}$ is universal \textsc{fit} for Schurian groupoids is much more surprising, because \textit{not every short exact sequence of Schurian groupoids splits}. A counterexample is the sequence in \cref{fig:vker_1}, where $\Hh$ is not a subgroupoid of $\Gg$. Therefore, \cref{prop:universal_FIT_for_Shurian} does not follow as an immediate consequence of \cref{thm:universal_prop_Gtilde}.
	\end{remark}
	\subsection{Universal lifting of split short exact sequences} We now consider again the category $\Gpd$ of all groupoids, not necessarily Schurian. \begin{lemma}In the construction of \cref{thm:construction_Gtilde}, suppose that the epimorphism $f$ already has a splitting $s$. Then $\tilde{s}$ can be defined in a way that the canonical inclusion $\mathbf{S}\to \tilde{\mathbf{S}}$ is a morphism in $\SES_\Hh^{\mathrm{split}}(\Gpd)$; i.e., the composition $\tilde{s}f$ equals the canonical morphism $\Gg\to \tilde{\Gg}$.\end{lemma}
	\begin{proof}
		Since $\Hh$ is connected, the image of $s$ is entirely contained in a connected component $\Gg_{\bar{\imath}}$ for an index $\bar{\imath} \in I$. The construction from \cref{thm:construction_Gtilde} requires that we choose arbitrary vertices $\lambda_i\in\Lambda_i$, and a vertex $\mu\in\Mu$ among the set of vertices $\mu_i = f^0(\lambda_i)$. Up to modifying this choices, then, we may assume that $s^0 (\mu) = \lambda_{\bar{\imath}}$. 
		
		For $\tilde{s}$, then, we define $\tilde{s}^\Set$ as the set-theoretic section $s^\Set$; and $\tilde{s}^\Gp$ as the homomorphism $H\to \tilde{G}$ obtained by $s^\Gp\colon H\to G_{\bar{\imath}}$ composed with the canonical map $G_{\bar{\imath}}\to \tilde{G}$. This is clearly a section.
	\end{proof}
	\begin{proposition}\label{prop:universal_FIT_for_split}
		The sequence $\tilde{\mathbf{S}}$ and the canonical inclusion $\mathbf{S}\to \tilde{\mathbf{S}}$ of split short exact sequences are universal, in the following sense: if $\mathbf{R}$ is a split \textsc{fit} sequence, and $\mathbf{S}\to \mathbf{R}$ is a morphism of split sequences, then there is a unique morphism of split sequences $\tilde{\mathbf{S}}\to \mathbf{R}$ that makes the obvious triangle commute.
		
		In other words, the functor $\mathbf{S}\mapsto \tilde{\mathbf{S}}$ from $\SES_\Hh^{\mathrm{split}}(\Gpd)$ to $\SES_\Hh^{\text{\upshape split,\textsc{fit}}}(\Gpd)$ admits the inclusion $\SES_\Hh^{\text{\upshape split,\textsc{fit}}}(\Gpd)\to \SES_\Hh^{\mathrm{split}}(\Gpd)$ as a right adjoint.
	\end{proposition}
	\begin{proof}
		Once we know that the functor $\mathbf{S}\mapsto \tilde{\mathbf{S}}$ from \cref{thm:construction_Gtilde} can be constructed in a way that respects the splittings, the rest of the proposition is a consequence of \cref{thm:universal_prop_Gtilde}.
	\end{proof}

	\section{On a classical notion of semidirect product}\label{sec:semidirect} 
	The term `semidirect product' applied to groupoids is no novelty in the mathematical literature. The product `group by coarse groupoid' in \cref{prop:product_groupd_coarse_groupoid} is an example. Some notions of product for groupoids are mentioned in the entire oeuvre of R.\@ Brown (see e.g.\@ \cite{BrownBookGroupoids}), more recently semidirect products have been discussed in Ibort and Marmo \cite{IbortMarmoGroupoids}, Metere and Montoli \cite{MetereMontoliSemidirectIntGpd}, and in many others places. Moreover, Zappa--Szép products are introduced in \cite{AguiarAndruskiewitsch}, and extensively discussed e.g.\@ in \cite{duwenig2024zappa}; and a notion of \textit{semidirect product of categories}, seemingly unrelated to our research, is defined in Steinberg \cite{SteinbergSemidirectCategories}.
	
	Semidirect products of groupoids in the sense of Brown do not provide a good Split Lemma in general. This issue constitutes the premise to our investigation, and the rationale for introducing \textit{crossed products} in \cref{subsec:crossed_Gpd}. 
	
	In this section, however, we make sense of these classical semidirect products, proving a Split-like Lemma that is not the categorial Split Lemma in the sense of Bourn and Janelidze \cite{BournJanelidzeProtomodularityDescSemProd}, but it is perhaps just as useful in the applications.
	\subsection{Groupoid (semistrong) actions on a quiver}
	In the rest of this paper, we shall mostly need what we call `semistrong' actions in $\Gpd_\Lambda$ (see \cite[\S1.4]{andruskiewitsch2005quiver}). However, for completeness, we first give the fundamentals of a theory of actions in $\Gpd$, where the set of vertices is not fixed.
	
	The following definition distances itself from the notion of \textit{action of a groupoid on a groupoid} which is given in \cite[Definition 1.1]{BrownGroupoidsAsCoefficients}. However, we adopt it here for three reasons. First, because it is the definition we need to work with. Second, because it appears as a more straightforward oidification of the notion of \textit{action of a group on a set}. Third, because it is compatible with the definition of \textit{action of a group $G$ on an object $X$} in a generic locally small category $\Cc$, which is defined as a group homomorphism $G\to \Aut_\Cc(X)$. Observe moreover that our definition is not new, since we are building on the notion of \textit{action of a category on a category} given by Tilson \cite{TilsonCatAsAlg}.
	
	Consider a category $\Cc\in \{\Quiv,\Quiv_\Lambda,\Gpd,\Gpd_\Lambda\}$, and an object $Q$ in $\Cc$. We denote by  $\Aut_{\Cc}(Q)$ the group of automorphism of $Q$ in $\Cc$; by $\End_{\Cc}(Q)$ the monoid of endomorphisms; and by $\FF_\Cc(Q)\subseteq \End_\Cc(Q)$ the submonoid of fully faithful endomorphisms, i.e.\@ the endomorphisms that are bijective on the arrows. Observe that $\FF_{\Cc}(Q) =\Aut_{\Cc}(Q)$ if $\Cc \in \{\Quiv_\Lambda,\Gpd_\Lambda\}$, because every strong endomorphism $f$ over $\Lambda$ has $f^0 = \id_\Lambda$.
	
	\begin{definition}
		Let $\Gg$ be a groupoid, and $Q$ a quiver. A \textit{left action} of $\Gg$ on $Q$ is a functor $\Gg\to \FF_\Quiv(Q)$, where the monoid $\FF_\Quiv(Q)$ is regarded as a category with a single object. Likewise, a \textit{right action} is a functor $\Gg^{op}\to \FF_\Quiv(Q)$, where $\Gg^{op}$ is the quiver with reversed arrows ($\source_{\Gg^{op}} =\target_\Gg$, $\target_{\Gg^{op}}= \source_\Gg$) and opposite groupoid operation.
	\end{definition}
	As we shall see, in the above definition $\FF_\Quiv(Q)$ can be replaced with $\Aut_\Quiv(Q)$ whenever $Q^0 = \im(\source)\cup \im(\target)$. 
	
	Let $\vartheta\colon \Gg\to \FF_\Quiv(Q)$ be a left action. We also denote an action by symbols such as $\triangleright$, and use the (slightly abusive) notation $g\triangleright x$ for $\vartheta^1_g(x)$ and $g\triangleright \lambda$ for $\vartheta^0_g(\lambda)$, where $g\in \Gg^1$, $x\in Q^1$, and $\lambda\in Q^0$.
	
	\begin{lemma}\label{lem:weak_action_properties}
		In the above setting, one has:
		\begin{enumerate}
			\item $gh\triangleright x = g\triangleright (h\triangleright x)$ for all $g\ot h\in \Gg\ot \Gg$, $x\in Q$;
			\item $1_\lambda\triangleright x = x$ for all $x\in Q^1$, $\lambda\in\Gg^0$;
			\item the inverse map of $\vartheta^1_g$ is $\vartheta^1_{g^{-1}}$ for all $g\in \Gg^1$;
			\item $\vartheta^0_g$ is bijective from $\im(\source_Q)\cup \im(\target_Q)$ to $\im(\source_Q)\cup \im(\target_Q)$, for all $g\in \Gg^1$. In particular, if $Q^0 = \im(\source_Q)\cup \im(\target_Q)$, then the image of $\vartheta$ is actually contained in $\Aut_\Quiv(Q)\subseteq \FF_\Quiv(Q)$.
		\end{enumerate}
	\end{lemma}
	\begin{proof} It is the same as for group actions, with some more technicalities.
		\begin{enumerate}
			\item Since $\vartheta$ is a functor, one has $\vartheta^1_g\vartheta^1_h = \vartheta^1_{gh}$ for all $g\ot h\in \Gg\ot \Gg$.
			\item One has $1_\lambda \triangleright (1_\lambda\triangleright x) = 1_\lambda1_\lambda\triangleright x = 1_\lambda\triangleright x$, whence $1_\lambda\triangleright x = x$ because $\vartheta^1_{1_\lambda}$ is invertible.
			\item Immediate computation, using the previous two points.
			\item If $\lambda$ is the source, resp.\@ the target of $x$, then $g\triangleright \lambda$ is the source, resp.\@ the target of $g\triangleright x$. Thus $\vartheta^0_g$ restricts indeed to an endomap of $\im(\source_Q)\cup \im(\target_Q)$ for all $g$. Since $\vartheta^1_{g^{-1}}$ is the inverse of $\vartheta^1_g$, for all $x\in Q^1$ one has
			\begin{align*}
				\vartheta^0_{g^{-1}}\vartheta^0_g \target(x)&= 	\vartheta^0_{g^{-1}}\target \vartheta^1_g(x)\\
				&= \target \vartheta^1_{g^{-1}}\vartheta^1_g(x)\\
				&= \target (x),
			\end{align*}
			and one similarly proves $\vartheta^0_g\vartheta^0_{g^{-1}} \target(x) = \target(x)$. With the same proof, one shows that $\vartheta^0_{g^{-1}}$ is the inverse of $\vartheta^0_g$ on $\im(\source_Q)$.\qedhere
		\end{enumerate}
	\end{proof}	
	We now give a version of actions over a fixed set of vertices. 
	
	In $\Quiv_\Lambda$, we define a \textit{strong action} of $\Gg$ on $Q$ as a groupoid morphism $\Gg\to \Aut_{\Quiv_\Lambda}(Q)$. However, this notion appears immediately to be very restrictive. In particular, if $Q$ is Schurian, then every $g\in \Gg$ acts as the identity. Most remarkably, the left multiplication $\Gg\ot \Gg\to \Gg$ is not a left strong action of $\Gg$ on $\Gg$. Thus we shall employ the following weaker, but much more useful definition. The term \textit{semistrong action} is our own, and we use it to distinguish it from more general actions; but the definition appeared in Andruskiewitsch \cite{andruskiewitsch2005quiver} and several other places.
	\begin{definition}[{\cite[\S1.4]{andruskiewitsch2005quiver}}]
		Let $\Gg$ be a groupoid over $\Lambda$, and $Q$ a quiver over $\Lambda$. 
		
		A \textit{left semistrong action} of $\Gg$ on $Q$ is a morphism $\triangleright\colon \Gg\ot Q\to Q$, satisfying $gh\triangleright x = g\triangleright (h\triangleright x)$ and $1_{\source(x)}\triangleright x = x$ for all $g\ot h\ot x\in \Gg\ot \Gg\ot Q$; and such that the action $g\triangleright\blank$ on the vertices is a bijection $\Lambda\to \Lambda$ and sends $\target(g)$ to $\source(g)$. A \textit{right semistrong action} $\triangleleft\colon Q\ot \Gg\to Q$ is defined analogously, with the action on the vertices satisfying $\source(g)\triangleleft g = \target(g)$.
	\end{definition}
	\begin{remark}
		The multiplication $\Gg\ot \Gg\to \Gg$ on a groupoid $\Gg$ yields a left and a right semistrong action (but not a strong action) of $\Gg$ on itself.
	\end{remark}
	We shall use the terms (\textit{semistrong}) \textit{left} or \textit{right $\Gg$-module}, and (\textit{semi\-strong}) \textit{left} or \textit{right $\Gg$-module algebra}, carrying the obvious meaning. We say that a quiver $Q$ is a (\textit{semi\-strong}) $\Gg$-\textit{bimodule} with respect to (semistrong) actions $\triangleright, \triangleleft$, left and right respectively, if the usual bimodule compatibility $(g\triangleright x)\triangleleft h = h\triangleright (x\triangleleft h)$ holds. 
	\begin{remark}
		As it was pointed out in \cite[\S 1.4]{andruskiewitsch2005quiver}, there is an equivalence between semistrong actions of $\Gg$ on $Q$, and strong morphisms of groupoids $\Gg\to \mathfrak{aut}(Q)$, where the groupoid $\mathfrak{aut}(Q)$ over $\Lambda$ is defined as
		\[
		\big(\mathfrak{aut}(Q)\big)^1= \{ (\lambda, f,\mu) \mid \lambda,\mu\in \Lambda\text{ and }f\colon Q(\mu,\Lambda)\to Q(\lambda,\Lambda)\text{ is a bijection} \},\]
		\[ \source(\lambda,f,\mu)=\lambda,\quad \target(\lambda,f,\mu) = \mu,
		\quad  (\lambda,f,\mu)\cdot (\mu,g,\nu) = (\lambda, f\circ g, \nu).\]
		Indeed, if $x\in \Gg^1$ is an arrow $\lambda\to \mu$, a semistrong action $\triangleright$ induces a bijection \[x\triangleright \blank\colon Q(\mu,\Lambda)\to Q(\lambda,\Lambda),\] and the action is entirely characterised by this family of bijections. 
		
		Observe that $\mathfrak{aut}(Q)$ is, in some sense, more natural than $\Aut_{\Quiv_\Lambda}(Q)$: because it is a groupoid and not a group, and because it allows us to encode the semistrong actions of groupoids on $Q$.
	\end{remark}
	
	\subsection{Semidirect products of groupoids} Let $\Aa$ and $\Bb$ be groupoids, and $\vartheta\colon \Bb\to \Aut_\Quiv(\Aa)$ be a left action of $\Bb$ on the quiver $\Aa$. With a slight abuse, we also use the symbol $\triangleright$ for both the actions $\vartheta^1, \vartheta^0$. Suppose moreover that $b\triangleright a_1 a_2 = (b\triangleright a_1)(b\triangleright a_2)$ whenever the left-hand side is defined---in other words, $\Aa$ is a $\Bb$-module algebra. 
	\begin{definition}
		The semidirect product $\Aa \lfibre{}{\triangleright} \Bb$ is the quiver
		\[ (\Aa^1\times \Bb^1)\twomapsright{\Source}{\Target}(\Aa^0\times \Bb^0), \]
		where $\Source(a\times b) = \source(a)\times \source(b)$ and $\Target(a\times b) = (b^{-1}\triangleright \target(a))\times \target(b)$, endowed with the multiplication
		\begin{align*} m\colon & (\Aa \lfibre{}{\triangleright} \Bb)\fibre{\Target}{\Source}(\Aa \lfibre{}{\triangleright} \Bb)\to \Aa \lfibre{}{\triangleright} \Bb \\
			& (a_1\times b_1)\cdot (a_2\times b_2) = a_1(b_1\triangleright a_2)\times (b_1b_2) \end{align*}
		and family of units $1_{\lambda\times \mu} = 1_\lambda\times 1_\mu$.
	\end{definition}
	
	\begin{lemma}
		With the above operations, and the inverses
		\[ (a\times b)^{-1} = (b^{-1}\triangleright a^{-1})\times b^{-1},  \]
		the semidirect product becomes a groupoid.
	\end{lemma}
	\begin{proof}
		First of all, we observe that 
		\begin{align*}
			& \Source((a_1\times b_1)\cdot (a_2\times b_2)) = \Source(a_1\times b_1),\\
			& \Target((a_1\times b_1)\cdot (a_2\times b_2))= \Target(a_2\times b_2).
		\end{align*}
		The first relation is obvious, while for the second relation one has
		\begin{align*} \Target (a_1(b_1\triangleright a_2)\times (b_1b_2)) &= \big(b_2^{-1}b_1^{-1}\triangleright \target(b_1\triangleright a_2)\big)\times \target(b_2)\\ & = \big(b_2^{-1}b_1^{-1}\triangleright (b_1\triangleright \target(a_2))\big)\times \target(b_2)\\ &= (b_2^{-1}\triangleright \target(a_2)) \times \target(b_2) \\ & = \Target(a_2\times b_2).\end{align*}
		Then, we check that $(a\times b)^{-1}$ has indeed source in $\Target(a\times b)$, and target in $\Source(a\times b)$:
		\begin{align*}
			\Source((b^{-1}\triangleright a^{-1})\times b^{-1}) &= (b^{-1}\triangleright \source(a^{-1}))\times \source(b^{-1}) = (b^{-1}\triangleright \target(a))\times \target(b)=\Target(a\times b),\\
			\Target((b^{-1}\triangleright a^{-1})\times b^{-1}) &= (b\triangleright \target(b^{-1}\triangleright a^{-1}))\times \target(b^{-1}) = \target(a^{-1})\times \target(b^{-1}) = \Source(a\times b).
		\end{align*}
		All the other verifications are routine, and they are the same as for the semidirect product of groups.
	\end{proof}
	\begin{remark}\label{rem:semidirectgroupbygroupoids}
		Let $G$ be a group, i.e.\@ a groupoid over a singleton $\{\bullet\}$. Let $\triangleright $ be an action of a groupoid $\Hh$ on the quiver $G$, such that $G$ is an $\Hh$-module algebra. Since $G^0 = \{\bullet\}$, the action $\triangleright$ is strong, and hence $\Gg = G\lfibre{}{\triangleright}\Hh$ has same source and target as the quiver $G\times \Bb$. This is exactly the classically known product `group by groupoid' that appears in \cite{BrownBookGroupoids}. Observe that $\Hh$, here, is not necessarily coarse. Semidirect products `group by groupoid' with $\Hh$ non necessarily coarse are used e.g.\@ in \cite{brown2002discontinuousaction}. The set $\Gg^0 = \{\bullet\}\times \Hh^0$ can be identified with $\Hh^0$, and the isotropy group $\Gg_\lambda$ is the group $G\rtimes_\triangleright \Hh_\lambda$.
	\end{remark}
	\begin{remark}
		The group bundle over the set of vertices $\Lambda$, with all isotropy groups isomorphic to $G$, is obtained as $G\lfibre{}{\triangleright} \One_\Lambda$. Here the action $\triangleright$ of $\One_\Lambda$ on $G$ is forced to be trivial, since $1_\lambda\triangleright g = g$ for all $g\in G$, $\lambda\in\Lambda$.
	\end{remark}
	For the case described in \cref{rem:semidirectgroupbygroupoids}, the following result is known (see Metere and Montoli \cite{MetereMontoliSemidirectIntGpd}, and later Ibort and Marmo \cite{IbortMarmoGroupoids})
	\begin{proposition}\label{prop:MetereMontoli}
		Let $\Nn\to \Gg\to \Hh$ be a split \textsc{fit} sequence in $\Gpd_\Lambda$, and hence $\Nn$ be a group bundle $N\rtimes \One_\Lambda$ for some group $N$. Then, one has $\Gg \cong N\rtimes_\triangleright \Hh$, where the action $\triangleright$ is induced by the conjugation in $\Gg$.
	\end{proposition}
	In \cref{sec:crossed}, we shall generalise the same result to a Split Lemma in $\Gpd$, from the viewpoint of crossed products. No such Split Lemma in $\Gpd$ exists in the language of these classical semidirect products: this is what we demonstrate by means of examples and counterexamples.
	\begin{example}
		\label{ex:9Schurian}
		Let $\Gg = \widehat{9}$, $\Nn = \widehat{\{0,3,6 \}}\sqcup\widehat{\{1,4,7 \}}\sqcup \widehat{\{2,5,8 \}}$ and $\Hh = \widehat{3}$, as in \cref{fig:crossed9}. Consider the splitting \textsc{fit} sequence $\Nn\to \Gg\to \Hh$ where $\Hh$ is identified with the quotient $\Gg/\Nn$. Clearly, it is not possible to describe $\Gg$ as a semidirect product $\Nn\lfibre{}{}\Hh$ in any possible way, since $\Gg^0$ has $9$ elements, while $|\Nn^0\times \Hh^0| = 9\cdot 3 = 27$.
		
		Let $\Kk = \widehat{\{ 0,3,6 \}}$ be a connected component of $\Nn$. We can hope to describe $\Gg$ as a semidirect product of $\Kk$ and $\Hh$, since now both the number of vertices ($3\cdot 3 = 9$) and of arrows ($3^2\cdot 3^2 = 9^2$) are the right ones. 
		
		In this special case, we are lucky. Indeed, every vertex $n = 0,\dots, 8$ can be written uniquely as $n = 3k + r$ for $r = 0,1,2$; and the map $+\colon \Kk^0\times \Hh^0\to \Gg^0$ is a bijection. It is easy, then, to describe $\Gg$ as a (semi)direct product $\Kk\times \Hh$.
		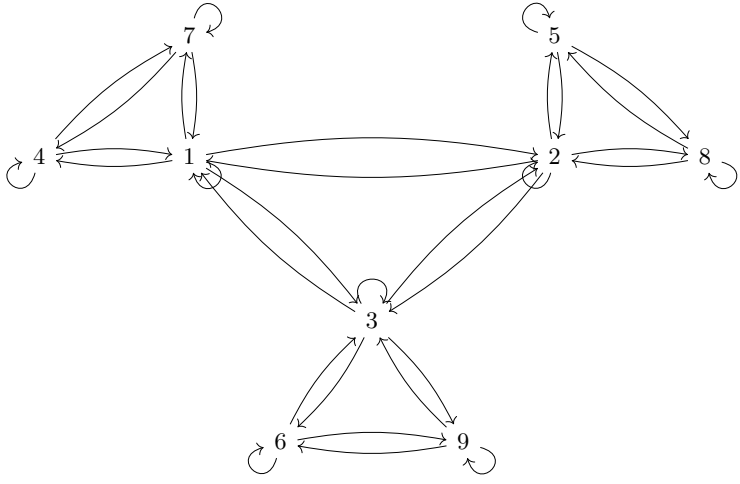
\begin{figure}[t]
			\[\begin{tikzcd}
				&& 6 &&&& 4 \\
				\\
				3 && 0 &&&& 1 && 7 \\
				\\
				\\
				&&&& 2 \\
				\\
				&&& 5 && 8
				\arrow[from=1-3, to=1-3, loop, in=15, out=75, distance=5mm]
				\arrow[bend left=10, from=1-3, to=3-1]
				\arrow[bend left=10, from=1-3, to=3-3]
				\arrow[from=1-7, to=1-7, loop, in=105, out=165, distance=5mm]
				\arrow[bend left=10, from=1-7, to=3-7]
				\arrow[bend left=10, from=1-7, to=3-9]
				\arrow[bend left=10, from=3-1, to=1-3]
				\arrow[from=3-1, to=3-1, loop, in=195, out=255, distance=5mm]
				\arrow[bend left=10, from=3-1, to=3-3]
				\arrow[bend left=10, from=3-3, to=1-3]
				\arrow[bend left=10, from=3-3, to=3-1]
				\arrow[from=3-3, to=3-3, loop, in=285, out=345, distance=5mm]
				\arrow[bend left=10, from=3-3, to=3-7]
				\arrow[bend left=10, from=3-3, to=6-5]
				\arrow[bend left=10, from=3-7, to=1-7]
				\arrow[bend left=10, from=3-7, to=3-3]
				\arrow[from=3-7, to=3-7, loop, in=195, out=255, distance=5mm]
				\arrow[bend left=10, from=3-7, to=3-9]
				\arrow[bend left=10, from=3-7, to=6-5]
				\arrow[bend left=10, from=3-9, to=1-7]
				\arrow[bend left=10, from=3-9, to=3-7]
				\arrow[from=3-9, to=3-9, loop, in=285, out=345, distance=5mm]
				\arrow[bend left=10, from=6-5, to=3-3]
				\arrow[bend left=10, from=6-5, to=3-7]
				\arrow[from=6-5, to=6-5, loop, in=60, out=120, distance=5mm]
				\arrow[bend left=10, from=6-5, to=8-4]
				\arrow[bend left=10, from=6-5, to=8-6]
				\arrow[bend left=10, from=8-4, to=6-5]
				\arrow[from=8-4, to=8-4, loop, in=195, out=255, distance=5mm]
				\arrow[bend left=10, from=8-4, to=8-6]
				\arrow[bend left=10, from=8-6, to=6-5]
				\arrow[bend left=10, from=8-6, to=8-4]
				\arrow[from=8-6, to=8-6, loop, in=285, out=345, distance=5mm]
			\end{tikzcd}\]		\caption{Pictorial representation of the Schurian groupoids $\Nn = \widehat{\{0,3,6 \}}\sqcup\widehat{\{1,4,7 \}}\sqcup \widehat{\{2,5,8 \}}$ and $\Hh = \widehat{3}$. }\label{fig:crossed9}
		\end{figure}
	\end{example}
	Two special circumstances occurred in the previous example: $\Gg$ was Schurian,\footnote{This allowed us to immediately conclude that $\Gg\cong \Kk\times \Hh$ by just looking at the sets of vertices. However, this hypothesis can probably be bypassed.} and $\Nn$ was a union of isomorphic connected components. This need not be true in general.
	\begin{example}
		Let $\Gg = \widehat{7}$ and $\Hh=\widehat{2}$. Even though there is a copy of $\Hh$ inside $\Gg$, and there are multiple ways to obtain a split epimorphism $f\colon \Gg\to \Hh$, there is no way to let all the components of $\ker(f)$ have the same number of vertices. Consider for instance $\Nn = \ker(f)$ and $\Hh$ as in \cref{fig:uneven}. Clearly, $\Gg \ncong \Kk\lfibre{}{} \Hh$ for all connected components $\Kk$ of $\Nn$.
		\begin{figure}[t]
			\[\begin{tikzcd}
				6 &&& 2 & 3 \\
				5 & 0 && 1 & 4
				\arrow[from=1-1, to=1-1, loop, in=60, out=120, distance=5mm]
				\arrow[bend left=10, from=1-1, to=2-1]
				\arrow[bend left=10, from=1-1, to=2-2]
				\arrow[from=1-4, to=1-4, loop, in=105, out=165, distance=5mm]
				\arrow[bend left=10, from=1-4, to=1-5]
				\arrow[bend left=10, from=1-4, to=2-4]
				\arrow[bend left=10, from=1-4, to=2-5]
				\arrow[bend left=10, from=1-5, to=1-4]
				\arrow[from=1-5, to=1-5, loop, in=15, out=75, distance=5mm]
				\arrow[bend left=10, from=1-5, to=2-4]
				\arrow[bend left=10, from=1-5, to=2-5]
				\arrow[bend left=10, from=2-1, to=1-1]
				\arrow[from=2-1, to=2-1, loop, in=195, out=255, distance=5mm]
				\arrow[bend left=10, from=2-1, to=2-2]
				\arrow[bend left=10, from=2-2, to=1-1]
				\arrow[bend left=10, from=2-2, to=2-1]
				\arrow[from=2-2, to=2-2, loop, in=240, out=300, distance=5mm]
				\arrow[bend left=10, from=2-2, to=2-4]
				\arrow[bend left=10, from=2-4, to=1-4]
				\arrow[bend left=10, from=2-4, to=1-5]
				\arrow[bend left=10, from=2-4, to=2-2]
				\arrow[from=2-4, to=2-4, loop, in=240, out=300, distance=5mm]
				\arrow[bend left=10, from=2-4, to=2-5]
				\arrow[bend left=10, from=2-5, to=1-4]
				\arrow[bend left=10, from=2-5, to=1-5]
				\arrow[bend left=10, from=2-5, to=2-4]
				\arrow[from=2-5, to=2-5, loop, in=285, out=345, distance=5mm]
			\end{tikzcd}\]
			\caption{The groupoids $\Hh = \widehat{2}$ and $\Nn= \widehat{\set{0,5,6}}\sqcup \widehat{\set{1,2,3,4}}$ inside $\Gg = \widehat{7}$.}\label{fig:uneven}
		\end{figure}
	\end{example}
	\section{Crossed products of groupoids and Split Lemma}\label{sec:crossed}
	In this section we give a notion of \textit{crossed products} in $\Gpd$, which is the right one to retrieve a (lifted) Split Lemma for groupoids---and hence the categorial semidirect product in the sense of Bourn and Janelidze \cite{BournJanelidzeProtomodularityDescSemProd}. In the category $\Gpd_\Lambda$ of groupoids over a fixed set of vertices, the notion will simplify, and it will become equivalent to a semidirect product in case one of the two groupoids is a group bundle. 
	
	The crucial idea is that quotients of groupoids are naturally bilateral: thus, in a splitting \textsc{fit} sequence $\Nn\to \Gg\to \Hh$ (where $\Hh$ is identified with a subgroupoid of $\Gg$), one generally has $\Gg = \Nn\Hh\Nn$ and not, as for groups, $\Gg = \Nn\Hh$. Thus if we want to retrieve a result akin to the Split Lemma, our 
	crossed products need to be based on the tensor product $\Nn\ot\Hh\ot\Nn$, rather than $\Nn\ot\Hh$. The underlying quiver is not exactly $\Nn\ot\Hh\ot\Nn$, though: we need to take a quotient, namely the \textit{balanced tensor product} $\Nn\otb\Hh\otb\Nn$. This is indispensable, if we want to comprise 
	semidirect products of groups as an instance of our notion. For $\Nn$ and $\Hh$ groups, one will naturally have $\Nn\otb \Hh\otb \Nn\cong \Nn\ot \Hh$ as quivers.
	\subsection{Crossed product in $\maybebm{\Gpd}$}\label{subsec:crossed_Gpd} Let $\Nn$ and $\Hh$ be groupoids, such that $\Hh^0\subseteq \Nn^0$, and every connected component of $\Nn$ contains exactly one vertex of $\Hh$. We shall discuss in \cref{subsec:Radford_Gpd} how these conditions arise naturally. For simplicity of notation, we denote by symbols such as $h,k,l$ the arrows in $\Hh$, by $a,b,c$ (or similar symbols) the arrows in $\Nn$, and by $\underline{a}, \underline{b}, \underline{c}$ (or similar symbols) the arrows in $\underline{\Nn} = \Nn^\circlearrowright(\Hh^0,\Hh^0)$. 
	
	Let $\triangleright$ be a left semistrong action of $\Hh$ on $\underline{\Nn}$. Define $\underline{a} \triangleleft h =h^{-1}\triangleright \underline{a}$, and observe that this is a right semistrong action of $\Hh$ on $\underline{\Nn}$, and that $\triangleright,\triangleleft$ satisfy the bimodule compatibility. Suppose moreover that $\underline{\Nn}$ is an $\Hh$-bimodule algebra, i.e.\@ that $h\triangleright (\underline{a}\, \underline{b}) = (h\triangleright \underline{a}) (h\triangleright \underline{b})$, which also implies $(\underline{a}\, \underline{b})\triangleleft h = (\underline{a}\triangleleft h)(\underline{b}\triangleleft h)$.
	
	With a slight abuse, since $\Hh^0$ is contained in $\Nn^0$, we make sense of the tensor products $\Nn\ot \Hh$ and $\Hh\ot \Nn$ in the obvious way. We define the \textit{balanced tensor product} $\Nn\otb \Hh\otb \Nn$ with respect to the $\Hh$-bimodule structure of $\underline{\Nn}$, as the quotient
	\[  (\Nn\ot\Hh\ot\Nn)/ (\sim,=), \]
	where $\sim$ is the equivalence relation generated by
	\[	 a\underline{b}\ot h\ot \underline{c} d\sim a\underline{b}(h\triangleright \underline{c} )\ot h \ot d\sim a \ot h \ot (\underline{b}\triangleleft h)\underline{c} d .\]
	Any two equivalent triples have same source and same target, thus $(\sim, =)$ is an equivalence pair.
	
	\begin{definition}
		The \textit{crossed product} $\Nn\rcross{}{\triangleright}\Hh\lcross{\triangleleft}{}\Nn$ is defined as the quiver $\Nn\otb \Hh\otb \Nn$, with multiplication
		\begin{align*} (a_1\otb h_1\otb b_1)(a_2\otb h_2\otb b_2) &= a_1 (h_1\triangleright b_1a_2)\otb h_1h_2\otb b_2\\
			&= a_1\otb h_1h_2 \otb (b_1a_2\triangleleft h_2)b_2, \end{align*}
		and units
		\[ 1_\lambda=  a\otb 1_\mu\otb a^{-1}  \]
		where $\mu$ is the unique vertex of $\Hh$ in the same connected component of $\Nn$ as $\lambda$, and $a$ is any arrow in $\Nn(\lambda,\mu)$. We simply write $\Nn\rcross{}{}\Hh\lcross{}{}\Nn$ when the actions are understood.
	\end{definition}
	Observe that $b_1a_2$ is in fact an arrow of $\underline{\Nn}$, because it is a loop on $\target(h_1)= \source(h_2)$. The verification that $ a_1 (h_1\triangleright b_1a_2)\otb h_1h_2\otb b_2$ and $a_1\otb h_1h_2 \otb (b_1a_2\triangleleft h_2)b_2$ are equal is an immediate computation:
	\begin{align*}
		a_1 (h_1\triangleright b_1a_2)\otb h_1h_2\otb b_2&  = a_1\otb h_1h_2 \otb ((h_1\triangleright b_1a_2)\triangleleft h_1h_2)b_2\\
		&= a_1\otb h_1h_2 \otb \big((h_1\triangleright b_1a_2\triangleleft h_1)\triangleleft h_2\big)b_2\\
		&= a_1\otb h_1h_2 \otb (b_1a_2\triangleleft h_2)b_2.
	\end{align*}
	\begin{proposition}\label{prop:crossed_is_gpd}
		The crossed product $\Nn\rcross{}{}\Hh\lcross{}{}\Nn$ is in fact a groupoid. The quiver $\Nn\otb \One\otb \Nn$ is a normal subgroupoid, isomorphic with $\Nn$.
	\end{proposition}
	\begin{proof}
		We first check that the multiplication is well-defined. One has
		\begin{align*}
			(a_1\underline{b}_1\otb h_1\otb \underline{c}_1d_1) (a_2\underline{b}_2\otb h_2\otb \underline{c}_2d_2) &= a_1\underline{b}_1(h_1\triangleright \underline{c}_1d_1a_2\underline{b}_2)\otb h_1h_2\otb \underline{c}_2d_2\\
			\overset{(\dagger)}&{=} a_1b_1(h_1\triangleright \underline{c}_1)(h_1\triangleright d_1a_2\underline{b}_2)\otb h_1h_2\otb \underline{c}_2d_2\\
			&= (a_1\underline{b}_1 (h_1\triangleright \underline{c}_1)\otb h_1\otb d_1) (a_2\underline{b}_2\otb h_2\otb \underline{c}_2d_2),
		\end{align*}
		where the step marked with $(\dagger)$ follows from the module algebra condition, and the fact that both $\underline{c}_1$ and $d_1a_2\underline{b}_2$ are loops over $\target(h_1)$. One also has
		\begin{align*}
			(a_1\underline{b}_1\otb h_1\otb \underline{c}_1d_1) (a_2\underline{b}_2\otb h_2\otb \underline{c}_2d_2) &= a_1\underline{b}_1(h_1\triangleright \underline{c}_1d_1a_2\underline{b}_2)\otb h_1h_2\otb \underline{c}_2d_2\\
			&= a_1\underline{b}_1 (h_1\triangleright \underline{c}_1d_1a_2\underline{b}_2)(h_1h_2\triangleright \underline{c}_2)\otb h_1h_2\otb d_2\\
			\overset{(\ddagger)}&{=} a_1\underline{b}_1\big( h_1\triangleright \underline{c}_1d_1a_2\underline{b}_2 (h_2\triangleright \underline{c}_2) \big)\otb h_1h_2\otb d_2\\
			&=  (a_1\underline{b}_1\otb h_1\otb \underline{c}_1d_1) (a_2\underline{b}_2 (h_2\triangleright \underline{c}_2)\otb h_2\otb d_2),
		\end{align*}
		where the step marked with $(\ddagger)$ follows again from the module algebra condition, together with the fact that both $\underline{c}_1d_1a_2\underline{b}_2$ and $h_2\triangleright \underline{c}_2$ are loops over $\target(h_1)$. The good definition of the multiplication with respect to the relations involving $\triangleleft$ are proven symmetrically, using the alternative form of the multiplication that involves $\triangleleft$.
		
		
		We now check that this multiplication provides a groupoid structure. As for the associativity, one has
		\begin{align*}
			&(a_1\otb h_1\otb b_1)\Big( (a_2\otb h_2\otb b_2)(a_3\otb h_3\otb b_3) \Big)\\
			&= (a_1\otb h_1\otb b_1) \big( a_2(h_2\triangleright b_2a_3)\otb h_2h_3\otb b_3 \big)\\
			&= a_1\Big(h_1\triangleright \big( b_1a_2(h_2\triangleright b_2a_3) \big)\Big)\otb h_1h_2h_3\otb b_3\\
			\overset{(\lozenge)}&{=} a_1(h_1\triangleright b_1a_2)(h_1h_2\triangleright b_2a_3)\otb h_1h_2h_3\otb b_3\\
			&= \Big(a_1(h_1\triangleright b_1a_2)\otb h_1h_2\otb b_2\Big) (a_3\otb h_3\otb b_3),
		\end{align*}
		where the step marked with $(\lozenge)$ follows from the module algebra condition, and from both $b_1a_2$ and $h_2\triangleright b_2a_3$ being loops on $\source(h_2)=\target(h_1)$.
		
		The units $a\otb 1\otb a^{-1}$ are well-defined: indeed, if $a$ and $b$ are two arrows in $\Nn$ connecting the vertex $\lambda\in\Nn^0$ with the unique $\mu\in \Hh^0$ that lies in the same connected component of $\Nn$ as $\lambda$, then $b^{-1}a$ and $a^{-1}b$ are loops on $\mu$, on which $1_\mu \in \Hh^1$ acts trivially; and hence
		\begin{equation}\label{eq:verif_does_not_dep_on_a}		a\otb 1\otb a^{-1} = bb^{-1}a\otb 1\otb a^{-1}bb^{-1} = b\otb 1\otb (b^{-1}a\triangleleft 1_\mu)a^{-1}bb^{-1} = b\otb 1\otb b^{-1},  \end{equation}
		as desired. One immediately verifies that $(a\otb 1\otb a^{-1}) (b\otb h \otb c) = b\otb h \otb c$ whenever $\source(b) = \source(a)$; and similarly on the other side. The inverse of $a\otb h\otb b$ is simply $b^{-1}\otb h^{-1}\otb a^{-1}$: indeed, the module algebra condition $h\triangleright (\underline{a}\, \underline{b}) = (h\triangleright \underline{a}) (h\triangleright \underline{b})$ implies $h\triangleright 1_{\target(h)} = 1_{\source(h)}$, and hence
		\[ (a\otb h\otb b)(a^{-1}\otb h^{-1}\otb b^{-1}) = a(h\triangleright bb^{-1})\otb hh^{-1}\otb a^{-1} = a\otb 1\otb a^{-1},  \]
		which is the unit on $\source(a\otb h\otb b)$; and similarly on the other side, using the description of the multiplication via $\triangleleft$.
		
		We finally observe that $\Nn\otb \One\otb \Nn$ is a normal subgroupoid: it is closed under multiplication, because
		\[
		(a\otb 1\otb b)(c\otb 1\otb d) = abc\otb 1\otb d, \]
		it is clearly closed under units and inverses, and the conjugation
		\[ (c\otb h \otb d)(a\otb 1\otb b) (c^{-1}\otb h^{-1} \otb d^{-1}) = c(h\triangleright da)(h\triangleright c^{-1})\otb 1\otb d^{-1} \]
		lies in $(\Nn\otb \One \otb \Nn)^1$ whenever it is well-defined. The map $\varphi\colon a\otb 1\otb b\mapsto ab$ is well-defined because $\target(a)=\source(b)$, and it is a strong morphism of groupoids $\Nn\otb \One\otb\Nn\to \Nn$. If $a$ is an arrow in $\Nn$, choose any arrow $b$ from $\target(a)$ to the unique vertex of $\Hh^0$ that lies in the same connected component as $a$, and define the strong morphism of groupoids $\psi\colon a\mapsto ab\otb 1_{\target(b)}\otb b^{-1}$: this does not depend on the choice of $b$ (the verification is essentially the same as \eqref{eq:verif_does_not_dep_on_a}), and it is clearly the inverse of $\varphi$.
	\end{proof}
	The form of the inverses, in this crossed product, looks much simpler than in the usual semidirect product of groups. However, in the case of groups, these two expressions are actually equivalent: we shall prove it in \cref{subsec:crossed_GpdLambda}.
	\subsection{Lifted Split Lemma in $\maybebm{\Gpd}$}\label{subsec:Radford_Gpd} Consider a split epimorphism of groupoids
	\[\begin{tikzcd}
		\Nn & \Gg & \Hh,
		\arrow["\iota", from=1-1, to=1-2]
		\arrow["\pi", shift left, from=1-2, to=1-3]
		\arrow["s", shift left, from=1-3, to=1-2]
	\end{tikzcd}\]
	with moreover $\Hh= \Gg/\Nn$. We identify $\Nn$ with $\im(\iota)$, and $\Hh$ with $\im(s)$.
	
	One has $\Gg = \Nn\Hh\Nn$. By definition of the quotient $\Gg/\Nn$, every $x\in \Gg$ can be written as $a h b$ for $a,b\in \Nn$, $h\in \Hh$. The triple $a\ot h\ot b$ is not unique, however it becomes unique modulo a suitable equivalence relation on the arrows.
	\begin{lemma}\label{lemma:NHNgenerates} The triple $a\ot h\ot b$, defined above, is unique modulo the equivalence relation $\sim$ generated by
		\[ a\underline{b}\ot h\ot \underline{c} d\sim a\underline{b}(h\triangleright \underline{c} )\ot h \ot d\sim a \ot h \ot (\underline{b}\triangleleft h)\underline{c} d ,\]
		where, as in \cref{subsec:crossed_Gpd}, we write $\underline{a}$ to indicate an element of $\underline{\Nn}= \Nn^\circlearrowright(\Hh^0,\Hh^0)$, and $\triangleright,\triangleleft$ are the left and right action by conjugation of $\Hh$ on $\underline{\Nn}$, respectively.\end{lemma}
	\begin{proof}
		Observe that equivalent triples have indeed the same product. 
		
		Conversely, let $a_1\ot h_1\ot b_1$ and $a_2\ot h_2\ot b_2$ have the same product in $\Gg$ (the reader may refer to \cref{fig:welldefined} throughout the proof). In particular, $\source(a_1) = \source(a_2)$ and $\target(b_1)=\target(b_2)$. Since every connected component of $\Nn$ contains exactly one vertex of $\Hh$, this implies that $h_1$ and $h_2$ have same source and same target. Observe that
		\[ a_2h_2b_2 = a_1a_1^{-1} a_2 h_2 b_2 b_1^{-1} b_1 = a_1h_1b_1, \]
		which, since $\Gg$ is a groupoid, implies $h_1 = a_1^{-1}a_2 h_2 b_2b_1^{-1}$, where $a_1^{-1}a_2$ and $b_2b_1^{-1}$ are loops in $\Nn$. But the projection $\pi\colon \Gg\to \Gg/\Nn$ is an isomorphism on the subgroupoid $\Hh$, thus the arrows $h_1$ and  $a_1^{-1}a_2 h_2 b_2b_1^{-1}$ can only have the same image if $h_1 = h_2$. If we substitute $h_1 = h_2$ in the equation $h_1 = a_1^{-1}a_2 h_2 b_2b_1^{-1}$, we get $a_1^{-1}a_2 = h_1 (b_1b_2^{-1}) h_1^{-1} = h_1\triangleright (b_1b_2^{-1})$, where the action is well-defined because $b_1b_2^{-1}$ is a loop. Thus
		\begin{align*}
			a_2 \ot h_2\ot b_2 &= a_1a_1^{-1}a_2\ot h_1\ot b_2\\
			&= a_1 (h_1\triangleright b_1b_2^{-1}) \ot h_1\ot b_2\\
			&\sim a_1\ot h_1\ot b_1b_2^{-1} b_2\\
			&= a_1\ot h_1\ot b_1,
		\end{align*}
		as desired.
		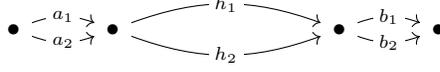
\begin{figure}[t]
			\[\begin{tikzcd}
				\bullet & \bullet &&& \bullet & \bullet
				\arrow["{a_1}"{description}, bend left=20, from=1-1, to=1-2]
				\arrow["{a_2}"{description}, bend right=20, from=1-1, to=1-2]
				\arrow[""{name=0, anchor=center, inner sep=0}, "{h_1}"{description}, bend left=20, from=1-2, to=1-5]
				\arrow[""{name=1, anchor=center, inner sep=0}, "{h_2}"{description}, bend right=20, from=1-2, to=1-5]
				\arrow["{b_1}"{description}, bend left=20, from=1-5, to=1-6]
				\arrow["{b_2}"{description}, bend right=20, from=1-5, to=1-6]
			\end{tikzcd}\]
			\caption{Two paths $a_1\ot h_1\ot b_1$ and $a_2\ot h_2\ot b_2$ in $\Gg$ with same product. This picture is for reference throughout the proof of \cref{lemma:NHNgenerates}, to help the reader check what compositions are allowed.}\label{fig:welldefined}
		\end{figure}
	\end{proof}
	As a consequence of \cref{lemma:NHNgenerates}, one gets a strong isomorphism of quivers \[\Gg\cong \Nn\otb \Hh\otb \Nn,\] where the strong morphism $\Gg\to \Nn\otb \Hh\otb \Nn$ sends $x\in \Gg$ to the unique class $a\otb h\otb b$ having $ahb= x$; and the strong morphism $\Nn\otb \Hh\otb \Nn$ is induced by the multiplication $a\otb h\otb b\mapsto ahb$, and it is well-defined because all equivalent triples $a\ot h\ot b$ have same product.
	\begin{theorem}\label{thm:Radford_Gpd}
		The strong isomorphism of quivers $\Gg\cong \Nn\otb \Hh \otb \Nn$ described above is a strong isomorphism of groupoids $\Gg\cong \Nn\rcross{}{\triangleright} \Hh\lcross{\triangleleft}{}\Nn$.
	\end{theorem}
	\begin{proof}
		Let $x\ot y\in(\Gg\ot \Gg)^1$, with $x= a_1h_1b_1$ and $y=a_2h_2b_2$ for $a_i, b_i\in \Nn^1$, $h_i\in \Hh^1$. Clearly, the product $xy$ can be written as 
		\[xy = a_1h_1b_1a_2h_2b_2
		\overset{(\dagger)}{=} a_1 (h_1 b_1a_2 h_1^{-1})h_1h_2 b_2 ,
		\]
		where the equality marked with $(\dagger)$ requires that $b_1a_2$ be a loop, but this is true, because both $\source(b_1)$ and $\target(a_2)$ lie in $\Hh^0$ and they both  lie in the same connected component of $\Nn$, thus they are the same vertex.
		
		Therefore, $xy$ admits the representation $a_1 (h_1\triangleright b_1a_2)\otb h_1h_2\otb b_2$ in $\Nn\otb \Hh\otb \Nn$, and this is exactly the product of $a_1\otb h_1\otb b_1$ and $a_2\otb h_2\otb b_2$ in $\Nn\rcross{}{\triangleright} \Hh\lcross{\triangleleft}{}\Nn$.
	\end{proof}
	\begin{corollary}\label{cor:lifted_Radford_Gpd}
		Let $\Nn\to\Gg\to\Hh$ be a split short exact sequence in $\Gpd$. Then it can be lifted to a universal splitting \textsc{fit} sequence $\tilde{\Nn}\to \tilde{\Gg}\to \Hh$ as in \cref{prop:universal_FIT_for_split}, where $\tilde{\Gg}$ is strongly isomorphic to the crossed product $\tilde{\Nn}\rcross{}{} \Hh\lcross{ }{}\tilde{\Nn}$ via the left and right actions by conjugation.
	\end{corollary}
	\begin{example} Consider again the coarse groupoid $\widehat{9}$ and the normal subgroupoid $\Nn = \widehat{\{0,3,6 \}}\sqcup\widehat{\{1,4,7 \}}\sqcup \widehat{\{2,5,8 \}}$. The quotient $\Gg/\Nn$ is a coarse groupoid on three vertices. As a section $\Gg/\Nn\to \Gg$, we may choose the map that picks, for each equivalence class of vertices, its smallest representative. The resulting immersion of $\Gg/\Nn$ into $\Gg$ is the groupoid $\Hh = \widehat{3}$; see again \cref{fig:crossed9}.
		
		Then, the groupoid $\Gg$ is recovered as the crossed product $\Nn\rcross{}{} \Hh\lcross{}{} \Nn$. Notice that the actions are all trivial, since the only loops in $\Nn$ are the units.
		
		Observe that $\Nn\otb \Hh\otb \Nn = \Nn\ot \Hh\ot \Nn$ holds in this case, since every equivalence class contains only one triple.
		
		From this example, we can also see how the arrows of $\Gg$ are explicitly recovered as elements of $\Nn\otb \Hh\otb \Nn$. For instance, the unit $[7,7]$ is $[7,1]\ot [1,1]\ot [1,7]$; and the arrow $[7,5]$ is $[7,1]\ot [1,2]\ot [2,5]$. 
	\end{example}
	\subsection{On the crossed product and Split Lemma in $\maybebm{\Gpd_\Lambda}$}\label{subsec:crossed_GpdLambda} We now consider the crossed product in case $\Nn$ is the kernel of a morphism in $\Gpd_\Lambda$, and hence $\Nn = \Nn^\circlearrowright$. \begin{remark}\label{rem:actionconn_Gpbundle} Recall that a \textit{group bundle} a bundle of groups that are all isomorphic with each other. Observe that the action of $h$ induces isomorphisms $h\triangleright\blank$   between the isotropy groups $\Nn_{\target(h)}$ and $\Nn_{\source(h)}$. Thus if $\Hh$ is connected, $\Nn$ is a group bundle. More generally, the isomorphism classes of the isotropy groups of $\Nn$ cannot be more than the number of connected components of $\Hh$.\end{remark}
	
	When $\Nn$ is a group bundle, our notion of crossed product will collapse to the \textit{semidirect product of a group by a groupoid} defined for instance in Brown \cite[\S11.4]{BrownBookGroupoids}. A semidirect product $\Nn\rtimes \Hh$ of groupoids is described by Ibort and Marmo \cite{IbortMarmoGroupoids}, who moreover provide a Split Lemma, but in their setting again $\Nn$ is forced to be a bundle of groups. We shall see how both these notions are generalised by our crossed product of groupoids, and Ibort and Marmo's Split Lemma is an instance of \cref{thm:Radford_Gpd}.
	
	\begin{remark}
		In the same setting as \cref{subsec:crossed_Gpd}, if $\Nn$ is a bundle of groups, the assumption that every connected component of $\Nn$ contains a vertex in $\Hh^0$ implies $\Nn^0 = \Hh^0 = \Lambda$, and $\underline{\Nn} = \Nn$. 
		
		As a consequence, there is a strong isomorphism $\Nn\otb \Hh\otb \Nn\cong \Nn\ot \Hh$ of quivers over $\Lambda$, given by
		\[ \varphi\colon \underline{a}\otb h\otb \underline{b} \mapsto \underline{a} (h\triangleright \underline{b})\ot h.\]
		This is easily seen to be well-defined. The inverse is the morphism 
		\[ \underline{a}\ot h\mapsto \underline{a}\otb h\otb 1.\]
		Alternatively, the fact that $\varphi$ is an isomorphism is immediate from the fact that 
		\[ a\otb h\otb b = a(h\triangleright b)\otb h\otb 1 \]
		for all $a\otb h\otb b\in (\Nn\otb \Hh\otb \Nn)^1$.
	\end{remark}
	Since every arrow in $\Nn^1$ lies in $\underline{\Nn}^1$, we henceforth suppress the underline from the notation.
	\begin{definition}
		For a bundle of groups $\Nn$ over $\Lambda$, a groupoid $\Hh$ over $\Lambda$, and a left semistrong action $\triangleright $ of $\Hh$ on $\Nn$, we define the crossed product $\Nn \rcross{}{\triangleright}\Hh$ as the quiver $\Nn\ot \Hh$, with multiplication
		\[ (a_1\ot h_1) (a_2\ot h_2) = a_1(h_1\triangleright a_2)\ot h_1h_2, \]
		units $1\ot 1$, and inverses $ (a\ot h)^{-1} = (h^{-1}\triangleright a^{-1})\ot h^{-1}$. We simply write $\Nn\rcross{}{}\Hh$ when the action is understood.
	\end{definition}
	We skip the verification that the above structure is a groupoid, since it is the same as for semidirect products of groups. We just observe that, in the definition, the product $a_1 (h_1\triangleright a_2)$ is well defined because the action $\triangleright$ is semistrong, which implies $\source(h_1\triangleright a_2) = \source(h_1) = \target(a_1)$. 
	\begin{lemma}\label{lem:two_sided_crossed_iso_one_sided}
		The morphism $\varphi$ is a strong isomorphism of groupoids
		\[ \Nn\rcross{}{\triangleright} \Hh\lcross{\triangleleft}{} \Nn \cong \Nn\rcross{}{\triangleright} \Hh. \]
	\end{lemma}
	\begin{proof}
		Since $\Nn^0 = \Hh^0$ and $\Nn  = \underline{\Nn}$, it is immediate to observe that $a\otb 1\otb a^{-1} = 1\otb 1\otb 1$, and hence $\varphi$ sends units into units. As for products, one has
		\begin{align*}
			\varphi(a_1\otb h_1\otb b_1)\varphi(a_2\otb h_2\otb b_2)& = \varphi(a_1\otb h_1\otb b_1)\varphi(a_2(h_2\triangleright b_2)\otb h_2\otb 1)\\
			&= \Big( a_1 (h_1\triangleright b_1)\ot h_1 \Big)\Big( a_2(h_2\triangleright b_2)\ot h_2\Big)\\
			&= a_1 (h_1\triangleright b_1) (h_1\triangleright (h_2\triangleright b_2))\ot h_1h_2\\
			&= a_1 (h_1\triangleright b_1) (h_1h_2\triangleright b_2)\ot h_1h_2\\
			&= \varphi \Big( a_1 (h_1\triangleright b_1a_2)\otb h_1h_2\otb b_2 \Big)\\
			&= \varphi\Big( (a_1\otb h_1\otb b_1) (a_2\otb h_2\otb b_2) \Big),
		\end{align*}
		as desired. 
	\end{proof}
	Observe that the suspiciously nice-looking inverse $(a\otb h\otb b)^{-1} = b^{-1}\otb h^{-1}\otb a^{-1}$ from the proof of \cref{prop:crossed_is_gpd}, in this new setting becomes
	\begin{align*} \big(a(h\triangleright b)\otb h\otb 1 \big)^{-1} = (a\otb h \otb b)^{-1} &= b^{-1}\otb h^{-1}\otb a^{-1}\\
		&= b^{-1}(h^{-1}\triangleright a^{-1})\otb h^{-1}\otb 1\\
		&= (h^{-1}h\triangleright b^{-1}) (h^{-1}\triangleright a^{-1})\otb h^{-1}\otb 1\\
		&= h^{-1}\triangleright \big( (h\triangleright b)^{-1}a^{-1} \big)\otb h^{-1}\otb 1\\
		&= h^{-1}\triangleright \big( a(h\triangleright b)  \big)^{-1}\otb h^{-1}\otb 1, \end{align*}
	which mirrors the expression of the inverses in $\Nn\rcross{}{} \Hh$. In some way, this explains why the inverse in the crossed product $\Nn\rcross{}{} \Hh$ look so complicated: it reflects a much nicer expression of the inverse, which lives in a `two-sided' crossed product $\Nn\rcross{}{}\Hh\lcross{}{}\Nn$.
	
	As a corollary, we retrieve the Split Lemma in $\Gpd_\Lambda$.
	\begin{corollary}\label{cor:Radford_GpdLambda}
		Let $\Nn\to \Gg\to \Hh$ be a splitting short exact sequence in $\Gpd_\Lambda$. One has $\Gg\cong \Nn\rcross{}{}\Hh$, where the action $\triangleright$ of $\Hh$ on $\Nn$ is induced by the conjugation in $\Gg$.
	\end{corollary}
	\begin{proof}
		Since $\Nn\to\Gg\to\Hh$ is a short exact sequence in $\Gpd_\Lambda$, the groupoid $\Nn$ is a bundle of groups over $\Lambda = \Hh^0 =\Nn^0$. The conclusion follows from \cref{thm:Radford_Gpd} together with \cref{lem:two_sided_crossed_iso_one_sided}.
	\end{proof}
	We now consider the crossed product $\Nn \rcross{}{}\Hh$ in $\Gpd_\Lambda$, in the case when $\Nn = N\rtimes \One_{\Lambda}$ is a group bundle. We would like to reinterpret this crossed product as a semidirect product $N\rtimes \Hh$ in the sense of \cref{sec:semidirect}.
	\begin{lemma}\label{lem:semidirect_is_crossed}
		In the hypotheses of \cref{lem:two_sided_crossed_iso_one_sided}, if $\Hh$ is connected, and hence $\Nn$ is a group bundle $\Nn \cong N\lfibre{}{} \One_{\Lambda}$, then there is a strong isomorphism of groupoids
		\[ N \lfibre{}{} \Hh  \cong \Nn\rcross{}{}\Hh,\]
		for some action $\btriangleright$ of $\Hh$ on $N$. This isomorphism is canonical if $\Hh$ is Schurian. 
	\end{lemma}
	\begin{proof} It is always possible to find a maximal Schurian subgroupoid $\Hh'$ of $\Hh$, which is a coarse groupoid isomorphic to $\widehat{\Lambda}$; see e.g.\@ \cite[Remark 5.6 and Lemma 5.7]{ferri2024dynamical}. We identify this subgroupoid $\Hh'$ with $\widehat{\Lambda}$, and use the notation $[\lambda,\mu]$ for the unique arrow $\lambda\to \mu$. 
		
		The map $[\lambda,\mu]\triangleright \blank\colon \Nn_\mu\to \Nn_\lambda$ is an isomorphism of groups. 
		Thus the maps $[\lambda,\mu]\triangleright \blank$ induce  isomorphisms between the isotropy groups of $\Nn$, being all isomorphic to a chosen isotropy group, say, $N = \Nn_\lambda$ for a vertex $\lambda\in\Lambda$. We denote by $\varphi_{\lambda,\mu}$ the isomorphism $N = \Nn_\lambda\to \Nn_\mu$.
		
		Let an action of $\Hh$ on $N$ be defined as
		\[ h\btriangleright a = \varphi_{\lambda,\source(h)}^{-1}(h\triangleright \varphi_{\lambda, \target(h)}(a) ) =\big([\lambda,\source(h)]\, h\, [\target(h),\lambda]\big)\triangleright a  \]
		(namely: $a$ is read as the representative of a loop in $\Nn_{\target(h)}$, so that $h$ can be let act as in the crossed product, and then the result is a loop in $\Nn_{\source(h)}$ that has a representative back in the group $N$). The verification that $\btriangleright$ is an action is left to the reader. We now define a strong morphism 
		\[ f\colon N\rtimes_{\btriangleright} \Hh \to \Nn\rcross{}{\triangleright}\Hh,\quad f(a\times h) = \varphi_{\lambda,\source(h)} (a)\ot h, \]
		whose inverse is clearly
		\[ n \ot h\mapsto  \varphi_{\lambda,\source(h)}^{-1}(n)\times h. \]
		We check that $f$ is a morphism of groupoids. Observe that $(a_1\times h_1) \cdot (a_2\times h_2)$ in $N\rtimes \Hh$ is defined if and only if $\target(a_1) = h_1\btriangleright \source(a_2)$ and $\target(h_1)=\source(h_2)$, if and only if $f(a_1\times h_1) \cdot f(a_2\times h_2)$ is defined in $\Nn\rcross{}{}\Hh$. One has
		\begin{align*}
			f(a_1\times h_1) \cdot f(a_2\times h_2) & = \Big(\varphi_{\lambda, \source(h_1)}(a_1)\ot h_1\Big)\cdot \Big(\varphi_{\lambda, \source(h_2)}(a_2)\ot h_2\Big)\\
			&= \varphi_{\lambda, \source(h_1)}(a_1) \cdot \Big( h_1\triangleright \varphi_{\lambda, \source(h_2)} (a_2) \Big)\ot h_1h_2\\
			&= \varphi_{\lambda, \source(h_1)}(a_1) \cdot \Big( h_1\triangleright \varphi_{\lambda, \target(h_1)} (a_2) \Big)\ot h_1h_2\\
			&= \Big(\varphi_{\lambda, \source(h_1)}(a_1) \cdot \varphi_{\lambda, \source(h_1)} (h_1\btriangleright a_2)\Big)\ot h_1h_2\\
			\overset{(\dagger)}&{=} \varphi_{\lambda, \source(h_1)}\Big( a_1 (h_1\btriangleright a_2)  \Big)\ot h_1h_2\\
			&= f((a_1\times h_1)\cdot (a_2\times h_2)),
		\end{align*}
		as desired, where the step marked with $(\dagger)$ follows from the fact that $\varphi_{\lambda, \source(h_1)}$ is a group homomorphism. Using that $\varphi_{\lambda, \source(h)}^{-1}$ is a group homomorphism, one similarly proves that $f^{-1}$ is also a groupoid morphism.
		
		Observe that the isomorphism $f$ depends on the choice of the maximal Schurian subgroupoid $\Hh'$. The choice $\Hh'= \Hh$ is forced if $\Hh$ is already Schurian, thus in this case $f$ is canonical.
	\end{proof}
	\begin{remark}\label{rem:prod_GP_times_Set_is_crossed}
		The product  $\Gg\cong G\times \widehat{\Lambda}$ of a group with a coarse groupoid that appears in \cref{prop:product_groupd_coarse_groupoid} is a semidirect product $G\lfibre{}{}\widehat{\Lambda}$, and hence it is isomorphic to the crossed product $\Gg^\circlearrowright\rcross{}{} \widehat{\Lambda}$, where $\Gg^\circlearrowright$ is a group bundle (because $\Gg$ is connected), and the action of the Schurian groupoid $\widehat{\Lambda}$ on $\Gg^\circlearrowright$ is given by the conjugation in $\Gg$.
		
		The advantage of this point of view is, now, that the decomposition is independent of the choice of a vertex $\lambda\in\Gg^0$. However, it still depends on the choice of a maximal Schurian subgroupoid $\widehat{\Lambda}$, and this element of indeterminacy cannot be overruled.
	\end{remark}
	We finally retrieve the Split Lemma in $\Gpd_\Lambda$ (Proposition \ref{prop:MetereMontoli}). 
	\begin{corollary}
		Let $\Nn\to \Gg\to \Hh$ be a splitting \textsc{fit} sequence in $\Gpd_\Lambda$, with $\Hh$ connected (and hence $\Nn$ a group bundle). Then $\Gg\cong \Nn_\lambda \lfibre{}{}\Hh$, where $\Nn_\lambda$ is any isotropy group of $\Nn$, and the action of $h\in\Hh^1$ on $\Nn_\lambda$ is defined as in \cref{lem:semidirect_is_crossed}. This isomorphism is canonical if moreover $\Hh$ is a coarse groupoid.
	\end{corollary}
	\begin{proof}
		It suffices to merge \cref{cor:Radford_GpdLambda} with \cref{lem:semidirect_is_crossed}.
	\end{proof}
	
	\subsection*{Acknowledgements} This work benefited from conversations with Alessandro Ardizzoni, Alan Cigoli, Ilaria Colazzo, Kenny De Commer, Marino Gran, Geoffrey Janssens, Isabel Martin-Lyons, Silvia Properzi, Paolo Saracco, and Leandro Vendramin. The author is especially grateful to Alessandro Ardizzoni for his many remarks and his feedback on this work, to Alan Cigoli and Marino Gran for pointing out the reference \cite{MetereMontoliSemidirectIntGpd}, and to Leandro Vendramin for his constant interest and his advice.
	
	The author was funded by the Università di Torino though a PNRR DM 118 scholarship; by the Vrije Universiteit Brussel through the bench fee OZR3762; by Leandro Vendramin through the FWO Senior Research Project G004124N; and by the European Union -- NextGe\-ne\-ra\-tionEU under NRRP, Mission 4 Component 2 CUP D53D23005960006 -- Call PRIN 2022
	No.\@ 104 of February 2, 2022 of the Italian Ministry of University and Research; Project
	2022S97PMY \textit{Structures for Quivers, Algebras and Representations} (SQUARE).
	
	\begin{CJK}{UTF8}{gbsn}
		\bibliographystyle{acm}
		\bibliography{../refs}
	\end{CJK}
\end{document}